%% file: main.tex
\pdfoutput=1 
\documentclass[reqno,11pt]{amsart}

\usepackage
[sorting=nty, 
style=alphabetic,
maxbibnames=1000, 
backend=biber
]
{biblatex}

\usepackage{pifont} 

\usepackage[margin=1in,letterpaper,portrait]{geometry}

\usepackage[utf8]{inputenc}

\usepackage[colorlinks=true, pdfstartview=FitV, linkcolor=purple, citecolor=blue,
filecolor=violet, urlcolor=violet]{hyperref}

\usepackage[dvipsnames]{xcolor}
\usepackage{
mathtools, 
mathdots, 
microtype, 
amssymb 
}
\usepackage[shortlabels]{enumitem} 

\usepackage{tikz}
\usetikzlibrary{arrows.meta} 
\usetikzlibrary{shapes.geometric} 
\usetikzlibrary{positioning} 
\usepackage{tikz-cd} 
\usetikzlibrary{calc,decorations.pathmorphing,shapes}
\usetikzlibrary{matrix,arrows,backgrounds,shapes.misc,patterns, 
}

\newcommand{\TiledSurface}{(S,M,M^*,P)}

\newcommand{\Algebra}{A}

\usepackage{epsfig}  
\input{xy}
\xyoption{poly}
\xyoption{2cell}
\xyoption{all}

\newcommand{\calc}{\mathcal{C}}

\newcommand{\calt}{\mathcal{T}}

\newcommand{\za}{\alpha}
\newcommand{\zb}{\beta}
\newcommand{\zd}{\delta}
\newcommand{\zD}{\Delta}
\newcommand{\ze}{\epsilon}
\newcommand{\zg}{\gamma}
\newcommand{\zG}{\Gamma}
\newcommand{\zl}{\lambda}

\newcommand{\zs}{\sigma}

\newcommand{\zL}{\Lambda}

\newcommand{\kb}{\Bbbk}
\newcommand{\ot}{\leftarrow}

\newcommand{\Hom}{\textup{Hom}}
\newcommand{\add}{\textup{add}}
\newcommand{\rad}{\textup{rad}\,}

\newcommand{\Pbar}{\overline{P}}
\newcommand{\Sbar}{\overline{S}}
\newcommand{\Lbar}{\overline{L}}

\newcommand{\Ext}{\textup{Ext}}
\newcommand{\im}{\textup{im}}
\newcommand{\mar}[1]{\mathsf{mar}(#1)}

\renewcommand{\dim}{\mathsf{dim}}
\newcommand{\proj}{\mathsf{proj}}
\newcommand{\inj}{\mathsf{inj}}

\theoremstyle{definition} 
\newtheorem{theorem}{Theorem}[section]
\newtheorem{lemma}[theorem]{Lemma}

\newtheorem{proposition}[theorem]{Proposition}
\newtheorem{corollary}[theorem]{Corollary}

\newtheorem{definition}[theorem]{Definition}
\newtheorem{example}[theorem]{Example}
\newtheorem{remark}[theorem]{Remark}
\newtheorem{question}[theorem]{Question}
\numberwithin{equation}{section}
\counterwithout{figure}{section}

\newcommand{\stringlength}
{\ell}
\newcommand{\LEFT}{\text{pre}}
\newcommand{\RIGHT}{\text{suff}}
\newcommand{\overlapLEFT}{\text{pre}}
\newcommand{\overlapRIGHT}{\text{suff}}

\newcommand{\basis}{\mathcal{B}}

\newcommand{\hook}[1]{\textnormal{\textsf{hook}}(#1)} 

\newcommand{\cohook}[1]{\textnormal{\textsf{cohook}}(#1)} 

\newcommand{\degree}[1]{\textnormal{\textsf{degree}}(#1)} 

\newcommand{\indegree}[1]{\textnormal{\textsf{indegree}}(#1)} 

\newcommand{\outdegree}[1]{\textnormal{\textsf{outdegree}}(#1)} 

\newcommand{\required}[1]{\textnormal{\textsf{Required}}(#1)}

\definecolor{forGreen}{RGB}{0,120,0}

\newcommand{\End}{\textup{End}}

\newcommand{\Qbar}{\overline{Q}}
\newcommand{\Ibar}{\overline{I}}
\newcommand{\Abar}{\overline{A}}

\newcommand{\Mbar}{\overline{M}}

\newcommand{\fbar}{\overline{f}}

\title{Maximal almost rigid modules over gentle algebras}
\author{Emily Barnard}
\address{Department of Mathematical Sciences, DePaul University, Chicago, IL 60614-3210, USA}\email{
e.barnard@depaul.edu}

\author{Raquel Coelho Sim\~{o}es}\address{School of Mathematical Sciences, University of Lancaster, Lancaster
LA1 4YF, United Kingdom}
\email{r.coelhosimoes@lancaster.ac.uk}

\author{Emily Gunawan}\address{Department of Mathematics and Statistics, University of Massachusetts Lowell, Lowell, MA 01854-2874, USA}
\email{emily\_gunawan@uml.edu}

\author{Ralf Schiffler}
\address{Department of Mathematics, University of Connecticut, Storrs, CT 06269-1009, USA}
\email{schiffler@math.uconn.edu}

\date{\today.   
}


\begin{document}

\begin{abstract}
We study maximal almost rigid modules over a gentle algebra $A$.  We prove that the number of indecomposable direct summands of every maximal almost rigid $A$-module is equal to the sum of the number of vertices and the number of arrows of the Gabriel quiver of $A$. Moreover, the algebra $A$, considered as an $A$-module, can be completed  to a maximal almost rigid module in a unique way.

Gentle algebras are precisely the tiling algebras of surfaces with marked points.  We show that the (permissible) triangulations of the surface of $A$ are in bijection with the maximal almost rigid $A$-modules. 

Furthermore, we study the endomorphism algebra  $C=\textup{End}_A T$ of a maximal almost rigid module $T$. We construct a fully faithful functor $G\colon \textup{mod}\,A\to \textup{mod}\, \overline{A}$ into the module category of a bigger gentle algebra $\overline{A}$ and show that $G$ maps maximal almost rigid $A$-modules to tilting $\overline{A}$-modules. In particular, $C$ and $\overline{A}$ are derived equivalent and $C$ is gentle.

After giving a geometric realization of the functor $G$, we obtain a tiling $G(\mathcal{T})$ of the surface of $\overline{A}$ as the image of the triangulation $\mathcal{T}$ corresponding to $T$. We then show that the tiling algebra of $G(\mathcal{T})$ is $C$. Moreover, the tiling algebra of $\mathcal{T}$ is obtained algebraically from $C$ as the tensor algebra with respect to the $C$-bimodule $\textup{Ext}_C^2(DC,C)$, which also is fundamental in cluster-tilting theory.
\end{abstract}

\maketitle

\setcounter{tocdepth}{1}

\tableofcontents

\section{Introduction}
\subsection{Maximal almost rigid modules} 

The notion of a maximal almost rigid module is a novel concept in representation theory that was introduced in 2023 by three of the four authors in collaboration with Emily Meehan in \cite{BGMS19}. 
For a module $T$ over a finite dimensional (non-commutative) algebra $A$, the classical notion of rigidity is defined by the condition that every short exact sequence of the form $0\to T\to E\to T\to 0$ splits. We weaken this condition and say that $T$ is \emph{almost rigid} if for all indecomposable summands $T_i,T_j$ of $T$ and all short exact sequences $0\to T_i\to E\to T_j\to 0$, either the sequence splits, or the module $E$ is indecomposable.

The original motivation in \cite{BGMS19} was to give an algebraic interpretation of Catalan combinatorics, in particular,  Cambrian lattices of type $\mathbb{A}$  and triangulations of polygons \cite{ReadingCambrianLattices06}. 
Given a path algebra $A$ of Dynkin type $\mathbb{A}_n$, the authors show that there is a bijection between the maximal almost rigid $A$-modules and triangulations of a certain $(n+1)$-gon $S$.

This bijection is realized by establishing an equivalence from the category of oriented diagonals in $S$ to the category $\textup{mod}\,A$. Furthermore, it is shown that the endomorphism algebra of a maximal almost rigid module in this case is isomorphic to a tilted algebra of type $\mathbb{A}_{2n-1}$, hence providing a connection to the cluster algebra of the same type as in \cite{ABS}.

In the current paper, we generalize these results to all gentle algebras using the geometric model of \cite{BCS21}.
 
The idea of allowing non-split short exact sequences with indecomposable middle term is a new perspective on representation theory that may seem strange from the point of view of classical homological algebra. Nevertheless, since their introduction in \cite{BGMS19}, maximal almost rigid modules did naturally arise in a number of other research projects in representation theory. 
In \cite{ACFGS23}, they appear in the study of cluster structures for the $\mathbb{A}_\infty$ singularity, where they correspond to
triangulations of the completed infinity-gon. In this context, mutation is defined and the exchange graph of maximal almost rigid subcategories is proven to be connected.
It is shown in \cite{Deq} that, for type $\mathbb{A}$, the two concepts of maximal almost rigidity and Jordan recoverability are complementary to each other, and in \cite{HansonYou}, maximal almost rigid modules appear in the study of bricks over preprojective algebras of type $\mathbb{A}$. A generalization of maximal almost rigid modules to Dynkin type $\mathbb{D}$ was introduced in \cite{ ChenZheng24}.

A forthcoming paper \cite{BHRS} introduces a different exact structure on the module category $\textup{mod}\,A$  of  any type $\mathbb{A}$  path algebra $A$ and proves that the maximal almost rigid $A$-modules are exactly the tilting modules relative to this exact structure. In this case, the category $\textup{mod}\,A$ is an example of a $0$-Auslander category, which were recently introduced in \cite{GorskyNakaokaPalu} in the setting of extriangulated categories, see also the survey \cite{Palu}.

\subsection{Gentle algebras} \emph{Gentle algebras} form an important and well-studied class of finite dimensional algebras. They were introduced by Assem and Skowro\'nski in \cite{AS87} as a special case of the string algebras defined by Butler and Ringel in \cite{BR87}. The class of all gentle algebras is closed under derived equivalences \cite{Schroer99, SZ03}, and important examples of gentle algebras are the tilted algebras of Dynkin type $\mathbb{A}$ and extended Dynkin type $\widetilde{\mathbb{A}}$, see \cite{AS87}.  More recently, gentle algebras became fundamental in the study of Fukaya categories, mirror symmetry and dimer models, see \cite{LP20,HKK17,Bocklandt}. The singularity category of a gentle algebra is studied in \cite{Kalck}.

The indecomposable modules over a gentle algebra, in fact over any string algebra, have the combinatorial structure of a string module or a band module \cite{BR87}. 
From the point of view of maximal almost rigid modules, the band modules don't play any role and will be mostly ignored throughout this paper. The combinatorics of string modules is a powerful tool that also yields a good understanding of morphisms and extensions. 
Already in \cite{BR87}, the authors give a combinatorial way to describe the \emph{irreducible} morphisms by adding hooks to or removing cohooks from the string of the module. A basis for the morphism spaces between two string modules was given in \cite{CB89}. Recently, the extensions  between string modules are described explicitly in \cite{CanakciSchroll21,CPS21}. 

\subsection{Geometric models for gentle algebras} An important feature of gentle algebras is the existence of a geometric models. A first example of this is the 2-Calabi-Yau tilted algebras arising from the cluster algebras of unpunctured surfaces of \cite{FST08} which were studied by Assem, Br\"ustle, Charbonneau-Jodoin and Plamondon \cite{ABCP10} (also \cite{CCS} for the polygon). 
This led to a geometric model for the cluster categories of these algebras in terms of arcs on the surface in \cite{CCS, BZ11}. 
Related constructions were given in \cite{DavidRoeslerAndSchiffler12,AmiotGrimeland16,AmiotTorusWithOneBoundaryComponent16}. 

In 2018, the following two major contributions realized geometric models for \emph{all} gentle algebras.
\begin{enumerate}
\item A geometric model for the derived category of a gentle algebra is given by Opper, Plamondon and Schroll in \cite{OPS} building on earlier work by Schroll in \cite{Schroll}.
\item A geometric model for the module category of a gentle algebra is given by Baur and the second author in \cite{BCS21}. 
\end{enumerate}
The two models are similar, but not the same. Since we work in the module category, we use (a modification) of the second model. We give here a quick overview and postpone the precise definitions to Section~\ref{sec:geometric model}. We point out that this model was very recently generalized to string algebras in \cite{BCS24}. 

A \emph{tiled surface} is a quadruple $\TiledSurface$, where $S $ is an oriented surface, $M$ and $M^*$ are two finite sets of marked points with $M$ a subset of the boundary of $S$, and $P$ is a tiling (or partial triangulation) of the marked surface $(S,M)$ such that every tile contains exactly one point of $M^*$. The authors of  \cite{BCS21} associate a gentle algebra, called a \emph{tiling algebra}, to every tiled surface and show that every gentle algebra arises this way. Moreover, the indecomposable string modules over the tiling algebra are in bijection with (permissible) arcs in the surface $(S,M^*)$, and the morphisms between string modules are given by (permissible) angles at intersections of these arcs.

Thus, there is a set of arcs on a surface that is in bijection with the set of indecomposable string modules over the tiling algebra. It is very natural to ask what type of module would correspond to a triangulation of the surface. 
As we will see in Theorem~\ref{thm intro 2}, these  are precisely the maximal almost rigid modules. 

\subsection{Main results}
Let $A$ be a gentle algebra with tiled surface $\TiledSurface$. 
Our first main result  characterizes two very particular maximal almost rigid modules. 
\begin{theorem}\label{thm intro 1}
(a) There is a unique maximal almost rigid $A$-module that contains all indecomposable projective modules as direct summands.

(b) There is a unique maximal almost rigid $A$-module that contains all indecomposable injectives as direct summands.
\end{theorem}
The proof only uses algebraic methods. We explicitly describe the indecomposable summands of these modules.

Our next result answers the question above. 
\begin{theorem}\label{thm intro 2}
 The bijection from string modules to permissible arcs induces a bijection from maximal almost rigid $A$-modules to permissible triangulations of the tiled surface.
\end{theorem}
In particular, this proves the existence of permissible triangulations. The combination of  the two theorems above yields the following result. Let $Q$ be the quiver of the algebra $A$.

\begin{corollary}\label{cor intro 3}
The number of indecomposable direct summands in a maximal almost rigid $A$-module is $|Q_0| + |Q_1|$.
\end{corollary}
This result can be thought of an analogy to  the number of summands of a \emph{tilting module} being equal to $|Q_0|$.

\medskip
Now let $T$ be a maximal almost rigid $A$-module with associated triangulation $\calt$. Then it is natural to consider the following two algebras. 
\begin{itemize}
\item The endomorphism algebra $C=\End_A T$.
\item The tiling algebra $B$ of the tiled surface given by the triangulation $\calt$. 
\end{itemize}

To study the endomorphism algebra $C$, we introduce a bigger gentle algebra $\Abar$ by replacing every arrow $\za\colon i\to j$ in the quiver of $A$ by a path $\xymatrix{i\ar[r]^{\za_a}&v_\za\ar[r]^{\za_b}&j}$, and every relation $\za\zb$ by the relation $\za_b\zb_a$. We show that  $\Abar$ is a gentle algebra of global dimension at most 2 and construct a functor $G\colon \textup{mod}\,A \to \textup{mod}\,\Abar$. 

\begin{theorem}
 \label{thm intro 3} With the notation above, we have the following.
 \begin{enumerate}[\rm(a)]
 \item   The functor $G$ is fully faithful.
\item   The module $G(T)$ is a tilting $\Abar$-module.
 \item   There is an isomorphism of algebras $C=\End_A \,T\cong \End_{\Abar}\, G(T)$.
 \item  The algebras  $\Abar$ and $C$  are derived equivalent.
 \item   The global dimension of $C$ is at most 2.
\end{enumerate}
\end{theorem}

We also construct the tiled surface of the algebra $\Abar$ as a refinement of the tiled surface of the algebra $A$, and we give a geometric realization of the functor $G$. This construction is essential for the understanding of the relationship between the two algebras $B$ and $C$. In fact, the tiled surface of $C$ is obtained from the tiled surface of $B$ by cutting 
one angle in every interior triangle of $\calt$. Going in the other direction, we have the  result below. Let $C=\oplus P(i)$ denote the direct sum of all indecomposable projective $C$-modules and $DC=\oplus I(i)$ the direct sum of the indecomposable injectives. Then $\Ext^2_C(DC,C)$ is a $C$-bimodule, and we denote the tensor algebra by $T_C(\Ext^2_C(DC,C))$.

\begin{theorem}
\label{thm intro 4}
Let $T$ be a maximal almost rigid $A$-module, $\calt $ the associated triangulation, $C=\End_A T$ and $B$ the tiling algebra of $\calt$. Then \[B\cong T_C(\Ext^2_C(DC,C)).\]
\end{theorem}
 Since $C$ is of global dimension at most 2, it follows that the algebra $B$ is the endomorphism algebra of a cluster-tilting object in Amiot's generalized cluster category $\calc_C$ of $C$, see \cite{Amiot09}. 
 
\begin{corollary}
 The algebra $B$ is 2-Calabi-Yau tilted and
 \[ B\cong \End_{\calc_C} C.\]
\end{corollary}

\begin{remark}
  $B$ is infinite dimensional if the global dimension of $A$ is infinite. 
\end{remark}

The result in the corollary is reminiscent of the main result in \cite{ABS} which realizes a relation between tilted algebras $C$ and cluster-tilted algebras $B= T_C(\Ext^2_C(DC,C))$. Recall that a tilted algebra is the endomorphism algebra of a tilting module over a \emph{hereditary} algebra. In that situation, there is an inverse construction that associates several tilted algebras $C$ to the algebra $B$ by taking the quotient by the annihilator of a local slice (a special kind of $B$-module that comes from tilting theory) \cite{ABS2}.
When the algebra $C$ is not a tilted algebra, there is in general no such inverse map; in fact, there is no natural candidate for $C$. 
 In our situation, the algebra $C$ is not a tilted algebra (unless the global dimension of $A$ is at most 1)
 but it is a natural choice that realizes the 2-Calabi-Yau algebra $B$ as a tensor algebra.

\subsection{Future directions}

The paper \cite{BGMS19} was inspired by Catalan combinatorics and it is shown 
in sections 8.1-8.2 there 
that the mutation of maximal almost rigid modules produces an  oriented exchange graph that is isomorphic to the Cambrian lattice. 
In a forthcoming paper, we will study mutations of maximal almost rigid modules and the resulting oriented exchange graph.

    Recall that the cardinality of $Q_0$ is equal to the number of summands of a (basic) tilting $A$-module. By Corollary~\ref{cor intro 3}, the cardinality of $Q_0\cup Q_1$ is the number of summands of a maximal almost rigid module. If we let $S=\oplus_{i\in Q_0} S(i)$ be the sum of all simple modules, we have 
    $|Q_0|=\dim\,\Hom(S,S)$, and
    $|Q_1|=\dim\,\Ext^1(S,S)$. So it is natural to ask if there are other classes of modules, which one might call maximal $m$-almost rigid modules, that capture the  dimension of $\cup_{i=0}^m\Ext^i(S,S) $.

\subsection{Structure of the article}
After recalling results on gentle algebras in section~\ref{sect 2}, we define and study their maximal almost rigid modules and prove Theorem~\ref{thm intro 1} in section~\ref{sec:MAR-algebra}. We also characterize the indecomposable modules that are direct summands of every maximal almost rigid module, which we call \emph{required summands}. In section~\ref{sec:geometric model}, we recall the geometric model of \cite{BCS21} and give a geometric realization of the required summands. Theorem~\ref{thm intro 2} is the main result of section~\ref{sect 5}. Section~\ref{sect 6} is devoted to the study of the algebras $C$ and $B$.  Theorem~\ref{thm intro 3} is the combination of Theorem~\ref{thm functor G}, Theorem~\ref{thm:tilting} and Corollary~\ref{cor:end-alg-tilt}, and Theorem~\ref{thm intro 4} is proved in Theorem~\ref{thm:tensor}. 
Section~\ref{sect examples} contains two detailed examples. 

\section{Background on gentle algebras}\label{sect 2}

In this section we will recall the definition of gentle algebras and the properties of their module categories necessary for this article. 

Let $\kb$ be an algebraically closed field and $Q$ be a connected quiver with set of vertices $Q_0$ and set of arrows $Q_1$. Given an arrow $\za$ in $Q_1$, we denote by $s(\za)$ the source of $\za$ and by $t(\za)$ the target of $\za$. An arrow $\za$ is called a \emph{loop} if $s(\za)=t(\za)$. 
A \emph{path} in $Q$ is a sequence $\za_1\za_2\ldots \za_\stringlength$ where each $\za_i\in Q_1$ and $t(\za_i)=s(\za_{i+1})$ for $i=1,\ldots, \stringlength-1$.

The \emph{outdegree} of a vertex $i \in Q_0$ is the number of arrows 
$\za$ in $Q_1$ with $s(\za)=i$, and the 
 \emph{indegree} of $i$ is the number of arrows $\za$ in $Q_1$ with $t(\za)=i$. 
 They are denoted by $\outdegree{i}$ and $\indegree{i}$, respectively. 
 The \emph{degree} of $i$ is the sum $\outdegree{i} + \indegree{i}$ and is denoted by $\degree{i}$.
 Note that if $i$ is incident to a loop, this loop is counted in both $\outdegree{i}$ and $\indegree{i}$.

Let $I$ be an admissible ideal of $\kb Q$ and let $A$ be the bound path algebra $\kb Q/I$. We denote by $\textup{mod}\,{A}$ the category of finitely generated right $A$-modules.

\begin{definition}[Gentle algebra~\cite{AS87}]
\label{def gentle} A finite dimensional algebra $A={\kb}Q/I$ is called \emph{gentle}
if 
the following conditions hold:
\begin{enumerate}
[label=(G\arabic*)]
\item\label{def gentle:itm1:degrees} If $i \in Q_0$, then $\indegree{i} \leq 2$ and $\outdegree{i} \leq 2$. 
\item\label{def gentle:itm2:string algebra} For each arrow $\za$ in $Q$, 
there is at most one arrow $\zb$ in $Q$ such that $\za\zb \notin I$
and 
there is at most one arrow $\zg$ in $Q$ such that $\zg\za \notin I$.
\item\label{def gentle:itm3:gentle algebra} For each arrow $\za$ in $Q$, 
there is at most one arrow $\zb'$ in $Q$ such that $\za\zb'\in I$
and 
there is at most one arrow $\zg'$ in $Q$ such that $\zg'\za \in I$.
\item\label{def gentle:itm4:ideal} $I$ is generated by paths of length $2$.
\end{enumerate}
\end{definition}

Let $A={\kb}Q/I$ be a gentle algebra.
For each $\za\in Q_1$ we associate a formal \emph{inverse arrow} $\za^{-1}$ with $s(\za^{-1})=t(\za)$ and $t(\za^{-1})=s(\za)$.
We write $Q_1^{-1}$ for the set of all formal inverse arrows $\za^{-1}$. 

A \emph{walk} in $Q$ is a sequence $w=\za_1 \za_2\ldots \za_\stringlength$, where $\stringlength \geq 1$ is called the \emph{length of $w$}, and each $\za_i\in Q_1\cup Q_1^{-1}$ such that $t(\za_i)=s(\za_{i+1})$, for $i=1,\ldots,\stringlength-1.$ 
Let $s(w)$ denote $s(\za_1)$ and $t(w)$ denote $t(\za_\stringlength)$. 
The \emph{inverse of a walk} $w=\za_1 \ldots \za_\stringlength$, denoted by $ w^{-1}$, is defined to be the walk $\za_\stringlength^{-1} \ldots \za_1^{-1}$.

A \emph{subwalk} of $w$ is a contiguous subsequence of $w$, that is, a subsequence of the form $\za_i \za_{i+1} \dots \za_j$ for some $1 \leq i \leq j \leq \stringlength$. 
A walk is called \emph{reduced} if it does not have subwalks of the form $\za\za^{-1}$ or $\za^{-1}\za$, with $\za\in Q_1$.

A \emph{(non-trivial) string of $A$} is a reduced walk in $Q$ that avoids relations, that is, there are no subwalks of the form $\za\zb$ such that $\za\zb\in I$ or $(\za\zb)^{-1}\in I$. 
For each vertex $v\in Q_0$, we associate a \emph{trivial string} $1_v$ with length $\stringlength=0$. The vertex $v$ is both the source and target of $1_v$. For technical reasons, we formally define the inverse of a trivial string $1_v$ which we denote by $1_v^{-1}$. 
We will also consider the empty string, which we will call \emph{zero string}, denoted by $0$.

A string $w=\za_1 \dots \za_\stringlength$ is \emph{direct} (\emph{inverse}, respectively) if 
$\za_i \in Q_1$ ($\za_i \in Q_1^{-1}$, respectively) for all $i=1,\dots,\stringlength$. 
A trivial string is considered to be both direct and inverse.
A direct string $w$ is \emph{right maximal} (\emph{left maximal}, respectively), if there is no arrow $\beta \in Q_1$ such that $w \beta$ ($\beta w $, respectively) is a string.  
A direct string is called \emph{maximal} if it is both right maximal and left maximal.

If a string $w$ is not direct nor inverse, then $w$ has a subwalk of the form $a^{-1} b$, called a \emph{hill}, or of the form $a b^{-1}$, called a \emph{valley}, where $a, b \in Q_1$ and $a \neq b$.

\begin{definition}[String modules~\cite{BR87}] \label{def:string modules}
For every string $w = \za_1 \ldots \za_\stringlength$ we define an indecomposable representation
$\displaystyle M(w)=(M(w)_x,M(w)_\za)_{x\in Q_0, \za\in Q_1}$ of $Q$  as follows. Let $v_1=s(\za_1)$ and $v_i=t(\za_{i-1})$, for $i=2,\ldots,\stringlength+1.$ Given a vertex $x$ of $Q_0$, let $I_x =\{i\mid v_i=x\}\subset\{1,2,\ldots,\stringlength+1\}$, and define $M(w)_x$ as the vector space with ordered basis $B_x=\{b_{x,i}\mid i\in I_x\}$. If $\za_i$ is an arrow, define $\za_i(b_{x,i})=b_{x,i+1}$, and if $\za_i$ is an inverse arrow, define $\za_i(b_{x,i+1})=b_{x,i}$. 
This defines an action of every arrow $\za$ of $Q$ for which $\za$ or $\za^{-1}$ appears in $w$ on the basis of $M(w)$. For arrows $\zb$ for which neither $\zb$ nor $\zb^{-1}$ appears in $w$, the action on $M(w)$ is defined to be zero. 

If $w=0$, then $M(w) = 0$ and if $w=1_v$, then $M(w)=S(v)$ is the simple module at $v$. 

From the definition of strings, it follows that $M(w)$ satisfies the relations in the ideal $I$ and hence $M(w)$ is a $\kb Q/I$-module. Moreover $M(w) \simeq M(w')$ if and only if $w'=w$ or $w'=w^{-1}$. The modules of the form $M(w)$ are called \emph{string modules}.

Note that, for every arrow $\za\colon x\to y$ of $Q$, the corresponding linear map $M(w)_\za\colon M(w)_x\to M(w)_y$ can be represented by a matrix relative to the bases $B_x $ and $B_y$. The entry of this matrix at position $(b_{x,i},b_{y,j})$ is equal to 1, if $M(w)_\za$ maps $b_{x,i}$ to $b_{y,j}$, and zero, otherwise. 

The \emph{coefficient quiver} $\Gamma(M(w))$ of $M(w)$ has as vertex set the basis $B=\cup_{x\in Q_0} B_x$ of $M(w)$ and there is an arrow $b_{x,i}\to b_{y,j}$ if  and only if the corresponding matrix entry of $M(w)_\za$ is nonzero.
\end{definition}

\begin{example}   
If $Q=\xymatrix{1\ar[r]_\za\ar@/^10pt/[rr]^\zg&2\ar[r]_\zb&3}$, and $w=\za\zb\zg^{-1}\za\zb$, then 
\[M(w)=
\xymatrix{
\kb^2\ar[r]_{\left(\begin{smallmatrix}1&0\\0&1\end{smallmatrix}\right)} \ar@/^10pt/[rr]^{\left(\begin{smallmatrix}0&1\\0&0\end{smallmatrix}\right)}&\kb^2\ar[r]_{\left(\begin{smallmatrix}1&0\\0&1\end{smallmatrix}\right)} 
&\kb^2}
\qquad
\textup{and} \qquad
\Gamma(M(w))=\xymatrix@R5pt@C3pt{
b_{1,1}\ar[rd] &&&b_{1,4}\ar[ldd]\ar[rd] \\
&b_{2,2}\ar[rd]&&&b_{2,5}\ar[rd]\\
&&b_{3,3}&&&b_{3,6}}
\]
\end{example}
Note that for any string $w$ the coefficient quiver $\Gamma(M(w))$ is of Dynkin type $\mathbb{A}$.
The poset whose Hasse diagram is a Dynkin type $\mathbb{A}$ poset is often called a fence poset in the literature. 
The submodule lattice of $M(w)$ is isomorphic to the lattice of order ideals (or down-sets) of the fence poset; see for example \cite{CanakciSchroll21}.

In this article we will only consider string modules. However, certain strings are called bands, and for these we can associate, in addition to the string module, a family of indecomposable modules, called \emph{band modules}. Every indecomposable module over a gentle algebra is either a string or a band module. For more details, we refer the reader to~\cite{BR87}.

\subsection{Morphisms between string modules}

In this section, we review a description of all morphisms between string $A$-modules given in~\cite{CB89}, and in particular the irreducible morphisms given in~\cite{BR87}. Two examples of Auslander-Reiten quivers of gentle algebras are given in section~\ref{sect examples}.  

\begin{definition}[Up-set and down-set]
 Let $w$ be a string of $A$ and let $w'$ be a subwalk of $w$. 
\begin{enumerate}
\item 
We say $w'$ is an \emph{up-set} of $w$  if the subquiver $\Gamma(M(w'))$ of $\Gamma(M(w))$ is closed under predecessors. 
\item Dually, $w'$ is a \emph{down-set} if the subquiver $\Gamma(M(w'))$ of $\Gamma(M(w))$ is closed under successors.
\end{enumerate}
\end{definition}

\begin{remark}[String submodule and string quotient module]\label{rem:string submodule and string quotient module}
Let $w=\za_1 \za_2 \dots \za_\stringlength$ be a string of $\Algebra$ and let $w'=\za_i \za_{i+1} \dots \za_j$ be a subwalk of $w$. 

Then $w'$ is an up-set of $w$ if and only if the following  two conditions hold:
\begin{enumerate}
\item $\za_{i-1}$ is inverse or $i=1$; 
\item $\za_{j+1}$ is direct or $j=\stringlength$.
\end{enumerate}
Note that $M(w')$ is a string module which is a quotient module of $M(w)$ if and only if $w'$ is an up-set of $w$.

Dually, $w'$ is an down-set of $w$ if and only if the following  two conditions hold:
\begin{enumerate}
\item $\za_{i-1}$ is direct or $i=1$;
\item $\za_{j+1}$ is inverse or $j=\stringlength$.
\end{enumerate}
Note that $M(w')$ is a string module which is a submodule of $M(w)$ if and only if $w'$ is an down-set of $w$.
\end{remark}

\begin{proposition}[Basis for morphisms 
\cite{CB89}]\label{prop: string_hom_basis}
Let $w$ and $w'$ be strings of a gentle algebra $A$ and let $M(w)$ and $M(w')$ be the corresponding string modules.
Then the following elements form a basis of $\Hom_A(M(w), M(w'))$, which we denote by $\basis(w,w')$: 
\begin{enumerate}
\item Each mono basis element has the form $M(w) \rightarrow M(w'') \hookrightarrow M(w')$ corresponding to a string $w''$ which is equivalent to $w$ and is a down-set subwalk in $w'$. 

\item Each epi basis element $M(w) \twoheadrightarrow M(w'')\rightarrow M(w')$ corresponds to a string $w''$ which is equivalent to $w'$ and is an up-set subwalk of $w$. 

\item Each basis element which is neither epi nor mono is of the form $M(w) \twoheadrightarrow M(w'') \hookrightarrow  M(w')$, where $w''$ is an up-set subwalk of $w$ and a down-set subwalk of $w'$.
\end{enumerate}

Figure \ref{fig:string_hom_basis} illustrates all three cases in Proposition~\ref{prop: string_hom_basis}.
\end{proposition}

\def\myxscale{.6}
\def\myyscale{.6}
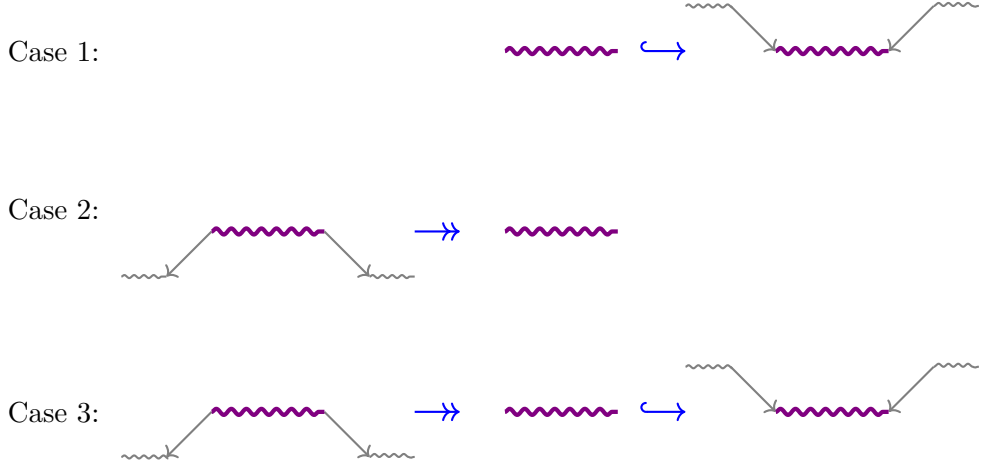
\begin{figure}
\begin{tikzpicture}[xscale=\myxscale,yscale=\myyscale]
\node at (-3,0) {Case 1:};
\coordinate (1) at (0+7,0);   
\coordinate (2) at (3-.5+7,0);

\draw[violet, ultra thick,decorate,decoration={snake,amplitude=.4mm,segment length=2mm}] 
(1) -- (2); 

\draw[thick,blue,right hook->] (5+5,0) -- (6+5,0);

\begin{scope}[shift={(13,0)}]
\coordinate (-3') at (-5+1+2,1); 
\coordinate (-2') at (-3+2,1); 
\coordinate (-1') at (-2+2,0);   
\coordinate (0') at (0.5+2,0);    
\coordinate (1') at (1.5+2,1);   
\coordinate (2') at (2.5+2,1);

\draw[gray, thick,decorate,decoration={snake,amplitude=.2mm,segment length=1.3mm}] 
(-3')-- (-2')  
(1') -- (2');
\draw[gray, thick, ->] (-2')-- (-1');
\draw[violet, ultra thick,decorate,decoration={snake,amplitude=.4mm,segment length=2mm}] (-1')-- (0')
;
\draw[gray, thick, <-] (0')-- (1'); 
\end{scope}
\node at (-3,1.5-5) {Case 2:};
\coordinate (1') at (-1.5,0-5);   
\coordinate (0) at (-0.5,0-5);    
\coordinate (1) at (0.5,1-5);   
\coordinate (2) at (3,1-5);
\coordinate (3) at (4,0-5);
\coordinate (4) at (5,0-5);
\coordinate (5) at (6-1,1-5);
\coordinate (6) at (7-1,1-5); 
\coordinate (7) at (10-3,1-5);
\coordinate (8) at (12.5-3,1-5);

\draw[gray, thick,decorate,decoration={snake,amplitude=.2mm,segment length=1.3mm}] 
(1') -- (0);
\draw[gray, thick, <-] (0)-- (1); 
\draw[violet,ultra thick,decorate,decoration={snake,amplitude=.4mm,segment length=2mm}] 
(1) -- (2);
\draw[gray, thick, ->] (2)-- (3); 
\draw[gray,thick,decorate,decoration={snake,amplitude=.2mm,segment length=1.3mm}] 
(3) -- (4);

\draw[thick,blue,->>] (5) -- (6);
\draw[violet,ultra thick,decorate,decoration={snake,amplitude=.4mm,segment length=2mm}] 
(7) -- (8);

\node at (-3,4-12) {Case 3:};
\coordinate (-1) at (-1-.5,0+3-12); 
\coordinate (0) at (0-.5,0+3-12); 
\coordinate (1) at (1-.5,1+3-12); 
\coordinate (2) at (3,1+3-12);
\coordinate (3) at (4,0+3-12);
\coordinate (4) at (5,0+3-12);

\draw[gray,thick,decorate,decoration={snake,amplitude=.2mm,segment length=1.3mm}] 
(-1) -- (0)
(3) -- (4);

\draw[gray, thick, <-] (0)-- (1);
\draw[gray, thick, ->] (2)-- (3);

\draw[violet,ultra thick,decorate,decoration={snake,amplitude=.4mm,segment length=2mm}] 
(1) -- (2);
 
\draw[thick,blue,->>] (6-1,4-12) -- (7-1,4-12);

\draw[violet,ultra thick,decorate,decoration={snake,amplitude=.4mm,segment length=2mm}] 
(9.5-1-1-.5,4-12) -- (12-1-1-.5,4-12);

\draw[thick,blue,right hook->] (14-3-1,4-12) -- (15-3-1,4-12);

\coordinate (-1') at (-1+16-4,5-12); 
\coordinate (0') at (0+16-4,5-12);    
\coordinate (1') at (1+16-4,4-12);   
\coordinate (2') at (3+16+.5-4,4-12);
\coordinate (3') at (4+16+.5-4,5-12);
\coordinate (4') at (5+16+.5-4,5-12);

\draw[gray,thick,decorate,decoration={snake,amplitude=.2mm,segment length=1.3mm}] 
(-1') -- (0')
(3') -- (4');

\draw[gray, thick, ->] 
(0')-- (1');
\draw[gray, thick, <-] 
(2')-- (3'); 

\draw[violet,ultra thick,decorate,decoration={snake,amplitude=.4mm,segment length=2mm}] 
(1') -- (2');
\end{tikzpicture}
\caption{Basis described in Proposition \ref{prop: string_hom_basis}, where the bold squiggly line corresponds to the string $w''$}\label{fig:string_hom_basis}
\end{figure}
In order to describe the irreducible morphisms between string modules, we need to discuss hooks and cohooks.

\begin{definition}[Hooks and cohooks]
Let $\alpha$ be an arrow in $Q_1$. The \emph{hook} of $\alpha$, denoted by $\hook{\alpha}$, is given by the right maximal direct string  of $(Q,I)$ starting at $x=s(\alpha)$ that does not start with $\alpha$, if $\outdegree{x}=2$. In the case that $\outdegree{x}=1$, then $\hook{\alpha}$ is the trivial string $1_x$.

The \emph{cohook} of $\alpha$, denoted by $\cohook{\alpha}$, 
is given by the left maximal  direct string of $(Q,I)$ ending at  $y=t(\alpha)$ that does not end with $\alpha$, if the indegree of $y$ is 2;  
in the case that the indegree of $y$ is $1$, then $\cohook{\alpha}$ is the trivial string $1_y$.
\end{definition}

Figures \ref{fig:hook} and \ref{fig:cohook} illustrate the hook and cohook of an arrow. 
Note that it follows from properties \ref{def gentle:itm1:degrees} and \ref{def gentle:itm2:string algebra} of $A$ that the hook and cohook of a given arrow are uniquely defined.

\def\myxscale{0.88}
\def\myyscale{0.75}
\begin{figure}[htb!]
\centering
\begin{tikzpicture}[xscale=\myxscale,yscale=\myyscale
]
\node(top) at (0,0) {$s(\alpha)$}; 
\node(Left1) at (1,-1) {$\cdot$ }; 
\node(Left2) at (2,-2) {$\cdot$ }; 
\node(Left3) at (3,-3) {$\cdot$ }; 
\draw (2.5, -2.5) 
node{$  \ddots $}; 
\node(Left4) at (4,-4) {$\cdot$ }; 
\draw[->] (top) -- (-1,-1) node[pos=0.7,right]{$\alpha$};
\draw[->] (top) -- (Left1);
\draw[->] (Left1) -- (Left2);
\draw[->] (Left3) -- (Left4);
\end{tikzpicture}
\caption{The hook of the arrow $\alpha$, starting at $s(\alpha)$.
}
\label{fig:hook}
\begin{tikzpicture}[xscale=\myxscale,yscale=\myyscale]
\node(bottom) at (0,0) {$t(\alpha)$}; 
\node(Left1) at (-1,1) {$\cdot$ }; 
\node(Left2) at (-2,2) {$\cdot$ }; 
\node(Left3) at (-3,3) {$\cdot$ }; 
\draw (-2.5, 2.5) 
node{$  \ddots $}; 
\node(Left4) at (-4,4) {$\cdot$ }; 
\draw[<-] (bottom) -- (1,1) node[pos=0.3,right]{$\alpha$};
\draw[<-] (bottom) -- (Left1);
\draw[<-] (Left1) -- (Left2);
\draw[<-] (Left3) -- (Left4);
\end{tikzpicture}
\caption{The cohook of
the arrow $\alpha$, ending at  $t(\alpha)$. 
}
\label{fig:cohook}
\end{figure}

\begin{definition}[Adding a hook and removing a cohook]\label{def:adding a hook}
Let $w$ be a nonzero string of $A$. 
Let string $w_\text{left}$ be a string defined as follows: 
\begin{enumerate}[(1)]
\item\label{def:adding a hook:add hook on left:itm1} If there is an arrow $\alpha \in Q_1$ for which $\alpha w$ is a string, then $w_\text{left} \coloneqq (\hook{\alpha})^{-1}\alpha w$. In this case, we say $w_\text{left}$ is obtained from $w$ by \emph{adding a hook on the left}. 
\item\label{def:adding a hook:zero morphism on left:itm2}  If there is no arrow $\alpha \in Q_1$ for which $\alpha w$ is a string and $w$ is a direct string, then $w_\text{left} \coloneqq 0$. 
\item\label{def:adding a hook:remove cohook on left:itm3} If there is no arrow $\alpha \in Q_1$ for which $\alpha w$ is a string and $w$ is not a direct string, then 
we can write $w=\cohook{\beta}\beta^{-1}w'$, where $\beta\in Q_1$
and $w'$ is a string.
In this case, define $w_\text{left}\coloneqq w'$, and we say $w_\text{left}$ is obtained from $w$ by \emph{removing a cohook on the left}.
\end{enumerate}

The definition of $w_\text{right}$ is dual: 
\begin{enumerate}[($1'$)]
\item\label{def:adding a hook:add hook on right:itm1} If there is an arrow $\alpha \in Q_1$ for which $w \alpha^{-1}$ is a string, then let $w_\text{right} \coloneqq  w \alpha^{-1} (\hook{\alpha})$. In this case, we say $w_\text{right}$ is obtained from $w$ by \emph{adding a hook on the right}.
\item\label{def:adding a hook:zero morphism on right:itm2}  If there is no arrow $\alpha \in Q_1$ for which $w\alpha^{-1}$ is a string and $w$ is an inverse string, then $w_\text{right} \coloneqq 0$.
\item\label{def:adding a hook:remove cohook on right:itm3} If there is no arrow $\alpha \in Q_1$ for which $w\alpha^{-1}$ is a string and $w$ is not an inverse string, then 
we can write $w=w'\beta \, \cohook{\beta}^{-1}$, where $\beta\in Q_1$
and $w'$ is a string.
In this case, define $w_\text{right}\coloneqq w'$, and we say $w_\text{right}$ is obtained from $w$ by \emph{removing a cohook on the right}.
\end{enumerate}
\end{definition}

Formally, one can define concatenation of strings (including trivial strings) via a pair of so called \emph{sign functions} (see \cite{BR87} for more details). From this we can deduce that the definition of adding a hook on the left (respectively right) of a string is indeed well-defined.

Note that if $w_\text{left}$ is obtained by adding a hook on the left, then $w$ is a down-set of $w_\text{left}$, and if  $w_\text{left}$ is obtained by removing a cohook, then $w_\text{left}$ is an up-set of $w$. In particular, there is a nonzero morphism $M(w) \rightarrow M(w_\text{left})$, which is the natural inclusion in the former case and the natural projection in the latter case. We have a similar statement involving $w_\text{right}$. The following proposition states that all irreducible morphisms between string modules are of this form.

\begin{proposition}[{\cite{BR87}}]
\label{prop:Butler Ringel:at most two arrows starting at M and ending at M}
Let $A$ be a gentle algebra, $w$ a nonzero string of $A$, and $M(w)$ the corresponding string module. 
Then, in the Auslander--Reiten quiver of $A$, there are at most two arrows starting at $M(w)$. These arrows correspond to the natural inclusions or projections $M(w) \rightarrow M(w_\text{left})$, provided $w_\text{left} \neq 0$, 
and $M(w) \rightarrow M(w_\text{right})$, provided $w_\text{right} \neq 0$.

Similarly, there are at most two arrows ending at $M(w)$. These arrows correspond to $M(u) \to M(w)$, where $u_{\text{left}}=w$, and $M(v) \to M(w)$, where $v_{\text{right}}=w$.
\end{proposition}

Given a string $w$, we will refer to the irreducible morphisms of the form $M(w) \to M(w_\text{left})$, 
and of the form $M(u) \to M(w)$, where $w=u_\text{left}$,
as being \emph{irreducible morphisms on the left} associated to $w$. 
\emph{Irreducible morphisms on the right} associated to $w$ are defined similarly.

\begin{lemma}\label{lem:add-hook-remove-cohook}
 Let $w$ be a nonzero string. Then we can say the following about irreducible morphisms starting at $M(w)$ and ending at $M(w)$. 
\begin{enumerate}[(1)]
\item\label{lem:add-hook-remove-cohook:itm1} We can add a hook on the left (respectively hook on the right) of $w$ if and only if $\alpha w$ (respectively $w\alpha^{-1}$) is a string for some $\alpha \in Q_1$.

\item\label{lem:add-hook-remove-cohook:itm2} We can obtain $w$ by removing a cohook on the left (respectively right) of another string $v$ if and only if there is $\alpha \in Q_1$ such that $\alpha^{-1}w$ (respectively $w\alpha$) is a string.

\item\label{lem:add-hook-remove-cohook:itm3} If $w$ is direct (respectively inverse), then 
it is impossible to 
remove a cohook on the left (respectively right) of $w$. 
\item\label{lem:add-hook-remove-cohook:itm4} If $w$ is direct (respectively inverse), then $w$ cannot be obtained by adding a hook on the right (respectively left) of another string.

\item\label{lem:add-hook-remove-cohook:itm5}
If $w$ is not direct (respectively, not inverse) then there is an irreducible morphism on the left  (respectively right) starting at $M(w)$. 

\item\label{lem:add-hook-remove-cohook:itm6} 
If $w$ is not inverse (respectively, not direct),  
then there is an irreducible morphism on the left (respectively right) ending at $M(w)$. 

\end{enumerate}
\end{lemma}
\begin{proof} 
We only prove part \ref{lem:add-hook-remove-cohook:itm6} for when $w$ is not an inverse string. The rest of the proofs are similar, and they follow from Definition \ref{def:adding a hook} and Proposition \ref{prop: string_hom_basis}.

First, suppose there is $\zb\in Q_1$ such that $\zb^{-1} w$ is a string.  Let $v=\cohook{\zb} \beta^{-1} w$. 
Then $w=v_{\text{left}}$ as in Definition \ref{def:adding a hook}\ref{def:adding a hook:remove cohook on left:itm3}, and we have an irreducible morphism $M(v)\to M(w)$ which is given by removing the cohook of $\zb$ on the left of $v$.

Now suppose there is no such $\beta$. Since $w$ is not inverse, we can write $w=\hook{\za}^{-1} \za v$ for some arrow $\za$ and nonzero string $v$. Then $w=v_{\text{left}}$ as in Definition \ref{def:adding a hook}\ref{def:adding a hook:remove cohook on left:itm3}, and we have an irreducible morphism $M(v)\to M(w)$ which is given by adding the hook of $\za$ on the left of $v$.
\end{proof}

\subsection{Extensions between string modules}

We will now describe a basis of the extension space between string modules, following~\cite{CPS21}, see also~\cite{CombinatoricsOfGentleAlgebras}. This description is in terms of arrow extensions and overlap extensions.

\begin{definition}
[Arrow extension]\label{def:arrow_extension}
Let $w$ and $v$ be strings. 
Suppose there is an arrow $a \in Q_1$ such that $vaw$ is a string. 
Then the (nonsplit) short exact sequence
\[
0 \to M(w) \to  M(vaw) \to M(v) \to 0
\] is called 
an \emph{arrow extension} of $M(v)$ by $M(w)$. See Figure~\ref{fig:arrow extension} for a schematic. 
\end{definition}

\begin{figure}[htbp]
\begin{tikzpicture}[node distance=1cm and 1.5cm]
\coordinate[label=left:{}] (1);
\coordinate[right=1cm of 1] (2);

\coordinate[right=1.3cm of 1] (3);
\coordinate[right=1cm of 3] (4);

\coordinate[above right=.4cm of 4] (5);
\coordinate[right=1cm of 5] (6);

\coordinate[below right=1cm of 6] (7);
\coordinate[right=1cm of 7] (8);

\coordinate[right=3cm of 4] (9);
\coordinate[right=1cm of 9] (10);

\coordinate[right=.3cm of 10] (11);
\coordinate[right=1cm of 11] (12);

\coordinate[right=-.3cm of 1] (1');
\coordinate[right=-1cm of 1'] (2');
\coordinate[right=-.7cm of 2',label=right:{$0$}] (3');

\coordinate[right=.3cm of 12] (12');
\coordinate[right=1cm of 12'] (13');
\coordinate[right=.3cm of 13',label=right:{$0$}] (14');

\draw[thick,->,blue] (3)--(4);
\draw[thick,->,blue] (9)--(10);
\draw[thick,->,blue] (2')--(1');
\draw[thick,->,blue] (12')--(13');

\draw[thick,decorate,decoration={snake,amplitude=.2mm
,
segment length=
1.3mm
}] 
(1)-- node [anchor=south,scale=.9]{$w$} (2) 
(5)-- node [anchor=south,scale=.9]{$v$} (6) 
(7)-- node [anchor=north,scale=.9]{$w$} (8) 
(11)-- node [anchor=north,scale=.9]{$v$} (12);

\draw[violet,ultra thick,->] 
(6)-- (7) node [pos=0.4,right,violet]{$a$};
\end{tikzpicture}
\caption{Arrow extension}\label{fig:arrow extension}
\bigskip 
\begin{tikzpicture}
\coordinate[label=left:{}] (1);
\coordinate[right=1cm of 1] (2);
\coordinate[above right=.8cm of 2,label=right:{}] (3);
\coordinate[right=1cm of 3] (4);
\coordinate[below right=.8cm of 4] (5);
\coordinate[right=1cm of 5] (6);

\coordinate[right=11mm of 6] (6''');

\coordinate[below=0.8cm of 6'''] (7);
\coordinate[right=1cm of 7] (8);
\coordinate[below right=.8cm of 8] (9);
\coordinate[right=1cm of 9] (10);
\coordinate[below right=.8cm of 10] (11);
\coordinate[right=1cm of 11] (12);
\coordinate[right=10.4cm of 1] (13);
\coordinate[right=1cm of 13] (14);
\coordinate[below right=.8cm of 14] (15);
\coordinate[right=1cm of 15] (16);
\coordinate[above right=.8cm of 16] (17);
\coordinate[right=1cm of 17] (18);

\coordinate[right=.25cm of 6] (6'');
\coordinate[right=.5cm of 6''] (7'');

\coordinate[right=3.5cm of 6,label=right:{$\oplus$}] (19);

\coordinate[above=0.8cm of 6'''] (7');
\coordinate[right=1cm of 7'] (8');
\coordinate[above right=.8cm of 8'] (9');
\coordinate[right=1cm of 9'] (10');
\coordinate[above right=.8cm of 10'] (11');
\coordinate[right=1cm of 11'] (12');
\coordinate[above right=.8cm of 12'] (13');

\coordinate[right=54mm of 6] (12'');
\coordinate[right=.5cm of 12''] (13'');

\coordinate[right=-.3cm of 1] (1');
\coordinate[right=-0.5cm of 1'] (2');
\coordinate[right=-.6cm of 2',label=right:{$0$}] (3');

\coordinate[right=.3cm of 18] (18');
\coordinate[right=.5cm of 18'] (19');
\coordinate[right=.2cm of 19',label=right:{$0$}] (19''');

\draw[thick,->,blue] (2')--(1');
\draw[thick,->,blue] (6'')--(7'');
\draw[thick,->,blue] (12'')--(13'');
\draw[thick,->,blue] (18')--(19');

\draw[thick,decorate,decoration={snake,amplitude=.2mm,segment length=
1.3mm
}] 
(1)-- node [anchor=north,scale=.9]{$w_\overlapLEFT$} (2) 
(5)-- node [anchor=north,scale=.9]{$w_\overlapRIGHT$} (6) 
(7)-- node [anchor=south,scale=.9]{$v_\overlapLEFT$} (8) 
(7')-- node [anchor=north,scale=.9]{$w_\overlapLEFT$} (8') 
(11)-- node [anchor=north,scale=.9]{$w_\overlapRIGHT$} (12) 
(11')-- node [anchor=south,scale=.9]{$v_\overlapRIGHT$} (12') 
(13)-- node [anchor=south,scale=.9]{$v_\overlapLEFT$} (14) 
(17)-- node [anchor=south,scale=.9]{$v_\overlapRIGHT$} (18);

\draw[violet,ultra thick,decorate,decoration={snake,amplitude=.4mm,segment length=2mm}] 
(3)-- node [anchor=south,scale=.9]{$e$} (4)
(9)-- node [pos=0.5,below=1pt]{$e$} (10)
(9')-- node [anchor=south,scale=.9]{$e$} (10')
(15)-- node [pos=0.5,below=1pt]{$e$} (16);

\draw[thick,->] (3)--node [pos=0.7, above=1pt]{$a^{-1}$} (2); 
\draw[thick,->] (4)-- (5) node [pos=0.2, right=3pt]{${b}$} ;
\draw[thick,->] (8)--node [pos=0.3, right=1pt]{${c}$}(9);
\draw[thick,->] (9')-- node [pos=0.7, above=1pt]{$a^{-1}$}(8');
\draw[thick,->] (10)--node [pos=0.2, right=3pt]{${b}$}(11);
\draw[thick,->] (11')--node [pos=0.7, above=1pt]{$d^{-1}$}(10');
\draw[thick,->] (14)--node [pos=0.3, right=1pt]{${c}$}(15);
\draw[thick,->] (17)--node [pos=0.7, above=1pt]{$d^{-1}$}(16);
\end{tikzpicture} 
\caption{Overlap extension}\label{fig:overlap extension}
\end{figure}

\begin{definition}
[Overlap extension]\label{def:overlap_extension}
Suppose we have strings 
\[ 
\text{
$w=w_\overlapLEFT \, a^{-1} \, e \, b  \, w_\overlapRIGHT$ and  
$v=v_\overlapLEFT \,  c \, e \, d^{-1} \, v_\overlapRIGHT$}
\]

where $a, b, c, d \in Q_1 \cup \{0\}$, 
$e$ is a nonzero string,
and 
$w_\overlapLEFT, w_\overlapRIGHT, v_\overlapLEFT, v_\overlapRIGHT$ are strings such that the following constraints are satisfied:
\begin{enumerate}[(1)]
\item \label{enum:a nonzero or c nonzero}
$a \neq 0 $ or $c \neq 0$;
\item \label{enum:b nonzero or d nonzero}
$b \neq 0 $ or $d \neq 0 $, and
\item \label{enum:c trivial overlap implies da and cb are strings}
if the string $e$ is a trivial string $1_x$ where $x=s(a)=s(b)=t(c)=t(d)$, then $da,cb \notin I$ (whenever these arrows exist). 
\end{enumerate} 
We call $e$ an \emph{overlap} of the string $v$ by $w$.

Let $E_1$ and $E_2$ denote the string modules corresponding to the strings 
\[ 
\text{
$e_1 = w_\overlapLEFT \, a^{-1} \, e \, d^{-1}  \, w_\overlapRIGHT$ and  
$e_2 = v_\overlapLEFT \,  c \, e \, b \, w_\overlapRIGHT$}.
\]
Then 
the (nonsplit) short exact sequence
\[
0 \to M(w) \to  E_1 \oplus E_2 \to M(v) \to 0
\]
is called an \emph{overlap extension of $M(v)$ by $M(w)$ with overlap $e$}. 
See Figure~\ref{fig:overlap extension} for a schematic. 
\end{definition}

Note that if the overlap $e$ is a trivial string, then we must have $a\neq b$ and $c \neq d$, since $w$ and $v$ are reduced walks. Moreover, condition \ref{enum:c trivial overlap implies da and cb are strings}
guarantees that $e_1$ and $e_2$ are indeed strings.

For example, consider the quiver 
$\xymatrix{1\ar[r]^c&2\ar[r]^b&3}$, with $cb=0$. Take the strings $w=b, v=c$, and let $e=1_2$ be the trivial string at vertex $2$. 
Then conditions \ref{enum:a nonzero or c nonzero} and \ref{enum:b nonzero or d nonzero} are satisfied, but not condition \ref{enum:c trivial overlap implies da and cb are strings} because $cb \in I$. 
The fact that $cb \in I$ means that $e_2=cb$ is not a string. 
Indeed, the modules $M(c),M(b)$ 
corresponding to $v$ and $w$ 
are both projective so there is no extension between them.

\begin{remark}
\label{rem:first string w is an up-set}
The overlap $e$ in Definition~\ref{def:overlap_extension} is an up-set of $w$ and a down-set of $v$. In particular, Proposition~\ref{prop: string_hom_basis} implies  $\Hom (M(w), M(v)) \neq 0$.
\end{remark}

The following theorem describes the dimension of the extension space between string modules as the number of arrow extensions and overlap extensions.

\begin{theorem}[{\cite[Theorem A]{CPS21}}]
\label{thm:CPS21}
Let $M(v)$ and $M(w)$ be string modules. 
Then the set of all overlap extensions and arrow extensions of $M(v)$ by $M(w)$ form a basis of $\Ext^1(M(v),M(w))$. 
\end{theorem}


\section{Maximal almost rigid modules}\label{sec:MAR-algebra}

Throughout, $A = \kb Q/I$ will denote a gentle algebra. In this section, we give the definition of maximal almost rigid modules over $A$. This naturally generalizes the definition of maximal almost rigid modules over path algebras of Dynkin type $\mathbb{A}$, given in~\cite{BGMS19}.
We will also prove the existence of a specific maximal almost rigid module using projective (respectively injective) modules, and describe all indecomposable modules which are summands of any maximal almost rigid module. 

\subsection{Maximal almost rigid modules}

\begin{definition}\label{def:mar}
Let $A$ be a gentle algebra. A basic $A$-module $T$ is {\it almost rigid} if 
\begin{enumerate}[label=(M\arabic*)]
\item \label{item:M1} 
$T$ is a direct sum of string modules, 
\item \label{item:M2} 
for each pair $M, N$ of indecomposable summands of $T$, if $0 \to N \to E \to M \to 0$ is a short exact sequence then it either splits or the middle term $E$ is indecomposable.
\end{enumerate}    

An $A$-module $T$ is {\it maximal almost rigid} (MAR, for short) if 
$T$ is an almost rigid $A$-module and it is maximal with respect to~\ref{item:M2}, that is,
\begin{enumerate}[label=(M3)]
\item  
$T\oplus L$ is not almost rigid, for every nonzero string $A$-module $L$.
\end{enumerate}    
\end{definition}

We denote by $\mar{A}$ the set of all maximal almost rigid modules over $A$.

\begin{tikzpicture}[scale=1]
\draw[line width=2mm, red, opacity=0.6] (0,0) circle (8mm);
\draw[line width=2mm, red, opacity=0.6] (-0.6,-0.6)--(0.6,0.6);
\node at (0,0) {bands};

\draw[line width=2mm, forGreen, opacity=0.6] (2,0) circle (8mm);
\node at (2,0) {strings};
\end{tikzpicture}

\begin{remark}
We exclude band modules in Definition~\ref{def:mar} in order to avoid infinitely many direct summands in an MAR module. 
Take for example the Kronecker algebra $A=\kb Q$ with 
\[Q=\begin{tikzcd} 1 \arrow[r,shift left] \ar[r,shift right] & 2. \end{tikzcd}\]
Then $T = \oplus_{\lambda \in \kb^*} M_\lambda$, where $M_\lambda$ is the band module given by 
\[\begin{tikzcd} \kb \arrow[r,shift left]{} {\text{1}}
\arrow[r,shift right]{}[below]{\lambda} & \kb, \end{tikzcd}\]
satisfies condition \ref{item:M2}. 
Indeed, if $\lambda \neq \lambda'$, then $\Ext^1(M_\lambda, M_{\lambda'})=0$. On the other hand, $\Ext^1(M_\lambda, M_\lambda)$ has dimension one (by the Auslander-Reiten formula), and it is generated by a non-split exact sequence with indecomposable middle term 
\[\xymatrix@C35pt{\kb^2 \ar@<2pt>[r]^{\left(\begin{smallmatrix}
    1&0\\0&1
\end{smallmatrix}\right)}\ar@<-2pt>[r]_{\left(\begin{smallmatrix}
    \zl&1\\0&\zl
\end{smallmatrix}\right)} &\kb^2}.\] 
\end{remark}

The following lemma will be useful in the proof of Proposition~\ref{prop:all-exts}. 

\begin{lemma}\label{lem:dim-Exts}
Let $A = \kb Q/I$ be a gentle algebra, and $M,N$ be (possibly isomorphic) string $A$-modules. Suppose there are no overlap extensions between $M$ and $N$. 
Then the following properties hold:
\begin{enumerate}[(1)]
\item\label{lem:dim-Exts:item1} $\Ext^1(M,N)$ is generated by arrow extensions. 
\item\label{lem:dim-Exts:item2} $\dim~\Ext^1(M,N) \leq 2$.
\item\label{lem:dim-Exts:item3}$\dim~\Ext^1(M,N)=2$ if and only if $M = S(i)$ and $N=S(j)$, where $i\neq j$ and  $\begin{tikzcd} i \arrow[r,shift left] \ar[r,shift right] & j \end{tikzcd}$ is a subquiver of $Q$. 
\end{enumerate}
\end{lemma}

\begin{proof}
Statement \ref{lem:dim-Exts:item1} follows from Theorem \ref{thm:CPS21}.

Let $w$ (respectively $v$) be the string associated to $M$ (respectively $N$). 
Since there are no overlap extensions between $M$ and $N$, the 
dimension of $\Ext^1 (M,N)$ is at most the number of arrows in $Q$ which start at $t(w)$ and end at $s(v)$. This number is at most $2$ since $A$ is a gentle algebra, proving part \ref{lem:dim-Exts:item2}.

Now suppose there are two distinct arrows $a,b$ in $Q$ from $t(w)$ to $s(v)$. Suppose $w$ is a non-trivial string, and write $w=w_1 \ldots w_r$, with each $w_i \in Q_1 \cup Q_1^{-1}$ and $r \geq 1$. We claim that $\Ext^1(M,N)$ is one-dimensional in this case. 
Indeed, if $w_r \in Q_1$, then either $w_r a \in I$ or $w_r b \in I$, since $A$ is gentle. Similarly, if $w_r \in Q_1^{-1}$, then $w_r = a$ or $w_r =b$. In both situations, there are either no arrows or there is a unique arrow, say $a$, for which $wav$ is a string, which proves our claim. 
Similarly,  $\Ext^1(M,N)$ is one-dimensional when $v$ is non-trivial. Thus if $\Ext^1(M,N)$ is two-dimensional then we must have $M=S(i)$ and $N=S(j)$, where $i \neq j$ since a gentle algebra cannot have two loops at a vertex. 

The converse is straightforward; we have the following non-equivalent arrow extensions of $M=S(i)$ by $N=S(j)$ 
\[0 \to S(j) \to M(a) \to S(i) \to 0 ~ \text{ and } ~
0 \to S(j) \to M(b) \to S(i) \to 0. 
\]
\end{proof}

The following proposition tells us that it is enough to consider  overlap extensions when checking whether a module is almost rigid or not. 

\begin{proposition}\label{prop:all-exts}
Let $M$ and $N$ be string $A$-modules.  Then  $\Ext^1(M,N)$ is generated by arrow extensions if and only if the middle term of any non-split extension $0 \to N \to E \to M \to 0$ is indecomposable. 
\end{proposition}
\begin{proof} First assume $\Ext^1(M,N)$ is generated by arrow extensions. We want to show that the middle term of any non-split extension $0 \to N \to E \to M \to 0$ is indecomposable.  This is clear when $\Ext^1 (M,N)$ is one-dimensional. If the Ext-space is two dimensional, then by Lemma~\ref{lem:dim-Exts}\ref{lem:dim-Exts:item3}, we have that $M$ and $N$ are two simple modules at different vertices of $Q$. Hence the dimension of $E$ is one at each of these two vertices and zero elsewhere, and it follows that $E$ is either $M \oplus N$ or indecomposable. 

The converse follows immediately from Theorem~\ref{thm:CPS21} 
and the fact that overlap extensions do not have an indecomposable middle term.
\end{proof}

\subsection{Summands of every MAR module}\label{sec:required summands}

We say that a string $A$-module $M$ is a \emph{required} summand if $M$ is a summand of every MAR module over $A$. 
Let $\required{Q,I}$ denote the set of required summands for $A= \kb Q/I$. In this section, we will give the list of elements in this set.

We need the following result, whose proof relies on a geometric model for $\textup{mod}\,{A}$ which is described in Section~\ref{sec:geometric model}.

\begin{proposition}\label{prop can complete ar to mar}
   If $\overline {T}$ is an almost rigid module over a gentle algebra $A$ then there exists a maximal almost rigid module  $T$ that contains $\overline{T}$ as a direct summand. 
\end{proposition}
\begin{proof}
We first show that the number of summands in an almost rigid module is bounded. 
Let $S$ be the surface with marked points and tiling $P$ associated to the algebra $A$ in Section \ref{sec:tiled surfaces}.  
String $A$-modules correspond to certain arcs in the surface (see Theorem~\ref{thm:permissible}), and if two such arcs cross transversally then the corresponding string modules have an overlap extension (see Proposition \ref{prop:overlap-cross}). 
By Definition~\ref{def:overlap_extension}, such a pair of modules is not allowed in an almost rigid module.  In particular an almost rigid $A$-module with $t$ indecomposable direct summands corresponds to a set of $t$ arcs without any transversal crossings. Therefore, the number $t$ is at most the number $n$ of arcs in a triangulation of the surface. 
    
Now suppose $T_0=\overline{T}$ is an almost rigid module.
If there is no string module $T'$ such that $T_0\oplus T'$ is almost rigid, then $T_0$ is maximal almost rigid and we are done. Otherwise, the module $T_1=T_0\oplus T'$ is almost rigid and we can repeat the argument. 
This procedure will stop, because an almost rigid module cannot have more than $n$ summands. 
\end{proof}

\begin{lemma}\label{lem: required w is direct or inverse}
If $M(w)$ is required, then w is a direct or inverse string.
\end{lemma}
\begin{proof}
Assume $w$ is not direct nor inverse. 
Then $w$ has a hill subwalk  $\za^{-1} \beta$, with $i=s(\beta)=s(\za)$, or a valley subwalk $\za \beta^{-1}$, with  $i = t(\beta) = t(\za)$; in either case, we have $\za , \beta \in Q_1$ and $\za \neq \beta$. 
If $w$ has a hill, then there is an overlap extension of the simple module $S(i)$ by $M$; if $w$ has a valley, there is an overlap extension of $M$ by $S(i)$. 
In either case, $M$ and $S(i)$ cannot both be summands of the same MAR module. 
Since $S(i)$ has no overlap self-extensions, Proposition~\ref{prop can complete ar to mar} implies that $S(i)$ can be completed to an MAR module $T$. 
Given that $M$ is not a summand of $T$, it follows that $M$ is not in $\required{Q,I}$. 
\end{proof}

\begin{lemma}
\label{lem:M in every mar iff at most one irreducible morphism starting and ending at M}
A string $A$-module $M$ is 
in $\required{Q,I}$ 
if and only if there is at most one irreducible morphism starting at $M$ and at most one irreducible morphism ending at~$M$.
\end{lemma}
\begin{proof}
Let $M$ be a string module and $w$ the corresponding string. First, we prove the `only if' part of the statement. Suppose there are two irreducible morphisms starting at $M$. Using Proposition~\ref{prop can complete ar to mar}, in order to  show that $M$ is not 
in $\required{Q,I}$, it suffices to exhibit a string module $N$ without overlap self-extensions such that $M\oplus N$ is not almost rigid.

If $w$ is neither direct nor inverse, then $M$ is not required due to Lemma \ref{lem: required w is direct or inverse}. So assume, without loss of generality, that $w$ is direct. First suppose $w$ is trivial, that is, $w=1_x$, for some $x \in Q_0$. 
Any irreducible morphism starting at $M=S(x)$ must correspond to adding a hook. 
Since there are two irreducible morphisms starting at $M$, 
there must be two distinct arrows $\alpha,\beta \in Q_1$ with $t(\alpha)=t(\beta) = x$: 
\begin{center}
\begin{tikzpicture}[xscale=1,yscale=1,>=latex]
\def\posetedgecolor{blue}
\def\posetedgenegativeslope{red}
\node(1) at (1,1) {$1$}; 
\node(2) at (2,0) {$x$}; 
\node(3) at (3,1) {$3$}; 
\draw[->] (1) -- (2) node[pos=0.3,right=3pt]{\color{red}{${\alpha}$}};  
\draw[<-] (2) -- (3) node[pos=0.3,right=3pt]{\color{forGreen}{${\beta}$}}; 
\end{tikzpicture}
\end{center}

Then there is an overlap extension of $N\coloneqq M(\alpha \beta^{-1})$ by $M=S(x)$:
\[0 \to S(x) \to  M({\beta^{-1}}) \oplus M(\alpha) \to M(\alpha\beta^{-1}) \to 0,\] 
so $M\oplus N$ is not almost rigid. 
We will now prove that $N$ has no overlap self-extensions. 
By Remark \ref{rem:first string w is an up-set}, the overlap would have to be an up-set and a down-set of $\alpha \beta^{-1}$. Thus the overlap would be $\alpha \beta^{-1}$, but this is impossible, because the items \ref{enum:a nonzero or c nonzero} and  
\ref{enum:b nonzero or d nonzero} of 
Definition~\ref{def:overlap_extension} would not be satisfied.

Now suppose $w$ is not trivial, and write $w=\alpha_1 \ldots \alpha_r$, where $r\geq 1$. 
By assumption, $M=M(w)$ is the start of two irreducible morphisms, and each of them is given by adding a hook or removing a cohook as described in Definition~\ref{def:adding a hook}. 
Since $w$ is direct, it is impossible to remove a cohook on the left of $w$ (see Lemma~\ref{lem:add-hook-remove-cohook}\ref{lem:add-hook-remove-cohook:itm3}), so one of the two irreducible morphisms corresponds to adding a hook on the left of $w$ (see Lemma~\ref{lem:add-hook-remove-cohook}\ref{lem:add-hook-remove-cohook:itm1}). 
In particular, there exists an arrow $\beta \in Q_1$ such that $\beta w$ is a string.
Consider the string module $N$ given by $v=\beta \za_1 \ldots \za_{r-1}$ (if $r=1$, then the string $v$ is simply $\beta$). 
Then there is an overlap extension of $N$ by $M$: 
\[0 \to M(\alpha_1 \dots \alpha_r) \to  M(\alpha_1 \dots \alpha_{r-1})  \oplus M(\beta \alpha_1 \dots \alpha_{r-1} \alpha_r) \to 
M(\beta \alpha_1 \dots \alpha_{r-1})  \to 0,\]
so $M\oplus N$ is not almost rigid.

We will now prove that $N$ has no overlap self-extensions.
For contradiction, suppose $N$ has an overlap self-extension with overlap $e$. 
Then $e$ is a subwalk of $v$ which is both a down-set and an up-set by Remark \ref{rem:first string w is an up-set}. Since $v$ is direct, $e$ must be the entire string $v$ 
or $e=1_x$, with $x=t(\alpha_{r-1}) = s(\beta)$. 
But the former contradicts items 
 \ref{enum:a nonzero or c nonzero} and  
\ref{enum:b nonzero or d nonzero} of 
Definition~\ref{def:overlap_extension}. So suppose we are in the latter case. 
By item~\ref{enum:c trivial overlap implies da and cb are strings} in Definition~\ref{def:overlap_extension}, we must have $\alpha_{r-1} \beta \neq 0$. If $\beta = \alpha_r$, this would imply the existence of an oriented cycle with no relations, which is a contradiction since $A$ is finite dimensional. Therefore $\beta \neq \alpha_r$. But this contradicts the fact that $\alpha_{r-1} \alpha_r \neq 0$ and $A$ is gentle (condition~\ref{def gentle:itm2:string algebra}). 
Thus $N$ has no overlap self-extensions.

The proof that $M$ is not a required summand of every MAR module if there are two irreducible morphisms ending at $M$ is similar.
This completes the proof of the `only if' part of the statement.

\smallskip

Conversely, suppose $M$ is not in $\required{Q,I}$. Then there is an MAR module with an indecomposable summand $N$ such that $M \oplus N$ is not almost rigid. We can assume, without loss of generality, that there is an overlap extension of the form $0 \to M \to E_1 \oplus E_2 \to N \to 0$. 
We claim there are two irreducible morphisms starting at $M$.

Let $a,b,c,d$, $e$, $w_{\overlapLEFT}$, and $w_{\overlapRIGHT}$ be as in Definition \ref{def:overlap_extension}. 

Suppose first that $w$ is not a trivial string. 
If $w$ is not direct nor inverse, then Lemma \ref{lem:add-hook-remove-cohook}\ref{lem:add-hook-remove-cohook:itm5} guarantees two irreducible morphisms starting at $M$. 

Next, suppose $w$ is direct.  Then $w$ is not inverse and so there is an irreducible morphism on the right starting at $M$ by Lemma \ref{lem:add-hook-remove-cohook}\ref{lem:add-hook-remove-cohook:itm5}. 
In addition, since $w$ is direct, the inverse arrow $a^{-1}$ does not exist in $w$, and so the arrow $c$ must exist by Definition \ref{def:overlap_extension}. 
Thus $cw=cebw_{\overlapRIGHT}$ is a string, and therefore there is an irreducible morphism given by adding the hook of $c$ on the left of $w$ by Definition \ref{def:adding a hook}\ref{def:adding a hook:add hook on left:itm1}.

Similarly, if $w$ is inverse, then $w$ is not direct, and so there is an irreducible morphism on the left starting at $M$ by Lemma \ref{lem:add-hook-remove-cohook}\ref{lem:add-hook-remove-cohook:itm5}. 
Moreover, the arrow $b$ does not exist in $w$ since $w$ is inverse. Then the arrow $d$ exists, and thus $wd^{-1}$ is a string. Hence there is an irreducible morphism given by adding the hook of $d$ on the right of $w$ by Definition \ref{def:adding a hook}\ref{def:adding a hook:add hook on right:itm1}. 

Now suppose $w$ is a trivial string. Then the arrows $a$ and $b$ don't exist, and so the arrows $c$ and $d$ must exist and again give rise to two irreducible morphisms corresponding to adding their respective hooks.
\end{proof}

\begin{proposition}[Summands of every MAR module]
\label{prop:required summands}
A string $A$-module $M$ is 
a summand of every MAR module 
if and only if one of the following conditions holds:
\begin{enumerate}[(1)]
\item\label{prop:required summands:itm1} $M=M(w)$, where $w$ is a
maximal direct 
nontrivial string. 
\item\label{prop:required summands:itm2} $M=S(i)$, where $i$ has degree $1$ in $Q$. 
\item \label{prop:required summands:itm3}
$M=S(i)$, where $i$ has degree $2$ and there is a direct string $ab$ 
such that $i=t(a)=s(b)$.

\item \label{prop:required summands:itm4}
$M=S(i)$, where $i$ has degree $2$ and there is a path $ab \in I$ such that $i=t(a)=s(b)$.
\end{enumerate}
\end{proposition}

\begin{remark}
Cases \ref{prop:required summands:itm3} and \ref{prop:required summands:itm4} can be combined into the following single case. $M=S(i)$ where $\indegree{i} = 1 = \outdegree{i}$. We prefer to state them separately, because they have different geometric interpretations; 
see the list of bijections at the start of the proof of Theorem~\ref{thm:required-model}. 
\end{remark}

The following corollary is a reformulation of Proposition~\ref{prop:required summands}.

\begin{corollary}\label{cor:required summands}
A string module $M$ is a required summand if and only if $M$ is a simple module $M=S(i)$ where $\indegree{i} \leq 1$ and $\outdegree{i} \leq 1$, or $M=M(w)$ where $w$ is 
a maximal direct string which is not trivial.
\end{corollary}

\begin{proof}
[Proof of Proposition~\ref{prop:required summands}]
By Lemma~\ref{lem:M in every mar iff at most one irreducible morphism starting and ending at M}, 
an indecomposable $A$-module $M$ is 
in $\required{Q,I}$
if and only if
there is at most one irreducible morphism starting at $M$ and at most one irreducible morphism ending at $M$.

First, assume $M$ is a simple module $S(i)$ at vertex $i$. 
An irreducible morphism starting at $S(i)$ is obtained by adding the hook of an arrow whose target is the vertex $i$.
Therefore, there is at most one irreducible morphism starting at $S(i)$ if and only if $\indegree{i} \leq 1$. 
Similarly, an irreducible morphism ending at $S(i)$ is obtained by removing the cohook of an arrow whose start is $i$.  
So there is at most one irreducible morphism ending at $S(i)$ if and only if $\outdegree{i} \leq 1$. 
This proves cases \ref{prop:required summands:itm2}, \ref{prop:required summands:itm3}, and \ref{prop:required summands:itm4} of the proposition. 

\smallskip 

Now, assume $M$ is not a simple module, and let
$w=w_1 w_2 \dots w_\stringlength$ be the string 
which corresponds to $M$, where $\stringlength \geq 1$. 
Lemma \ref{lem: required w is direct or inverse} tells us that if 
$w$ is neither direct nor inverse then 
$M$ is not in $\required{Q,I}$.
Without loss of generality, let $w$ be a direct string. 
Since $w$ is not a trivial string, $w$ is not an inverse string. 
In the rest of the proof, we will show that $w$ must be maximal.

Since $w$ is not an inverse string, Lemma \ref{lem:add-hook-remove-cohook}\ref{lem:add-hook-remove-cohook:itm5} guarantees an irreducible morphism on the right starting at $M$.  
Since $w$ is a direct string, we cannot remove a cohook on the left, by Lemma \ref{lem:add-hook-remove-cohook}\ref{lem:add-hook-remove-cohook:itm3}. By Lemma \ref{lem:add-hook-remove-cohook}\ref{lem:add-hook-remove-cohook:itm1}, we can add a hook on the left of $w$ if and only if there is $\alpha \in Q_1$ such that $\alpha w$ is a string, that is, $w$ is not left maximal. This shows that there are two irreducible morphisms starting at $M(w)$ if and only if $w$ is not left maximal. 

\smallskip 

Since $w$ is not an inverse string, Lemma 
\ref{lem:add-hook-remove-cohook}\ref{lem:add-hook-remove-cohook:itm6} guarantees a morphism on the left ending at $M$. 
Since $w$ is a direct string, $w$ cannot be the result of adding a hook on the right of another string, by Lemma \ref{lem:add-hook-remove-cohook}\ref{lem:add-hook-remove-cohook:itm4}. 
We can remove a cohook on right of a string $v$ to obtain $w$ if and only if there is an arrow $\beta$ such that $w\beta$ is a string (see Lemma \ref{lem:add-hook-remove-cohook}\ref{lem:add-hook-remove-cohook:itm2}), that is, $w$ is not right maximal.
This shows that there are two irreducible morphisms ending at $M(w)$ if and only if $w$ is not right maximal.

Hence, there is at most one irreducible morphism starting at $M(w)$ and at most one irreducible morphism ending at $M(w)$ 
if and only if $w$ is maximal. This proves case \ref{prop:required summands:itm1} of the proposition.
\end{proof}

\begin{example}
\label{ex:required summands:as modules}
Consider $A=\kb Q/I$ with $Q$ the quiver 
 $\vcenter{\vbox{\xymatrix@R1pt{&4\ar[dr]^\zg\\
 1\ar[ru]^\zb&&2\ar[ll]^\za\ar[r]^\zd&3
}}}$ and  $I=\langle\za\zb, \zg\za\rangle$. 
Then we have all types of required summands described in Proposition \ref{prop:required summands}:
$M(\alpha)$ and $M(\zb\zg\zd)$ are both of type \ref{prop:required summands:itm1}, 
and $S(3),S(4),S(1)$ are of types \ref{prop:required summands:itm2}, \ref{prop:required summands:itm3}, \ref{prop:required summands:itm4}, respectively. See Figure~\ref{fig:triangulations-proj-inj-mars} for an illustration of these summands in the geometric model we use in this paper.
\end{example}

\begin{remark} 
In type $A$, each hook and each cohook gives rise to a required summand of any MAR module \cite[Lemma 3.9]{BGMS19}. This is not true for a gentle algebra in general. For example, consider the algebra given by the quiver 
\begin{center}
\begin{tikzpicture}
\node (1) at (-1,1) {$1$};
\node (2) at (-1,-1) {$2$};
\node (3) at (0,0) {$3$};
\node (4) at (1.2,0) {$4$};

\draw[->,shorten <=-2pt, shorten >=-2pt] (1) -- (3) node[below, pos=0.4] {$\za$};
\draw[->,shorten <=-2pt, shorten >=-2pt] (2) -- (3) node[below, pos=0.6] {$\zb$};
\draw[->,shorten <=-2pt, shorten >=-2pt] (3) -- (4) node[below, pos=0.6] {$\zg$};

\draw[densely dashed] (-0.4,0.4) to [out=40,in=100]  (0.5,0);
\end{tikzpicture}
\end{center}
with $I=\langle \za \zg \rangle$. 
Then the cohook of $\za$ is $\zb$, but $M(\zb)$ is not a required summand, since $\zb$ is not a trivial string nor a maximal direct string.

However, if a string $w$ over a gentle algebra given by $(Q,I)$ is both a hook \emph{and} a cohook, then $M(w) \in \required{Q,I}$. Indeed, if $w$ is a trivial string, then $w=1_x$, where $x$ has indegree $1$ and outdegree $1$, and so $M = S(i)$ is a required summand. Otherwise, $w$ is right maximal (respectively left maximal) because $w$ is a hook (respectively cohook), and so $w$ is a maximal direct non-trivial string. 
\end{remark}

\subsection{Existence of MAR modules}
\label{sec:Mproj} 

In this section, we give a general construction for MAR modules over an arbitrary gentle algebra. 
In particular, we establish that the sum of all indecomposable projective (respectively, injective) $A$-modules can be completed to a maximal almost rigid module in a unique way. Throughout, $M_{\proj}$ will denote the basic sum of the union of the projectives $P(i)$, their radicals $\rad P(i)$, and the required summands $\required{Q,I}$.
Dually, $M_{\inj}$ will denote be the basic sum of the union of the injectives $I(i)$, their socle quotients $I(i)/S(i)$, and the required summands $\required{Q,I}$.

First, we need two preliminary lemmas.

\def\mystringmodule{N}
\begin{lemma} 
\label{lem:no overlap extensions of the following forms}
There are no overlap extensions of the following forms.
\begin{enumerate}[(1)]
\item \label{enum:ses X to E to P splits}
$0 \to \mystringmodule \to E \to P \to 0$, where $\mystringmodule$ is any indecomposable module and $P$ is an indecomposable projective module

\item \label{enum:X to E to R}
$0 \to \mystringmodule \to E \to R \to 0$, where $\mystringmodule$ is any string module and $R$ is an indecomposable summand of the radical of a projective module
\end{enumerate}
\end{lemma}
\begin{proof}
Since $\Ext^1 (P, -)=0$ for every projective module $P$,  there is no overlap extension of the form~\ref{enum:ses X to E to P splits}.

We now prove that there is no overlap extension of the form \ref{enum:X to E to R}. 
Suppose, for the sake of contradiction, that there is an overlap extension 
$0 \to \mystringmodule \to E \to R \to 0$ with $\mystringmodule$ a string module and $R=M(v)$ an indecomposable summand of the radical of a projective module. 
Then there are two cases: either $v$ is a nontrivial direct string 
which is right maximal, or $v =1_x$ is a trivial string such that 
the outdegree of $x$ is at most one. 

We use the notation of Definition~\ref{def:overlap_extension}. 
By Remark \ref{rem:first string w is an up-set}, the overlap $e$ is a down-set of $v$ and is therefore a suffix of $v$. 
In both cases, $d=0$, so 
condition \ref{enum:b nonzero or d nonzero} of Definition~\ref{def:overlap_extension} implies that  $0 \neq b$, and in particular $vb$ is a string. 

If we are in the former case where $v$ is a nontrivial right maximal direct string, then this contradicts the fact that $v$ is right maximal. 

If we are in the latter case, then $v = 1_x=e$ and so $c=0$.  
Then condition \ref{enum:a nonzero or c nonzero} of Definition~\ref{def:overlap_extension} implies that $a\neq 0$ as well. Thus $a^{-1}b$ is a string, which means that $a\neq b$, and $s(a) = x= s(b)$. However, this contradicts the fact that $\outdegree{x}\leq 1$. 
This concludes the proof that there is no overlap extension of the form \ref{enum:X to E to R}. 
\end{proof}

\begin{lemma}\label{lem:Simple module in Mproj}
Let $i\in Q_0$.
Then $S(i)$ is a summand of $M_\proj$ 
if and only if 
$\outdegree{i}<2$.  
In particular, we have the following cases.
\begin{enumerate}
\item If $\outdegree{i}=0$, then $S(i)=P(i)$.
\item Let $\outdegree{i}=1$. 
\begin{enumerate}
\item If $\indegree{i}\leq 1$, then $S(i)$ is in $\required{Q,I}$. 
\item If $\indegree{i}=2$, then $S(i)$ is a summand of the radical of an indecomposable projective module.
\end{enumerate}
\end{enumerate}
\end{lemma}
\begin{proof}
Suppose $S(i)$ is a summand of $M_\proj$. Then $S(i)$ is a projective module, a summand of the radical $\rad P(j)$ of an indecomposable projective module $P(j)$, or it lies in $\required{Q,I}$. 
If $S(i)$ is projective, then $\outdegree{i}=0$. If $S(i)$ is a summand of $\rad P(j)$, for some $j \in Q_0$, then $\outdegree{i} \leq 1$. If $S(i)$ is in $\required{Q,I}$, then $\outdegree{i} \leq 1$ by Corollary \ref{cor:required summands}. Therefore, if $S(i)$ is a summand of $M_\proj$ then $\outdegree{i}<2$. 

Conversely, suppose $\outdegree{i}<2$, and we will show that $S(i)$ is a summand of $M_\proj$. 
We break this down into cases.
\begin{enumerate}
\item 
If $\outdegree{i}=0$, then $S(i)=P(i)$ is a summand of $M_\proj$.

\item 
Suppose $\outdegree{i}=1$. 
\begin{enumerate} 
\item 
If $\indegree{i} \leq 1$, then  $S(i)$ is in $\required{Q,I}$ by Corollary~\ref{cor:required summands}.
\item 
Suppose $\indegree{i}=2$. 

Let $a\colon j \to i$ and $b\colon j' \to i$ be the two arrows with target $i$ and let $c$ denote the arrow which starts at $i$. Since $A$ is gentle, exactly one of $ac$ and $bc$ is a relation - let us assume it is the former, as illustrated in Figure~\ref{fig:Simple module in Mproj}, where the picture on the right illustrates the case where $i = j = t(c)$, or in other words $a=c$ is a loop. Then $S(i)$ is a summand of the radical of the projective $P(j) = M(a)$.

\end{enumerate}
\end{enumerate} 
\end{proof}

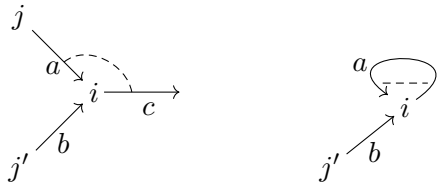
\begin{figure}[htbp!]
\begin{center}
\begin{tikzpicture}
\node (j) at (-1,1) {$j$};
\node (j') at (-1,-1) {$j'$};
\node (i) at (0,0) {$i$};
\node (ii) at (1.2,0) {};

\draw[->,shorten <=-2pt, shorten >=-2pt] (j) -- (i) node[below, pos=0.4] {$a$};
\draw[->,shorten <=-2pt, shorten >=-2pt] (j') -- (i) node[below, pos=0.6] {$b$};
\draw[->,shorten <=-2pt, shorten >=-2pt] (i) -- (ii) node[below, pos=0.6] {$c$};

\draw[densely dashed] (-0.4,0.4) to [out=40,in=100]  (0.5,0);
\end{tikzpicture}\qquad\qquad
\begin{tikzpicture}[xscale=1,yscale=0.8]
\node (a) at (-0.6,0.7) {$a$};
\node (j') at (-1,-1) {$j'$};
\node (i) at (0,0) {$i$};
\node (ii) at (1.2,0) {};

\draw[->] (i) edge[out=40,in=140,loop,shorten <=-2pt, shorten >=-2pt] ();

\draw[->,shorten <=-2pt, shorten >=-2pt] (j') -- (i) node[below, pos=0.6] {$b$};

\draw[densely dashed] (-0.3,0.4) to (0.3,0.4);
\end{tikzpicture}
\caption{Proof of Lemma \ref{lem:Simple module in Mproj}: case where $\indegree{i}=2$ and $\outdegree{i}=1.$}
\label{fig:Simple module in Mproj}
\end{center}
\end{figure}

\begin{theorem}
\label{thm:Mproj is unique mar containing proj}
\begin{enumerate}[(a)]
\item \label{thm:Mproj is unique mar containing proj:itm1}
$M_{{\proj}}$ is maximal almost rigid, and it   is the unique maximal almost rigid module that contains all indecomposable projectives as direct summands.

\item 
Similarly, $M_{\inj}$ is maximal almost rigid, and it   is the unique maximal almost rigid module that contains all indecomposable injectives as direct summands.
\end{enumerate}
\end{theorem}

\begin{proof}
We will prove part \ref{thm:Mproj is unique mar containing proj:itm1} of the theorem; the proof of the second part is similar.

We know from the definition of required summands (Section \ref{sec:required summands}) and Proposition~\ref{prop can complete ar to mar} that there are no overlap extensions of the forms $0 \to X \to E \to Y \to 0$ or $0 \to Y \to E \to X \to 0$, where $Y$ is a required summand and $X$ is a string module. 

It follows from Lemma \ref{lem:no overlap extensions of the following forms} that there are no overlap extensions of the form
$0 \to X \to E \to Y \to 0$, where each of $X$ and $Y$ is an indecomposable summand of $M_\proj$ which is projective or a summand of the radical of a projective module. 
This concludes the proof that $M_\proj$ is almost rigid. 

\smallskip 

Let $M_\proj'$ denote the basic sum of $ \{ P(i) \mid i \in Q_0 \} \cup \required{Q,I}$, and let $N$ be a string module which is not a summand of $M_\proj$. 
We will now prove that $M_\proj' \oplus N$ is not almost rigid.

Case 1: $N$ is a simple module $S(j)$. Since $N$ is not a summand of $M_\proj$, the outdegree of $j$ is $2$ by Lemma \ref{lem:Simple module in Mproj}. Then there is an overlap extension of $N$ by $P(j)$ whose overlap is the trivial string $1_j$. 

Case 2: $N=M(w)$ where $w$ is a non-trivial string. Without loss of generality, assume that the last letter $c$ of $w$ is an arrow. 

Case 2.1: There is an arrow $b$ such that $wb$ is a string. 
Then there is an overlap extension of $N$ by $P(t(c))$ whose overlap is the trivial string $1_{t(c)}$. 

Case 2.2:  There is no arrow $b$ such that $wb$ is a string. 
We claim $w$ is not direct. Indeed,
if $w$ is direct, then it is right maximal, by assumption. If $w$ is also left maximal, then $w$ is a maximal direct string, which means that $N$ is a required summand; 
if $w$ is not left maximal, then there is an arrow $a$ such that $aw$ is a string, implying that $N$ is a summand of the radical of the projective  $P(s(a))$. Either way, $w$ being direct implies that $N$ is a summand of $M_\proj$, so $w$ is not direct. 

There are two possibilities: either $w$ has at least one valley, or $w$ has no valley and it has exactly one hill.

Case 2.2.1: Suppose $w$ has a valley $cd^{-1}$ for some $c,d \in Q_1$.  
Then there is an overlap extension of $N$ by $P(t(c))$ whose overlap is the trivial string $1_{t(c)}$.

Case 2.2.2: 
Suppose $w$ has no valley and it has exactly one hill. 
Then $w$ is of the form $w=u v$ where $u$ is inverse and $v$ is direct.
If there is $b\in Q_1$ such that $u^{-1}b$ is a string, then $w^{-1}$ belongs to Case 2.1. 
If there is no $b\in Q_1$ such that $u^{-1}b$ is a string, then $N=P(s(v))$ and so $N$ is a summand of $M_\proj$.
This concludes the proof that $M_\proj' \oplus N$ is not almost rigid.

Since $M_\proj' \oplus N$ being not almost rigid implies that $M_\proj \oplus N$ is not almost rigid, we conclude that $M_\proj$ is indeed a maximal almost rigid module. 

Furthermore, any almost rigid module containing $M_\proj'$ as a summand cannot contain another indecomposable summand $N$ not in $\{ \text{summand of $\rad P(x)$} \mid x \in Q_0\} \cup \{P(x) \mid x \in Q_0\} \cup \required{Q,I}$. Therefore, $M_\proj$ is the unique MAR module containing all indecomposable projectives.
\end{proof}

We refer the reader to Section~\ref{sect examples} for examples of $M_\proj$ and $M_\inj$; see also Figures~\ref{fig:triangulations-proj-inj-mars} and~\ref{fig orpheus} for a geometric interpretation of these MAR modules, which is explained in Section~\ref{sec:geometric model}.

\begin{proposition}\label{prop projective MAR summands}
The number of indecomposable summands in the module $M_\proj$ 
is $|Q_0| + |Q_1|$.
\end{proposition}

\begin{proof}
We define a bijective map $f$ from $Q_0 \cup Q_1$
to the indecomposable summands of $M_\proj$ as follows. 

If $a \in Q_1$, 
let $f(a)$ be the string module $M(w)$ where $w$ is the right maximal string starting at $a$. 
Note that $f$ is well defined because $A$ is gentle. 
Furthermore, $M(w)$ is an indecomposable summand of $M_\proj$, as we now explain. 
Given an arrow $a \in Q_1$, there are two possibilities: either there are no direct strings $ba$, or there is a direct string $ba$.
If there are no direct strings $ba$, then $w$ is a nontrivial maximal direct string, so $M(w)$ is in $\required{Q,I}$ by Proposition \ref{prop:required summands}.
Otherwise, if there is a direct string $ba$ with $b \neq a$, then $w$ is a nontrivial direct string which is a summand of $\rad P(s(b))$. Either way, $f(a)=M(w)$ is an indecomposable summand of $M_\proj$.

Suppose $i\in Q_0$. 
If $\outdegree{i}=2$, let $f(i)$ be the projective $P(i)$ at $i$. 
Then $f(i)$ is an indecomposable summand of $M_\proj$. 
Here we note that $P(i)$ is not a direct string, so no arrow is sent to $P(i)$ by $f$.

If $\outdegree{i} < 2$, we let $f(i)$ be the simple module $S(i)$ at $i$. 
Since $\outdegree{i}<2$, the simple module $S(i)$ is a summand of $M_\proj$ by Lemma \ref{lem:Simple module in Mproj}.  
We again note that no arrow is sent to $S(i)$ by~$f$. 
This concludes the argument that $f$ is an injective map from $Q_0\cup Q_1$ to the indecomposable summands of $M_\proj$. 

\def\mysummand{M}

We now show that $f$ is a surjection. Suppose $\mysummand$ is an indecomposable summand of $M_\proj$.

First, suppose $\mysummand = S(i)$ is simple at vertex $i$. Then by Lemma~\ref{lem:Simple module in Mproj}, $\outdegree{i} < 2$ and so $\mysummand = f(i)$. From now on assume $\mysummand$ is not simple. 

Suppose $\mysummand$ is a projective $P(i)$. Since $\mysummand$ is not simple, we only have to consider the cases illustrated in Figure~\ref{fig:Projectives}. If $P(i)$ corresponds to a direct string, then there is an arrow $a$ such that $f(a)=P(i)$.
If $P(i)$ does not correspond to a direct string, then $\outdegree{i}=2$, and we have $f(i)=P(i)$. 

Now, suppose $\mysummand$ is a summand of the radical of a projective. Then, since $\mysummand$ is not simple, its corresponding string is direct and its first letter $a$ is such that $f(a) = M$; see Figure~\ref{fig:A summand of the radical of a projective}.

Finally, suppose $\mysummand \in \required{Q,I}$. It follows from Proposition~\ref{prop:required summands} and the fact that $\mysummand$ is not simple, that the string $w$ corresponding to $\mysummand$ is direct of length $\geq 1$. Then $M= f(a)$, where $a$ is the first letter of $w$. 
\end{proof}

\begin{figure}
\begin{tikzpicture}[xscale=1,yscale=0.8]
\node (i) at (0,3) {$i$};
\node (j) at (1,2) {};
\node (k) at (2,1) {};
\node (end) at (3,0) {$.$};

\draw[->,shorten <=-2pt, shorten >=-2pt] (i) -- (j) node[pos=0.3,right] {$a$};
\draw[densely dashed, shorten <=-2pt, shorten >=-2pt] (j) -- (k);
\draw[->,shorten <=-2pt, shorten >=-2pt] (k) -- (end);
\end{tikzpicture}
\qquad
\begin{tikzpicture}[xscale=1,yscale=0.8]
\node (i) at (0,3) {$i$};
\node (jleft) at (-1,2) {};
\node (j) at (1,2) {};
\node (k) at (2,1) {};
\node (kleft) at (-2,1) {};
\node (end) at (3,0) {$.$};

\draw[->,shorten <=-2pt, shorten >=-2pt] (i) -- (j);
\draw[->,shorten <=-2pt, shorten >=-2pt] (i) -- (jleft);
\draw[densely dashed, shorten <=-2pt, shorten >=-2pt] (j) -- (k);
\draw[densely dashed, shorten <=-2pt, shorten >=-2pt] (jleft) -- (kleft);
\draw[->,shorten <=-2pt, shorten >=-2pt] (k) -- (end);
\end{tikzpicture}
\caption{Proof of Proposition~\ref{prop projective MAR summands}, case where $M =P(i)$. Left: $P(i)$ corresponds to a direct string. Right: $P(i)$ does not correspond to a direct string.}
\label{fig:Projectives}

\begin{tikzpicture}[xscale=1,yscale=0.8]
\node (j) at (0,3) {$j$};
\node (i) at (1,2) {$i$};
\node (k) at (2,1) {};
\node (end) at (3,0) {$.$};

\draw[->,shorten <=-2pt, shorten >=-2pt] (j) -- (i);
\draw[->,shorten <=-2pt, shorten >=-2pt] (i) -- (k)  node[pos=0.5,right] {$a$};
\draw[densely dashed,shorten <=-2pt, shorten >=-2pt] (k) -- (end);

\draw[densely dotted,rotate=-45] (k) ellipse (16mm and 6mm);
\end{tikzpicture}
\caption{Proof of Proposition~\ref{prop projective MAR summands}, case where $M$ is a summand of the radical of a projective $P(j)$.}
\label{fig:A summand of the radical of a projective}
\end{figure}


\section{Geometric model for gentle algebra} \label{sec:geometric model}
In this section, we set up our notation for a geometric model for the module category of a gentle algebra $A$, following~\cite{BCS21,Cha23,BCS24}.

\subsection{Tiled  surfaces}
\label{sec:tiled surfaces}

Let $S$ be a connected oriented Riemann surface without punctures with non-empty boundary. 
A \emph{boundary segment} in a marked surface $(S,N)$ is a segment of the boundary of $S$ between two neighboring points of $N$.

To the surface $S$ we add two disjoint finite sets of marked points $M$ and $M^*$ according to the following conditions:
\begin{itemize} 
\item 
Each point $p\in M$ belongs to the boundary $\partial S$ of $S$, and every boundary component  of $S$ has at least one point of $M$.
\item The points $p^*\in M^*$ may be on the boundary or in the interior of $S$ such that every boundary segment in $(S,M)$ has exactly one point of $M^*$. 
\end{itemize} 

There is no condition on the points of $M^*$ that lie in the interior of $S$.
We sometimes refer to the points in $M$ as black marked points and to the points in $M^*$ as red marked points.
To differentiate between the marked points in $M$ and $M^*$, we will illustrate a marked point in $M^*$ using a star \textcolor{red}{ $\star$}.
The points $p^*\in M^*$ which are in the interior of $S$ are called \emph{punctures}.

We thus obtain three marked surfaces $(S,M),(S,M^*)$ and $(S,M,M^*)=(S,M\cup M^*)$.

By isotopy we always mean isotopy in the marked surface  $(S,M, M^*)$. In particular an isotopy may not move a curve over a point of $M^*$.
A \emph{crossing} of two curves will always mean a transversal crossing in the interior of $S$. The \emph{intersection} of two curves is the set-theoretic intersection, hence it  can contain crossings and common endpoints.

In order to define a tiling of  $(S,M, M^*)$ we need the notion of a $P$-arc.

\begin{definition}
    A \emph{$P$-arc} in $(S,M,M^*)$ is the isotopy class of a curve $\tau$ in $S$ satisfying the following conditions:
\begin{itemize}
    \item The endpoints of $\tau$ are in $M$;
    \item $\tau$ intersects the boundary of $S$ only in its endpoints;
    \item $\tau$ has no self-crossings (the endpoints of $\tau$ may coincide);
    \item $\tau $ does not go through a point of $M^*$;
    \item $\tau$ is not isotopic to a point.
\end{itemize}
\end{definition}

Let us emphasize that by our notion of isotopy the two arcs shown in Figure~\ref{fig:bigon-monogon} are $P$-arcs.
\begin{figure}
    \centering
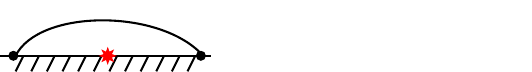
    \caption{Two examples of $P$-arcs.}
    \label{fig:bigon-monogon}
\end{figure}
 
\begin{definition}\label{def:tiling}
A \emph{tiling} of $(S,M,M^*)$ 
is a set $P$ of pairwise noncrossing $P$-arcs that subdivides the surface into polygons and once-punctured polygons, called \emph{tiles}, such that each tile contains exactly one marked point of $M^*$. 

The quadruple $\TiledSurface$ is called a \emph{tiled surface}. 
\end{definition}

We call a tile \emph{internal} if each of its sides is a $P$-arc, and \emph{external} otherwise. 
By definition each tile $\zD$ is of one of the following two types. 

\begin{enumerate}
[label=(\roman*)]
\item If $\Delta$ is a once-punctured polygon, then $\Delta$ is an internal $m$-gon with $m\ge 1$ containing a unique red point $p^*\in M^*$ in its interior.
\item  If $\Delta$ is a polygon, then $\Delta$ is an external $m$-gon where $m\ge 2$,  with exactly one boundary edge (and a unique red marked point on this edge); and $\Delta$ has no red points or boundary components in its interior. 
\end{enumerate}
Examples of tiles are shown in Figure~\ref{fig:tiles}.
\begin{figure}
\centering
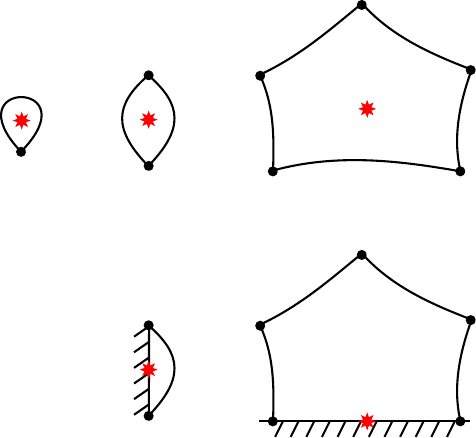
\caption{Examples of internal tiles (top) and external tiles (bottom) of a tiled surface $\TiledSurface$}
\label{fig:tiles}
\end{figure}
Note that the external tiles shown in Figure \ref{fig:tiles} (bottom) are considered to be a $2$-gon and a $5$-gon, respectively. 
Also note that a tile may be a self-folded internal or external $m$-gon (necessarily $m \geq 4$); for example, Figure \ref{fig ex 7.5} (top left) contains a self-folded external $5$-gon.

The next lemma follows immediately from the description of the tiles given above. 
\begin{lemma}\label{lem:each black point has an arc}
Let $\TiledSurface$ be a tiled surface and $x\in M$. Then there is an arc in $P$ with endpoint $x.$
\end{lemma}

To every tiled surface, we will associate an algebra. To do so, we need the notions of complete fans and $P$-angles.

\begin{definition} 
Let $\TiledSurface$ be a tiled surface and $x\in M$.
\begin{enumerate}[label=(\roman*)]
    \item  The \emph{complete fan at $x$} is the sequence $\tau_1,\tau_2,\ldots,\tau_k$ of $P$-arcs in $P$ that are crossed when traveling around $x$ in clockwise direction. It is possible that $\tau_i=\tau_j$ with $i\ne j$. An example is shown in Figure~\ref{fig:fan}.
\item A \emph{$P$-angle at $x$} is an ordered triple $(x,\tau_i,\tau_{i+1})$, where $\tau_i,\tau_{i+1}$ are consecutive $P$-arcs in the complete fan  at $x$. 
\end{enumerate}
\end{definition}
\begin{figure}
    \centering   
\small\scalebox{1.2}{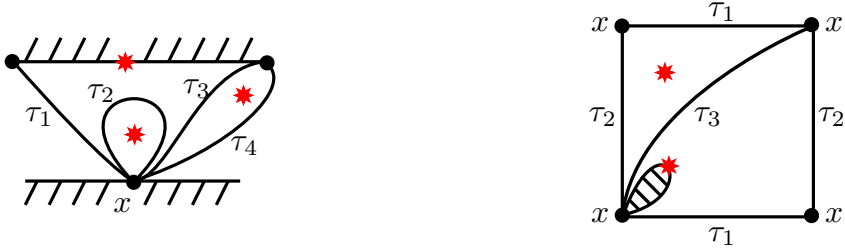}
\caption{Two examples of complete fans. In the left picture the complete fan at $x$ is $\tau_1,\tau_2,\tau_2,\tau_3,\tau_4$. Note that $\tau_2$ appears twice. The right picture is on the torus with one boundary component. The complete fan at $x$ is $\tau_1,\tau_2,\tau_3,\tau_1,\tau_2,\tau_3$.}
    \label{fig:fan}
\end{figure}
We sometimes consider the complete fan at $x$ as a sequence of $P$-angles at $x$.

\begin{remark}
    \label{rem:angle tile}
    Since the $P$-angle $(x,\tau_i,\tau_{i+1})$ is ordered, it lies inside a unique tile $\zD$ in $P$. We say $(x,\tau_i,\tau_{i+1})$ is a $P$-angle in the tile $\zD.$
\end{remark}

\begin{definition}\label{def:tiling algebra}
Let $\TiledSurface$ be a tiled surface. Its \emph{tiling algebra} $A_P$ is the bound quiver algebra $\kb Q_P/I_P$
defined as follows.

The vertices in $(Q_P)_0$ are in bijection with $P$ 
and the arrows
$\za: v\to v'$ in $(Q_P)_1$ correspond bijectively to the $P$-angles $(x, \tau_v, \tau_{v'})$ in $P$.

The ideal $I_P$ is the two-sided ideal generated by all paths  
$\za\zb$, with  arrows $\za\colon v\to v', \zb\colon v'\to v''$ and corresponding  $P$-angles $\za=(x,\tau_v, \tau_{v'})$ and $\zb=(y,\tau_{v'}, \tau_{v''})$,  that satisfy one of the following conditions:
\begin{enumerate}[label={(\roman*)}]
\item 
$\za=\zb$, (then $\za $ is a loop, see Figure \ref{fig:relations:a equals b}), or
\item 
$\za\neq \zb$ and the corresponding sequence of $P$-angles is not a contiguous subsequence of a complete fan (see  Figures~\ref{fig:relations:case 2}
and \ref{fig:relations:case 3}).
\end{enumerate}
\end{definition}

\begin{figure}
\begin{center}
\begin{tikzpicture}[xscale=1,yscale=1]
 \draw[thick, ] (0,0) .. controls (2.5,1.7) and (-2.5,1.7) .. (0,0) ; 
 \node at (0,.7) [red] {\ding{88}};
 \node at (-0.6,1.45) {$\tau_v=\tau_{v'}=\tau_{v''}$};
\node at (0,0) [fill,circle,inner sep=1.1] {};
\node at (0,-0.3) {$x=y$};
\end{tikzpicture}
\begin{tikzpicture}
\node at (-2,-1) {}; 
\node (i) at (-1,0) {$v$};
\node (j) at (0,0) {$v$};
\node (k) at (1.2,0) {$v$};

\draw[thick, ->,shorten <=-2pt, shorten >=-2pt] (i) -- (j) node[below, pos=0.6] {$\za$};
\draw[thick, ->,shorten <=-2pt, shorten >=-2pt] (j) -- (k) node[below, pos=0.6] {$\za$};

\draw[densely dashed] (-0.5,0) to [out=90,in=90]  (0.5,0);
\end{tikzpicture}
\caption{Local configuration of a tiling and the path $\za \za \in I_P$, when $\za=\zb$.}
\label{fig:relations:a equals b}
\begin{tikzpicture}
\draw[thick, ] (0,0) -- (-1,1.5) node[right, pos=0.5] {$\tau_v$};
\draw[thick, ] (0,0) -- (2,0) node[above, pos=0.5] {$\tau_{v'}$};
\draw[thick, ] (2,0) -- (3,1.5) node[left, pos=0.5] {$\tau_{v''}$};
\node at (0,0) [fill,circle,inner sep=1.1] {};
\node at (-0.25,-0.25) [] {$x$};
\node at (2,0) [fill,circle,inner sep=1.1] {};
\node at (2+.25,-0.25) [] {$y$};
\end{tikzpicture}
\begin{tikzpicture}
\node at (0,0) [fill,circle,inner sep=1.1] {};
    \node at (-0.25,-0.25) [] {$x$};
    \node at (2,0) [fill,circle,inner sep=1.1] {};
    \node at (2+.25,-0.25) [] {$y$};

\node at (1,0.05) [red] {\ding{88}};

\draw[thick, ] (0,0) .. controls (1,0.5) .. (2,0) node[midway, above] {$\tau_{v}=\tau_{v''}$};
    \draw[thick, ] (2,0) .. controls (1,-.5) .. (0,0) node[midway, below] {$\tau_{v'}$};
\end{tikzpicture}
\begin{tikzpicture}
\node at (-1,-1) {}; 
\node (i) at (-1,0) {$v$};
\node (j) at (0,0) {$v'$};
\node (k) at (1.2,0) {$v''$};

\draw[thick, ->,shorten <=-2pt, shorten >=-2pt] (i) -- (j) node[below, pos=0.6] {$\za$};
\draw[thick, ->,shorten <=-2pt, shorten >=-2pt] (j) -- (k) node[below, pos=0.6] {$\zb$};

\draw[densely dashed] (-0.5,0) to [out=90,in=90]  (0.5,0);
\end{tikzpicture}
\caption{Local configuration of a tiling when $\za\neq \zb$, and  
the points $x$ and $y$ are distinct; the $P$-arcs $\tau_v$ and $\tau_{v''}$ may be distinct or equal. Right: the path $\za\zb \in I_P$ }\label{fig:relations:case 2}
\begin{tikzpicture}
\node at (0,0) [fill,circle,inner sep=1.1] {};
\node at (-0.25,-0.25) [] {$x = y$};
\draw[thick, ] (0,0) .. controls (0.9,1) and (-0.9,1) .. (0,0) node[above, pos=0.25] {$\tau_{v'}$};

\draw[thick, ] (0,0) -- (-1.3,1.5) node[left, pos=0.5] {$\tau_v$};
\draw[thick, ] (0,0) -- (1.3,1.5) node[right, pos=0.5] {$\tau_{v''}$};
\end{tikzpicture}
\begin{tikzpicture}
\node at (0,0) [fill,circle,inner sep=1.1] {};
\node at (-0.25,-0.25) [] {$x = y$};
    
\draw[thick, ] (0,0) -- (.5,.9) node[above, pos=0.9] {$\tau_{v}$};
\draw[thick, ] (0,0) -- (-.5,.9) node[above, pos=0.9] {$\tau_{v''}$};

\draw[thick, ] (0,0) .. controls (4,2.5) and (-4,2) .. (0,0) node[above, pos=0.4] {$\tau_{v'}$};
    
\end{tikzpicture}
\begin{tikzpicture}
\node at (0,0) [fill,circle,inner sep=1.1] {};
    \node at (-0.25,-0.25) [] {$x = y$};

\draw[thick, ] (0,0) .. controls (2.5,1.7) and (-2.5,1.7) .. (0,0) node[above=-1pt, pos=0.5] {$\tau_v=\tau_{v''}$};
    
\draw[thick, ] (0,0) .. controls (0.9,1) and (-0.9,1) .. (0,0) node[above, pos=0.25] {$\tau_{v'}$};
    
\end{tikzpicture}
\caption{Local configuration of a tiling and the path $\za\zb \in I_P$, when $\za\neq \zb$, and the points $x$ and $y$ are equal; the $P$-arcs $\tau_v$ and $\tau_{v''}$ may be distinct or equal. }
\label{fig:relations:case 3}
\end{center}
\end{figure}

The following  is a combination  of Theorems~2.9 and 2.10 from \cite{BCS21}.

\begin{theorem}\label{thm:tiling algebra iff gentle}
The tiling algebra $A_P$ of a tiled surface $\TiledSurface$ is  gentle, and every gentle algebra arises in this way.
\end{theorem}

Examples are given in section~\ref{sect examples}.

\subsection{Permissible arcs}\label{sec:permissible arcs}

We give an overview of the correspondence between string modules and certain arcs called permissible arcs, following
\cite[Section 3.1]{BCS21} with slight modifications, as given in~\cite{Cha23}.  We will need another notion of arcs in the surface, different from the $P$-arcs in Section~\ref{sec:tiled surfaces}

\begin{definition}
[arcs]
An \emph{arc} of $\TiledSurface$ is a curve $\gamma$ in $S$, considered up to isotopy of immersed arcs, such that:
\begin{enumerate} 
\item the endpoints of $\gamma$ are in $M^*$;
\item except for the endpoints, $\gamma$ is disjoint from $M^*$ and from the boundary of $S$;
\item $\gamma$ does not cut out an unpunctured monogon;
\item $\gamma$ does not contain any contractible kinks.
\end{enumerate}
\end{definition}

Note that we allow an arc to cross itself a finite number of times. Further note that  whenever an arc intersects a $P$-arc, the intersection is a transversal crossing.

Each indecomposable string module can be mapped to an arc which is unique up to isotopy. 
However, not every arc of $(S,M^*)$ corresponds to an indecomposable module. 
Those which correspond to indecomposable modules are called ``permissible", and are defined below. We note that these arcs are called ``zigzag arcs'' in \cite{Cha23}.

\begin{definition}[A variation of {\cite[Definition 3.1]{BCS21}}]\label{def:permissible arc}
Let $\gamma$ be an arc of $\TiledSurface$, and assume that $\gamma$ is chosen such that the number of crossings with $P$ is minimal.
Choose an orientation of $\gamma$, and 
let $p_1, p_2, \ldots, p_{\stringlength+1}$ be the  sequence of crossing points between $\gamma$ and $P$ in order along $\gamma$. Also let $p_0$ and $p_{\stringlength+2}$ denote the starting and ending points of $\zg$. 
Let $\tau_i$ be the $P$-arc in $P$ containing $p_i$; note that in general it is possible to have $i\neq j$ and $\tau_i=\tau_j$.

We subdivide $\gamma$ into a sequence of segments $\gamma^0, \gamma^1, \ldots, \gamma^{\stringlength+1}$, where $\gamma^i$ is the segment of $\gamma$ between 
$p_i$ and $p_{i+1}$, where $0\le i\le \stringlength+1$. See Figure~\ref{fig:arc-segments}. 
The segments $\gamma^1,\dots,\gamma^{\stringlength}$ are called the \emph{interior segments}, and the segments $\gamma^0$ and $\gamma^{\stringlength+1}$ are called the \emph{end segments}.

An interior segment $\gamma^i$ is called a \emph{permissible segment} if it satisfies the following:
\begin{itemize}
\item The pair of consecutive $P$-arcs $\tau_i$ and $\tau_{i+1}$ share an endpoint $x_i\in M$; and 
\item 
$\tau_i, \tau_{i+1}$, and $\gamma^{i}$ bound a simply-connected triangle 
such that the angle opposite of $\gamma^i$ is a $P$-angle $(x,\tau_i,\tau_{i+1})$ or $(x,\tau_{i+1},\tau_{i})$, and we write $\angle\gamma^i$ to denote this $P$-angle.
    \end{itemize}
Otherwise, $\gamma^i$ is called \emph{non-permissible}.
A permissible segment $\gamma^i$ is illustrated in Figure~\ref{fig:permissible-segment}. 
When $\gamma^i$ is a permissible segment, we  
say that $\gamma^i$ (or $\gamma$) traverses the $P$-angle $\angle\gamma^i$.

We say that $\gamma$ is a \emph{permissible arc} provided that each of its interior segments $\gamma^i$ is permissible for $i=1,2,\ldots \stringlength$.
\end{definition}

\begin{remark}
Two curves (with the minimum number of crossings with $P$) are representatives of the same permissible arc if and only if 
their endpoints in $M^*$ are equal, and they traverse the same (ordered) sequence of $P$-angles.

An arc of $\TiledSurface$ which crosses $P$ exactly once is permissible.
\label{rem:an arc which crosses P exactly once in permissible}
\end{remark}

\begin{figure}[htbp]
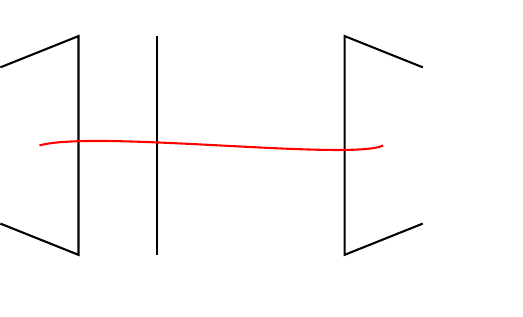
\caption{Arc segments.}.\label{fig:arc-segments}
\end{figure}

\begin{figure}[htbp]
\def\svgwidth{4cm}
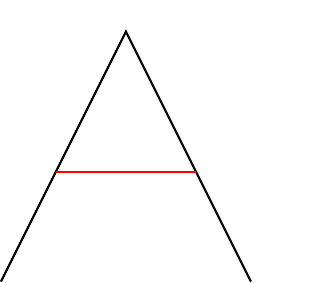
\caption{A permissible segment corresponding to the $P$-angle $\angle\gamma^i$ which is equal to $(x_i,\tau_{i+1},\tau_i)$.}
\label{fig:permissible-segment} 
\end{figure}

The following theorem follows immediately from \cite[Theorem 3.8]{BCS21}.

\begin{theorem}[A variation of {\cite[Theorem 3.8]{BCS21}}]\label{thm:permissible}
Let $A$ be the tiling algebra of a tiled surface $\TiledSurface$. Then there is a bijection between isotopy classes of permissible arcs  
and isomorphism classes of string $A$-modules.
\end{theorem}

The correspondence between oriented permissible arcs $\gamma$ and strings $w(\gamma)$ is described in the following two definitions; following the proof of \cite[Theorem 3.8]{BCS21}.

\begin{definition}\label{def:string-from-arc}
[String $w(\gamma)$ corresponding to a permissible arc $\gamma$] \label{def:string associated to oriented permissible arc}
Let $\gamma$ be an orientation of a permissible arc, chosen so that its crossings with $P$ are minimal, and let $\gamma^0, \gamma^1, \dots, \gamma^{\stringlength+1}$ be the sequence of segments defined above. 

If $\stringlength=0$ (that is, $\gamma$ crosses $P$ exactly once), then we associate to $\gamma$ the trivial string $w(\gamma)=1_v$, where the vertex $v\in Q_0$ corresponds to the $P$-arc crossed by $\gamma$.

Otherwise, $\stringlength\geq 1$, 
and each of the interior segments $\gamma^1, \gamma^2, \dots, \gamma^{\stringlength}$ corresponds to an arrow or inverse arrow $\za_i$ in $Q_1 \cup Q_1^{-1}$.  
Then we associate to $\gamma$ the string $w(\gamma) = \za_1 \cdots \za_\stringlength$.
\end{definition}

\begin{definition}[Permissible arc $\gamma_w$ corresponding to a string]
Suppose $w=\za_1 \dots \za_\stringlength$ is a nonzero and nontrivial string. 
Let $v_1=s(\za_1)$ and $v_i=t(\za_{i-1})$, for $i=2,\ldots,\stringlength+1.$ 
By Definition \ref{def:tiling algebra}, the vertices in $Q_0$ correspond bijectively to the $P$-arcs in $P$, so each of $v_1, \dots, v_{\stringlength+1}$ corresponds to a $P$-arc (possibly $v_i$ and $v_j$ correspond to the same $P$-arc for some $i\neq j$). By abuse of notation, we will use $v_i$ to refer to the $P$-arc corresponding to the vertex $v_i$.

Since there is an arrow or inverse arrow $\za_1$ from $P$-arc $v_1$ to $P$-arc $v_2$, there is a $P$-angle $(p,v_1, v_2)$ or $(p,v_2, v_1)$; let $\Delta_1$ denote the tile containing this $P$-angle. We can choose a point $p_1$ in the interior of the $P$-arc $v_1$ and a point $p_2$ in the interior of the $P$-arc $v_2$ and connect them by a curve $\gamma^1$ in the tile $\Delta_1$ in such a way that $v_1$, $v_2$, and $\gamma^1$ bound a simply-connected triangle, that is, $\gamma^1$ is a permissible segment (defined in Definition \ref{def:permissible arc}). 
Proceed in this way with the remaining $P$-arcs $v_2,\dots, v_{\stringlength}$ to obtain curves $\gamma^2, \dots,\gamma^\stringlength$ in the tiles $\Delta_2,\dots,\Delta_\stringlength$. 

The $P$-arc $v_1$ is a side of two tiles (which may coincide if they are a self-folded tile): the tile $\Delta_1$ above and a tile $\Delta_0$. By definition of tiles, there is exactly one red point $p_0^*\in M^*$ in $\Delta_0$.  Let $\gamma^0$ be a curve in the interior of $\Delta_0$ which connects $p_1$ and $p_0^*$.
On the right end of the string, we can similarly define a curve $\gamma^{\stringlength+1}$ in  the tile  $\Delta_{\stringlength+1}$ whose endpoint is the unique red point $p_{\stringlength+2}^*\in M^*$ of the tile. 

The curve $\gamma_w$ is defined to be the concatenation of the curves $\gamma^0,\dots,\gamma^{\stringlength+1}$. Similarly, if $w = 1_v$ is a trivial string, then the corresponding arc $\gamma_w$ is a curve with no self-crossings which crosses only  $v$, and whose endpoints are the red points of the tiles for which $v$ is a side.
\end{definition}


\subsection{Required permissible arcs}

We now describe the permissible arcs which correspond to the required summands of any MAR module (see Proposition \ref{prop:required summands}). First, let us distinguish two types of permissible arcs.

\begin{definition} Let $\zg$ be  a permissible arc of the tiled surface $\TiledSurface$.
\begin{enumerate}
\item 
We say that $\zg$ is a \emph{permissible boundary segment} if $\gamma$ is a boundary segment of  the marked  surface $(S,M^*)$. 
\item 
We say that $\zg$  is a \emph{bigon connector} if $\gamma$ 
satisfies the following conditions: 
\begin{itemize}
\item $\gamma$ crosses $P$ at exactly one point;
\item one of the endpoints of $\gamma$ lies in an external bigon tile, and
\item $\gamma$ is not a boundary segment of $(S,M^*)$.
\end{itemize}
\end{enumerate}
\end{definition}

See Figure~\ref{fig:connector} for an illustration of a bigon connector. See Example \ref{ex:required permissible arcs} for examples of permissible boundary segments and a bigon connector.  

\begin{figure}
\centering
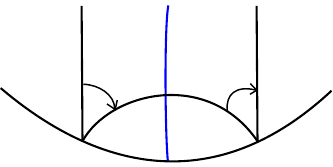
\caption{A bigon connector. Note that the arrows $\alpha$ and $\beta$ may coincide. }
\label{fig:connector}
\end{figure}

\begin{theorem}\label{thm:required-model}  
The required summands of any maximal almost rigid module are precisely those corresponding to the permissible boundary segments and the bigon connectors.
\end{theorem}
\begin{proof}
In each of the following four cases,  there is a bijection  between the  set of elements in (a) and the set of elements in (b). 

\begin{enumerate}
\item 

\begin{enumerate}
    \item maximal direct strings in $A$ containing at least one arrow.
    \item complete fans with at least one $P$-angle.
\end{enumerate}

\item 
\begin{enumerate}
    \item vertices $i$ of degree $1$ in $Q$.
    \item $P$-arcs $\tau_i$ in $P$  with endpoints $x,y$ such that $\tau_i$ is the only element in the complete fan at $x$, and $\tau_i$ is the first or the last element in the complete fan at $y$.
\end{enumerate}

\item 
\begin{enumerate}
    \item vertices $i$ of degree $2$ in $Q$ such that there is a direct string $\za\zb$ 
with $i=t(\za)=s(\zb)$.
    \item  $P$-arcs $\tau_i$ in $P$  with endpoints $x,y$ such that $\tau_i$ is the only element in the complete fan at $x$, and $\tau_i$ is not the first nor the last element in the complete fan at $y$.
\end{enumerate}

\item 
\begin{enumerate}
    \item vertices $i$ of degree $2$ in $Q$ such that there is a path $\za\zb \in I$ with $i=t(\za)=s(\zb)$.
    \item $P$-arcs $\tau_i$ that are crossed by a bigon connector
\end{enumerate}

\end{enumerate}
Indeed, the existence of the bijection in cases (1)--(3) follows directly from the definition of the tiling algebra; we point out that \cite{Schroll} gave a bijection between the set $1(a) \cup 2(a) \cup 3(a)$ and the complete fans in the Brauer graph associated to $A$.  
We will prove the  case (4) below. This will complete the proof of the theorem, because, on the one hand, the cases (1)--(4)(a) are precisely those of Proposition~\ref{prop:required summands}, and on the other hand, the cases (1)--(3)(b) form a partition of the set of permissible boundary segments, and (4b)  is in bijection with the set of bigon connectors.

It remains to show that there is a bijection between (4a) and (4b).
Assume first that $\za=\zb$. Then $\za$ is a loop in $Q$ and the $P$-arc $\tau_i$ is a loop in the surface. 
The condition that the degree of $i$ is 2 implies that $\za$ is the only arrow at vertex $i$. Since $Q$ is connected, it follows that $Q$ has only one vertex $i$ and only one arrow $\za$. The corresponding tiled surface is a disk with one black point and one red point on the boundary, and one red point in the interior, see Figure~\ref{fig:one loop}. There are precisely two permissible arcs, the permissible boundary segment $\zd$ and the arc $\zg$ connecting the two red points. The latter is a bigon connector.

Now assume that $\za\ne \zb$. The condition that the path $\za\zb$ lies in the ideal $I$ implies that the corresponding $P$-angles are not consecutive angles of a complete fan. Since the degree of $i$ in $Q$ is 2, it follows that the two $P$-angles are at two distinct points $x,y\in M$, and the $P$-arc $\tau_i$ connects $x$ and $y.$  Consider the two tiles $\zD_1,\zD_2$ of $P$ incident to $\tau_i$ and let $z_1^*,z_2^*$ denote their respective red points. Again using the condition that $i$ has degree 2, we see that one of the tiles, say $\zD_1$, must by an external bigon tile containing the boundary points $x,y$ and $z_1^*$. The other tile can be interior or not. In either case, there is a unique permissible arc $\zg_i$ connecting $z_1^*$ and $z_2^*$. This arc $\zg_i$ is a bigon connector. This shows that there is an injective map $i\mapsto \zg_i$ from the set (4a) to the set (4b).

Conversely, since each bigon connector $\zg$ crosses $P$ at exactly one point, there is a unique $\tau_i$ in $P$ that is crossed by $\zg$. Let $x,y\in M$ denote the endpoints of $\tau_i$. 
We first consider the case where $x=y$, and $\tau_i$ is a loop.
Let $\Delta_1$ be the internal tile bounded by $\tau_i$ containing one endpoint of $\gamma$.
Since $\gamma$ is a bigon connector, it travels through precisely one other tile $\Delta_2$, which is an external bigon.
Thus, $\Delta_2$ is bounded by  $\tau_i$ and all of $\partial S$.
In particular, $M=\{x\}$.

Now we break into cases depending on the number $k$ of distinct $P$-arcs in $P$ incident to $x$.
The case where $k=1$ is shown on the left in Figure~\ref{fig:one loop}.
In this case, $Q$ consists of a single vertex $i$ with one arrow, which is a loop.
Thus, $\gamma$ is the unique bigon connector, and it is mapped to the unique vertex $i$ with degree $2$ and path $\alpha\beta\in I$ (where $\alpha=\beta)$.

Suppose that $k>1$. An example is shown on the right in Figure~\ref{fig:one loop}.
As above, since $M=\{x\}$, each $P$-arc is a loop with endpoint $x$.
Note that in the complete fan at $x$, $\tau_i$ is the first arc in clockwise order (because $\tau_i$ is the unique arc that is a boundary segment for an external tile).
For convenience, we rename it $\tau_1$.
In the complete fan at $x$, following $\tau_1$ in clockwise order, we have $\tau_2, \tau_3, \ldots, \tau_{2k}$, where $\tau_i$ and $\tau_j$ are distinct for $i,j\in [k]$, and $\tau_{2k-i+1} = \tau_i$, for $i\in[k]$.
(In particular, $\tau_{k+1}= \tau_k$ and $\tau_{2k}=\tau_1$.)

We write $\alpha_i$ for the arrow corresponding to the $P$-angle $(x, \tau_i, \tau_{i+1})$ where $i\in [k]$.
Observe that only $\alpha_k$ is a loop, and $\alpha_i$ for $i\in [k-1]$ has distinct vertices in $Q$.
Finally, we write $\bar{\alpha_i}$ for the arrow corresponding to the $P$-angle $(x, \tau_{2k-i},\tau_{2k-i+1})$ for $i\in[k-1]$.
Since $x$ is the unique element of $S$, this is a complete description of $Q$.
Following Definition~\ref{def:tiling algebra}, the ideal $I$ is generated by $\alpha_i\bar{\alpha_i}$, $\bar{\alpha_i}\alpha_i$ and $(\alpha_k)^2$.
We observe that precisely one vertex of $Q$ has degree 2, and it is the vertex $1$ corresponding to $\tau_1$.
Moreover, $t(\bar{\alpha_1})=1=s(\alpha_1)$, and $\bar{\alpha_1}\alpha_1 \in I$ as desired.
This completes the proof in the case where $x=y$.

Now we assume $x\ne y$, and let recall that we let $\tau_i$ be the unique $P$-arc in $P$ crossed by $\gamma$.
Since one of the endpoints of $\gamma$ is in an external tile, we have that the vertex $i\in Q$ has degree at most 2. 
Moreover if the degree of $i$ is strictly less than 2, then the complete fan at $x$ or $y$ contains only $\tau_i$.
This implies that both tiles incident to $\tau_i$ contain a boundary segment, contradicting the definition of a bigon connector.

Thus $i$ has degree 2 and each point $x,y$ has a non-trivial complete fan. Thus there exist $P$-angles $\za=(x,\tau_h,\tau_i)$ and $\zb=(y,\tau_i,\tau_j)$ which give rise to a path $\xymatrix{h\ar[r]^\za&i\ar[r]^\zb&j}$. Since $x\ne y$ this path is in the ideal $I$. 
This shows that the vertex $i$ is an element of the set (4a), and we have an injective map $\zg\mapsto i$ from the set (4b) to the set (4a). 
\begin{figure}
    \centering
   \scriptsize 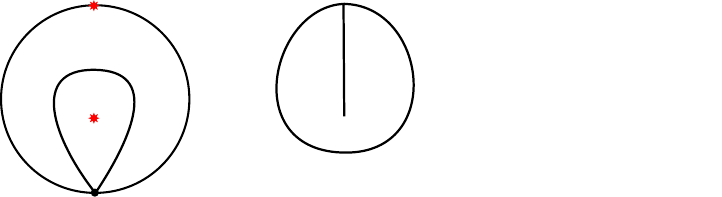
\caption{Left: The case where $Q$ has only one vertex $i$ and one loop arrow $\za\colon i\to i$, 
    and $I$ is generated by $\za^2$. The first picture shows the tiled surface and the second picture shows the two permissible arcs $\zd$ and $\zg$. The arc $\zg$ is a bigon connector. Right: The case when $k=3$. Here the quiver has three vertices $1,2,3$ and five arrows $1\to 2, \,1\ot 2,\, 2\to 3, \, 2\ot 3 $ and a loop at 3. }
    \label{fig:one loop}
\end{figure}

\end{proof}

\begin{example}\label{ex:required permissible arcs}
Consider the tiled annulus of Figure \ref{fig ex 7.5} (top left).
There are four permissible boundary segments (three on the outer boundary and one on the inner boundary), 
and there is one bigon connector, the permissible arc between red points $y_3, y_4 \in M^*$ which cross the $P$-arc $\tau_1$; these five permissible arcs are labeled $\zg_1, \dots, \zg_5$ in Figure \ref{fig ex 7.5} (top right). In addition, Figure~\ref{fig:triangulations-proj-inj-mars} illustrates these five permissible arcs in red.

These five permissible arcs correspond to the five required summands listed in Example \ref{ex:required summands:as modules}:
\begin{enumerate}
\item $M(\za)$ corresponds to the outer boundary segment ($\zg_2$) between $y_2$ and $y_3$; $M(\zb \zg \zd)$ corresponds to the outer boundary segment ($\zg_5$) between $y_1$ and $y_3$.
\item $S(3)$ corresponds to the outer boundary segment ($\zg_1$) between $y_1$ and $y_2$
\item $S(4)$ corresponds to the sole inner boundary segment ($\zg_3$)
\item $S(1)$ corresponds to the bigon connector ($\zg_4$) between $y_3$ and $y_4$.
\end{enumerate}

\end{example}

\subsection{Extensions via crossings and intersections}
We close this section  recalling  two results that describe morphisms and overlap extensions combinatorially in terms of intersections of permissible arcs.

Let $\alpha$ and $\beta$ be two distinct permissible arcs.
We say that $\alpha$ and $\beta$ have a \emph{nontrivial intersection} if each isotopy representative of $\alpha$ and $\beta$ intersect (either in their interior or at shared endpoints).
Throughout we choose representatives of $\alpha$ and $\beta$ with the fewest possible intersections.

\begin{definition}\label{def:M_angle}
Let $\alpha$ and $\beta$ be distinct permissible arcs, with nontrivial intersection $x=\alpha(t)=\beta(t)$ for some $t\in[0,1]$.
\begin{enumerate}[(i)]
\item 
If $x$ is an endpoint, choose an orientation of $\alpha$ and $\beta$ so that  then $\alpha(1) = x = \beta(1)$. 
\item 
If $x$ is in the interior of $\alpha$ and $\beta$, choose any orientations of $\alpha$ and $\beta$ (note that there are four choices).
\end{enumerate}
Let $\alpha^\circ, \beta^\circ$ denote the restriction of $\alpha$ and $\beta$ to the interval $[t-\epsilon, t]$ for $\epsilon>0$ sufficiently small.
The \emph{angle} is the ordered triple $(x,\za^\circ,\zb^\circ)$ such that the direction of the angle from $\za^\circ $ to $ \zb^\circ$ at $x$ 
 is counterclockwise.

Given an orientation $\za$ of an arc, let $\overline{\za}$ denote the opposite orientation of the arc.
We define an angle $(x,\za^\circ,\zb^\circ)$ 
to be equivalent to itself and to the ``opposite" angle $(x,\overline{\za}^\circ,\overline{\zb}^\circ)$. 
This means that, if $x$ is a crossing in the interior of the arcs, there are two angles up to equivalence.
\end{definition}
\begin{definition}\label{def:permissible-angle} 
A $P$-arc $\tau$ of the tiling $P$ is said to \emph{traverse the angle} $(x,\za^\circ,\zb^\circ)$ if $\tau$ crosses $\za$ and $\zb$ in such a way that the walk along $x,\za,\tau,\zb,x$ encloses a simply connected region in $S$. 
We say that the angle $(x,\za^\circ,\zb^\circ)$ is \emph{permissible} if it is traversed by a $P$-arc of $P$. 
\end{definition}

An illustration of the angles of a crossing
is given in Figure~\ref{fig anglecrossing}. The orientation of the arc $\za$ in the left picture is opposite to its orientation on the right. This leads to two different angles $(x,\za^\circ,\zb^\circ)$, $(x,\zb^\circ,\overline{\za}^\circ)$, which are indicated by curved blue arrows. 
The angle in the left picture is permissible and the one in the right picture isn't. 
The case where the arcs intersect at the endpoints  is shown in Figure~\ref{fig angles 1 later}.
\begin{figure}[ht!]
    \centering
    \Large
\scalebox{.7}{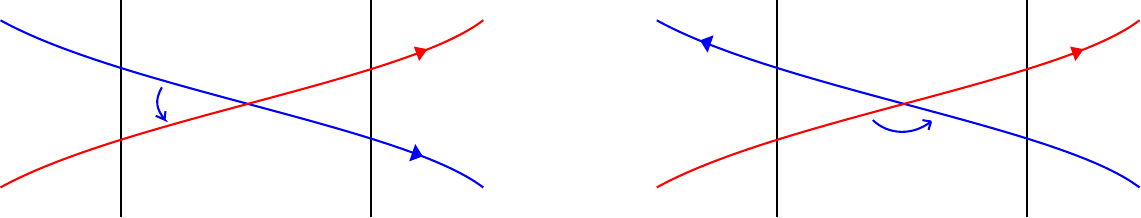}
    \caption{Two examples of angles  at a crossing point $x$ of two permissible arcs $\za,\zb$ which illustrate the dependence on the chosen orientation. The angle on the left is $(x,\za^\circ,\zb^\circ)$ and the one on the right is $(x,\zb^\circ,\overline{\za}^\circ)$, where the bar on $\overline{\za}$ indicates the change of orientation. }
    \label{fig anglecrossing}
\end{figure}

The following result concerns morphisms between string modules. This has been stated in \cite[Proposition 3.21]{BCS21},   
with different terminology. Subsequently, this result has been generalized to encompass all higher extensions in \cite[Theorem 2.28]{Cha23}, and to all string algebras in \cite[Theorem 4.11]{BCS24}.

\begin{theorem}
[Morphisms and intersections] \label{thm hom basis} 
Let $\gamma_1$ and $\gamma_2$ be distinct permissible arcs, and let $M_1$ and $M_2$ be the corresponding indecomposable $A$-modules. 
Let $\basis(\gamma_1, \gamma_2)$ be the basis of $\Hom_A(M_1, M_2$ as described in Proposition~\ref{prop: string_hom_basis}.
Let $\mathcal{I}(\gamma_1, \gamma_2)$ be the set of equivalence classes of permissible angles $(x, \gamma_1^\circ, \gamma_2^\circ)$ of $\gamma_1$ and~$\gamma_2$.
There is a bijective map $\phi: \mathcal{I}(\gamma_1,\gamma_2) \to \basis(\gamma_1, \gamma_2)$.
\end{theorem}

The following result concerns overlap extensions and is a special case of \cite[Theorem 2.28]{Cha23}. 
For the reader's convenience, we include a proof here. We note that our proof relies on string combinatorics whereas the proof in \cite{Cha23} is via the bounded derived category of the algebra.

\begin{proposition}[Overlap extensions and crossings]\label{prop:overlap-cross}
Let $M_1$ and $M_2$ be two (possibly isomorphic) string modules. 
Denote by $\gamma_1$ and $\gamma_2$ the permissible arcs corresponding to $M_1$ and $M_2$, respectively. The following statements are equivalent:
\begin{enumerate}[(1)]
\item \label{prop:overlap-cross:itm1} There is an overlap extension between $M_1$ and $M_2$.
\item \label{prop:overlap-cross:itm2} $\gamma_1$ and $\gamma_2$ cross each other in their interior.
\end{enumerate}

Moreover, the overlap extensions of $M_2$ by $M_1$ are in bijection with the equivalence classes of permissible angles  of the form $(x, \gamma_1^\circ, \gamma_2^\circ)$, where $x = \zg_1 (t) = \zg_2 (t)$, for some $t \in (0,1)$.
\end{proposition}

\def\myPangle{\mathbf{a}}

\begin{proof}
First, we show  \ref{prop:overlap-cross:itm1} implies \ref{prop:overlap-cross:itm2}. 
Let $w_1$ and $w_2$ be two strings with an overlap extension of $M_1$ by $M_2$. 
By Definition~\ref{def:overlap_extension}, there are $a,b,c,d\in Q_1\cup\{0\}$, such that 
$w_1=w_{1_\LEFT} a^{-1} e b w_{1_\RIGHT}$ and $w_2= w_{2_\LEFT} c e d^{-1} w_{2_\RIGHT}$, where at least one of $a$ and $c$ (respectively $b$ and $d$) is nonzero. 

We want to show that the corresponding arcs $\gamma_{w_1}, \gamma_{w_2}$ cross in the interior of $S\setminus M^*$.

As in earlier proofs, we abuse notation here and write $a$ when we mean both the arrow in $Q_1$ and the correspond angle in the surface $(S, M, M^*, P)$.
Let $\Delta_0$ (respectively $\Delta_t$) denote the tile containing $a$ and $c$ (respectively $b$ and $d$).
Let $\Delta_1, \ldots, \Delta_{t-1}$ be the tiles containing each arrow in the string $e$, in order.

We can draw the segments of $\gamma_{w_1}$ and $\gamma_{w_2}$ in $\Delta_0, \ldots, \Delta_{t-1}$ in such a way that there are no crossings, since the strings of $w_1$ and $w_2$ coincide on $e$.
Now, suppose without loss of generality, that $d\neq 0$. This means that in tile $\Delta_t$, the arc $\gamma_{w_2}$ crosses the $P$-angle associated to the arrow $d$. 
The arc $\gamma_{w_1}$ either finishes at the red point of 
$\Delta_t$ or it crosses the $P$-angle corresponding to $b$, which is distinct from the {$P$-}angle corresponding to $d$. 
Either way, the segments of $\gamma_{w_1}$ and $\gamma_{w_2}$ in $\Delta_t$ must cross 
(see Figure \ref{fig:prop:overlap-cross:overlapimpliescross}), and thus \ref{prop:overlap-cross:itm2} holds.

Next, we show \ref{prop:overlap-cross:itm2} implies \ref{prop:overlap-cross:itm1}. 
Suppose $\gamma_1$ and $\gamma_2$ cross at a point $x$.
Without loss of generality, we may assume that $x$ lies in the interior of a tile which we  denote by $\Delta$. Denote by $(\gamma_1)_\Delta$ and $(\gamma_2)_\Delta$ the respective segments of $\gamma_1$ and   $\gamma_2$ in $\Delta$ which include the point $x$. Since $x \not\in M^*$, we see that at least one of $(\gamma_1)_\Delta$ and $(\gamma_2)_\Delta$ is not an end segment.

Suppose, without loss of generality, that $(\gamma_1)_\Delta$ is not an
end segment. 
Then, this segment corresponds to an $P$-angle $\myPangle$ 
at $\Delta$.

Since the segment $\gamma_1$ is permissible, the red point of $\Delta$ does not lie in the triangle formed by $(\gamma_1)_\Delta$ and the angle $\myPangle$. 
Hence, since $(\gamma_1)_\Delta$ and $(\gamma_2)_\Delta$ cross each other, it follows that $\gamma_1$ and $\gamma_2$ have a common intersection with $P$, say at $P$-arc $\rho$ 
(see Figure \ref{fig:prop:overlap-cross:crossing_c}).

This $P$-arc $\rho$ is one of the $P$-arcs in the maximal sequence $\mathcal{E}= (\rho_1, \ldots, \rho_r)$ of consecutive $P$-arcs of $P$ crossed by $\gamma_1$ and $\gamma_2$. 
Let $w_1$ and $w_2$ be the strings corresponding to $\zg_1,\zg_2$, respectively, and let $e$ the string corresponding to $\mathcal{E}$. 

By maximality of $\mathcal{E}$, it follows that one of $w_1$ and $w_2$ must be of the form 
$p_\LEFT aeb^{-1}p_\RIGHT$, and the other must be of the form 
$q_\LEFT c^{-1}edq_\RIGHT$, where 
$p_\LEFT, p_\RIGHT, q_\LEFT, q_\RIGHT$ are strings such that one of 
$p_\LEFT a$ and 
$q_\LEFT c^{-1}$ (respectively 
$b^{-1} p_\RIGHT$ and 
$dq_\RIGHT$) are nonzero. 
Indeed, if this were not the case, then we would be able to undo the crossing at $x$ by considering isotopic arcs. 
(For example, $a$ is zero precisely when $\gamma_v$ and $\gamma_e$ share a common endpoint at the tile containing $\rho_1$. See Figure \ref{fig:prop:overlap-cross:case_all_nonempty} for an illustration of the case where 
$p_\LEFT a, q_\LEFT c^{-1}, b^{-1} p_\RIGHT$ and $dq_\RIGHT$ are all nonzero. )

Thus it follows that there is an overlap extension between the string modules $M(w_1)$ and $M(w_2)$, as required. The last statement of the proposition
follows from above. 
\end{proof}

\begin{figure}[htbp!]
\small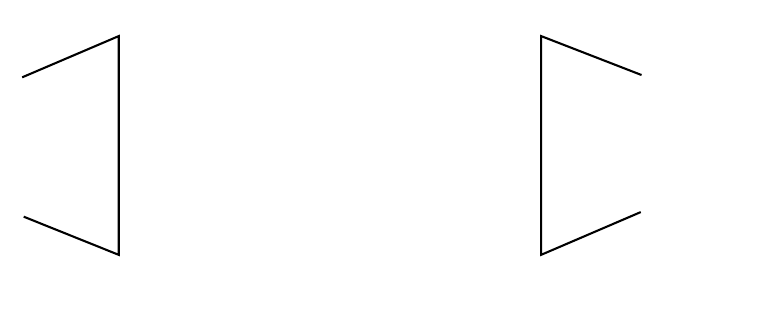
\caption{Proof of Proposition \ref{prop:overlap-cross}, \ref{prop:overlap-cross:itm1} implies \ref{prop:overlap-cross:itm2}: the segments of $\gamma_{w_1}$ and $\gamma_{w_2}$ must cross in $\Delta_t$.}
\label{fig:prop:overlap-cross:overlapimpliescross}
 
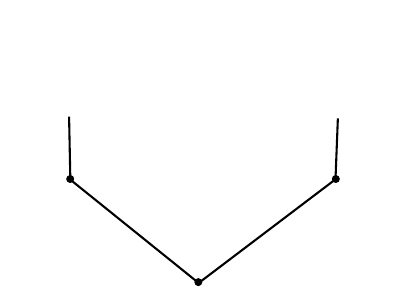
\caption{Proof of Proposition \ref{prop:overlap-cross},~\ref{prop:overlap-cross:itm2} implies~\ref{prop:overlap-cross:itm1}: the arcs $\zg_1$ and $\zg_2$ cross in the interior of $\zD$ and both cross the $P$-arc $\rho$ of $P$.}\label{fig:prop:overlap-cross:crossing_c}

\end{figure}

\begin{figure}
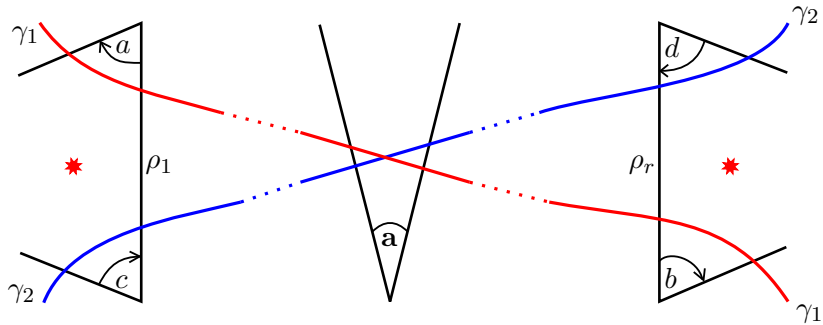
\caption{Proof of Proposition~\ref{prop:overlap-cross},~\ref{prop:overlap-cross:itm2} implies~\ref{prop:overlap-cross:itm1}: case where all $a,b,c,d$ are nonzero.}\label{fig:prop:overlap-cross:case_all_nonempty}
\end{figure}

Arrow extensions between string modules can also be described in terms of intersections of permissible arcs, as the following remark describes. For more details, see \cite[Theorem 2.28]{Cha23}.

\begin{remark}[Arrow extensions and certain non-permissible intersections at endpoints]\label{rem:arrow extensions correspond to certain intersections at endpoints}
Let $M_1, M_2$ be two string modules, possibly isomorphic, and let  $\gamma_1, \gamma_2$ denote the permissible arcs corresponding to $M_1$ and $M_2$, respectively. The following conditions are equivalent.
\begin{enumerate}
\item There is an arrow extension of $M_2$ by $M_1$.
\item $\gamma_1$ and $\gamma_2$ intersect at an endpoint $x\in M^*$ in such a way that there is a marked point $x' \in M$ and a quadrilateral as in Figure~\ref{fig:arrow-extension}, where $\tau_2$ (respectively $\tau_1$) is the first $P$-arc in $P$ crossed by the end segment $\gamma_2^\circ$ (respectively $\gamma_1^\circ$) at $x$. 
In other words, 
we have 
a (non-permissible) angle $(x,\gamma_2^\circ,\gamma_1^\circ)$ and a $P$-angle $(x',\tau_2,\tau_1)$ which form a quadrilateral as in Figure~\ref{fig:arrow-extension}. 
\end{enumerate}
\end{remark}

\begin{figure}[htbp!]
\centering
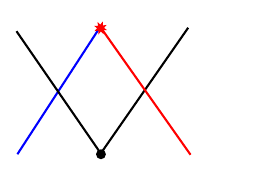
\caption{Intersection corresponding to an arrow extension  as described in Remark \ref{rem:arrow extensions correspond to certain intersections at endpoints}.} 
    \label{fig:arrow-extension}
\end{figure}

\section{Triangulations correspond to MAR modules}
\label{sect 5}

In this section we give a geometric characterization of maximal almost rigid modules over a gentle algebra. The first step for this characterization is  Proposition~\ref{prop:mar-maximal-noncrossing}, which is a consequence of the geometric interpretation of overlap  extensions (Proposition~\ref{prop:overlap-cross}), and the fact that almost rigidity is equivalent to non-existence of overlap extensions (Proposition~\ref{prop:all-exts}). 

\begin{proposition}\label{prop:mar-maximal-noncrossing}
Let $A$ be a gentle algebra and $\TiledSurface$ be the corresponding tiled surface. 
 The bijection from string modules to permissible arcs described in Section \ref{sec:permissible arcs} induces a bijection from the set of maximal almost rigid $A$-modules to the set of maximal collections of noncrossing permissible arcs of $\TiledSurface$.
\end{proposition}
\begin{proof}
Let $M(\gamma), M(\gamma')$ be string modules, with corresponding permissible arcs $\gamma$ and $\gamma'$, respectively. It follows from Proposition~\ref{prop:all-exts}, Proposition~\ref{prop:overlap-cross} and the fact that every overlap extension has non-indecomposable middle term, that $M(\gamma) \oplus M(\gamma')$ is almost rigid if and only if $\gamma$ and $\gamma'$ admit no crossings. 
Therefore, a module lies in $\mar{A}$ if and only if the corresponding set of permissible arcs admits no crossings and is maximal with respect to this property. 
\end{proof}

\begin{lemma}\label{lem:maximalistriangulation} 
A maximal set of noncrossing permissible arcs of $\TiledSurface$ is an ideal triangulation of the surface $(S,M^*)$.
\end{lemma}
\begin{proof}
Let $\calt$ be a maximal set of noncrossing permissible arcs and suppose it is not a triangulation. 
Then there is a non-permissible arc $\alpha$ such that $\calt \cup \{\alpha\}$ is noncrossing. Let $\alpha^i$ be an interior segment of $\alpha$ which is not permissible, and let $\zD$ be the tile containing $\alpha^i$. Let $p$ denote the red point of $\zD$.

Since $\calt \cup \{\alpha\}$ is noncrossing, $\alpha$ does not admit self-crossings. 
Hence $\alpha^i$ crosses through non-consecutive $P$-arcs forming the boundary of $\Delta$, as shown in Figure~\ref{fig:segment}.
The tile $\zD$ is subdivided into two parts $\zD',\zD''$ by $\za^i$, where we agree that the red point $p$ lies in $\zD''$. Since each tile of $\TiledSurface$ contains exactly one red point by Definition \ref{def:tiling}, the region $\zD'$ has no red point.

\begin{figure}[ht!]
 \centering 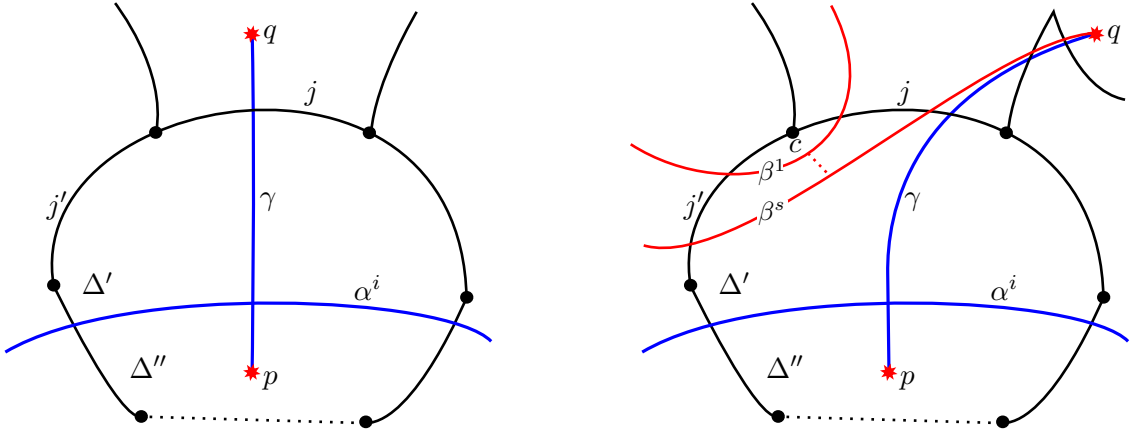
\caption{Non-permissible segment $\alpha^i$.}
\label{fig:segment}
\end{figure}

Since $\za^i$ is not permissible, the region $\zD'$  has at least two $P$-angles. 
So there is at least one $P$-arc $j$ such that $j$ is a side of $\zD'$ and $j$ does not cross $\za^i$.

Case 1: 
First, suppose there are no arcs in $\calt$ which cross the $P$-arc $j$ of $P$. This situation is illustrated in the left picture in Figure~\ref{fig:segment}. 
The $P$-arc $j$ is a side of another tile $\zG$ (which may be equal to $\zD$ if $\zD$ is a self-folded tile). Let $q$ denote the red point of the tile $\zG$, and  
let $\gamma$ denote the arc connecting $q$ with the red point $p$ of $\zD$ 
and such that it crosses $P$ only once and at arc $j$. By Remark \ref{rem:an arc which crosses P exactly once in permissible}, $\gamma$ is permissible. 
We know $\gamma$ does not belong to $\calt$, since  $\calt \cup \{\alpha\}$ is noncrossing and $\gamma$ crosses $\alpha$. 
Furthermore, $\gamma$ does not cross any arc in $\calt$, because a permissible arc that crosses $\gamma$ must also cross $j$, but by assumption no arcs in $\calt$ cross $j$. 
Therefore, $\calt \cup \{\gamma\}$ is a set of noncrossing permissible arcs, 
contradicting the fact that $\calt$ is a maximal set of noncrossing permissible arcs.

Case 2:
There are arcs in $\calt$ crossing the $P$-arc $j$.
This situation is illustrated in the right picture in Figure~\ref{fig:segment}. 
By~\cite[Proposition 2.10]{FST08}, there are finitely many arcs in a triangulation. Since $\calt$ is a subset of a triangulation, there are finitely many arcs in $\calt$, and so in particular there are finitely many arcs in $\calt$ which cross $j$.

We can assume without loss of generality that some of these arcs traverse
the $P$-angle $(c,j,j')$ at marked point $c \in M$ defined by $j$ and its predecessor $j'$ in the clockwise order around $\Delta'$. 
Let $\beta^1, \ldots, \beta^s$ denote the segments of arcs in $\calt$ which traverse the $P$-angle $(c,j,j')$, listed in increasing order 
with respect to their distance from the point $c$. 
Thus $\zb^s$ is the segment furthest away from the point~$c.$

Let $\zeta$ be the arc in $\mathcal{T}$ containing the segment $\zb^s$, and suppose it is oriented in such a way that the segment $\zb^s$ crosses $j'$ before $j$. Let $q$ denote the terminal point of $\zeta$. Consider the arc $\gamma$ whose initial endpoint is $p$ (the red point in $\zD$), whose interior segments coincide with the segments of $\zeta$ after $\zb^s$, and whose terminal point is $q$.
See the right picture in Figure~\ref{fig:segment}.

Since $\zeta$ is permissible, so is $\gamma$. Moreover, given the order defined on the set $\beta^1, \ldots, \beta^s$, we have that $\gamma$ does not cross any arc in $\calt$. Thus, $\calt \cup \{\gamma\}$ is a set of permissible arcs with no crossings, contradicting the fact that $\calt$ is a maximal set of noncrossing permissible arcs.
\end{proof}

\begin{definition}\label{def permissible triangulation}
A triangulation $\calt$ of $(S,M^*)$ is said to be a {\it permissible triangulation} if all the arcs in $\calt$ are permissible arcs of $\TiledSurface$.
\end{definition} 

The following result follows immediately from Proposition~\ref{prop:mar-maximal-noncrossing} and Lemma~\ref{lem:maximalistriangulation}. It generalizes \cite[Theorem 6.8]{BGMS19}, which considers the case where the algebra is of Dynkin type $\mathbb{A}$. Examples are given in section~\ref{sect examples}.

\begin{theorem}\label{thm:mar=permissible-triangulation}
The bijection from string modules to permissible arcs   described in Section \ref{sec:permissible arcs} induces a bijection from maximal almost rigid $A$-modules to permissible triangulations of $(S,M^*)$.
\end{theorem}

This characterization of maximal almost rigid modules allows us to deduce that the number of indecomposable summands is fixed for all maximal almost rigid modules. 

\begin{corollary}\label{cor MAR summands} 
The number of indecomposable direct summands in any MAR $A$-module is $|Q_0| + |Q_1|$.
\end{corollary}
\begin{proof}
The number of indecomposable summands of an MAR $A$-module is equal to the number of arcs in the permissible triangulation of $(S,M^*)$ corresponding to it, by Theorem \ref{thm:mar=permissible-triangulation}.
Since every triangulation of the same marked surface has the same number of arcs, by \cite[Proposition 2.10]{FST08}, it follows that every MAR $A$-module has the same number of indecomposable summands. 
Hence, it suffices to count the number of summands in one particular MAR module, say $M_\proj$. 
In Proposition \ref{prop projective MAR summands}, we proved that this number is $|Q_0| + |Q_1|$.
\end{proof}

\begin{question}
    Recall that the cardinality of $Q_0$ is equal to the number of summands of a (basic) tilting $A$-module. By Corollary~\ref{cor MAR summands}, the cardinality of $Q_0\cup Q_1$ is the number of summands of a maximal almost rigid module. If we let $S=\oplus_{i\in Q_0} S(i)$ be the sum of all simple modules, we have 
    $|Q_0|=\dim\,\Hom(S,S)$, and
    $|Q_1|=\dim\,\Ext^1(S,S)$. So it is natural to ask if there are other classes of modules, which one might call maximal $m$-almost rigid modules, that capture the dimension of $\cup_{i=0}^{m}\, \Ext^i(S,S). $ 
\end{question}

\section{Endomorphism algebras of maximal almost rigid modules}\label{sect 6}

Let $A=\kb Q/I$ be a gentle algebra.
This section is devoted to the study of the endomorphism algebra $C$ of an MAR $A$-module $T$. 
We will prove that $C$ can be realized as the endomorphism algebra of a tilting module over a different gentle algebra, denoted by $\Abar$ (see Corollary~\ref{cor:end-alg-tilt}). 
We will also see that $C$ is gentle of global dimension at most $2$ (Corollary~\ref{cor endoalg}), and we explicitly describe the corresponding tiled surface (Theorem~\ref{thm endoalg}).
Finally, we will describe the tensor algebra of $C$ via the triangulation corresponding to the MAR module (Theorem~\ref{thm:tensor}).

Throughout this section, $\TiledSurface$ denotes the tiled surface associated to $A$, and $T= T_1 \oplus \cdots \oplus T_m$ denotes a MAR $A$-module, with $m=|Q_0|+|Q_1|$. So, in particular, the $T_i$'s are pairwise non-isomorphic. The triangulation in $(S,M^*)$ associated to $T$ will be denoted by $\calt$.

\subsection{The gentle algebra \texorpdfstring{$\Abar$}{Abar}}\label{sect 7.1}

We define a new algebra $\Abar=\kb\Qbar/\Ibar$ from $A$ as follows. The quiver $\Qbar$ is obtained from $Q$ by replacing every arrow $\za \colon i\to j$ with a path of length two $\xymatrix{i\ar[r]^{\za_a}&v_\za\ar[r]^{\za_b}&j}$. More precisely, $\Qbar_0=Q_0\cup Q_1$, where we use the notation $v_\za$ for the vertex in $\Qbar$ corresponding to the arrow $\za$ in $Q$. Moreover $\Qbar_1=Q_1\times\{a,b\}$, where  $\za_a=(\za,a)$ is an arrow from $s(\za)$ to $v_\za$ and $\za_b=(\za,b)$ is an arrow from $v_\za$ to $t(\za)$. 

Since $A$ is gentle, the ideal $I$ is generated by paths of length two. For every path $\za\zb\in I$, we obtain a new path $\za_b\zb_a$ from $v_\za$ to $v_\zb$ in $\Qbar$. We define $\Ibar$ as the 2-sided ideal generated by all paths of the form $\za_b\zb_a$, with $\za\zb\in I$.

We illustrate the definition in an example.
\begin{example}
 \label{exAbar}
 Let $A=\kb Q/I$ with $Q$ the quiver 
 $\xymatrix{1\ar@<2pt>[r]^\za&2\ar@<2pt>[l]^\zb\ar@<2pt>[r]^\zg&3\ar@<2pt>[l]^\zd}$ and $I=\langle \za\zb,\zb\za,\zg\zd,\zd\zg\rangle$.
 Then $\Abar=\Qbar/\Ibar$, with $\Qbar $ the quiver
 $\xymatrix@R5pt{1\ar[r]^{\za_a}&v_\za\ar[r]^{\za_b}&
 2\ar[r]^{\zg_a}\ar[ld]^{\zb_a}&v_\zg\ar[r]^{\zg_b}&3\ar[ld]^{\zd_a}\\
 &v_\zb\ar[lu]^{\zb_b}&&v_\zd\ar[lu]^{\zd_b}}$
 and $\Ibar=\langle  \za_b\zb_a,\zb_b\za_a,\zg_b\zd_a,\zd_b\zg_a\rangle$.
\end{example}

\begin{proposition}\label{probAbargentle}
The algebra $\Abar$ is  gentle and of global dimension at most 2.
\end{proposition}
\begin{proof}
We need to check conditions \ref{def gentle:itm1:degrees}--\ref{def gentle:itm4:ideal} 
of Definition~\ref{def gentle}. By definition, $\Abar$ satisfies condition \ref{def gentle:itm4:ideal}.  
Moreover, the new vertices $v_\za$ are the source of exactly one arrow $\za_b$ and the target of exactly one arrow $\za_a$. The old vertices $Q_0\subset \Qbar_0$ have the same number of incoming and outgoing arrows as in $Q$. Thus $\Qbar $ satisfies \ref{def gentle:itm1:degrees}.
 
For every arrow $\za_a$, there is exactly one arrow $\za_b$ such that $\za_a\za_b$ is a path. This path does not lie in the ideal $\Ibar$. On the other hand, if $\zg\in\Qbar_1$ is an arrow such that $\zg\za_a$ is a path then $\zg=\zb_b$ for some arrow $\zb\in Q_1$.  Moreover, $\zb\za\in I$ if and only if $\zb_b\za_a\in \Ibar$. 
A similar argument holds for arrows of the form $\za_b$. This shows that $\Abar$ satisfies conditions \ref{def gentle:itm2:string algebra} and \ref{def gentle:itm3:gentle algebra}, and hence $\Abar$ is gentle. 

It follows directly from the definition of $\Ibar$ that there are no overlapping relations, that is, whenever $\za\zb, \zg\zd\in \Ibar$ then $\zb\ne \zg$. Therefore the global dimension of $\Abar $ is at most 2, see for example \cite{Bardzell}.
\end{proof}
\begin{proposition}\label{prop:projective Abar}
Let $\za\colon i\to j$ be an arrow in $Q$ and let $\Pbar(i)$ and $\Pbar(v_\za)$ denote the projective $\Abar$-modules at vertices $i$ and $v_\za$, respectively. Then we have the following conditions:
\begin{enumerate}[(1)]
    \item\label{prop:projective Abar:itm1}
$\Pbar(v_\za)$ is a direct summand of $\rad \Pbar(i)$.
\item\label{prop:projective Abar:itm2} The socle of $\Pbar(i)$ is a direct sum of simple modules of the form $\Sbar(\ell)$, with $\ell\in Q_0$.
\item\label{prop:projective Abar:itm3} The socle of $\Pbar(v_\za)$ is a simple module of the form $\Sbar(\ell)$, with $\ell\in Q_0$.
\item\label{prop:projective Abar:itm4} If $\overline{L}$ is a submodule of $\Pbar(i)$ such that the top of $\overline{L}$ is of the form $\overline{S}(v_\zb)$ 
for some $\zb\in Q_1$, then  $\overline{L}=\Pbar(v_\zb).$ 
\end{enumerate}
\end{proposition}
\begin{proof} Recall that the relations in $\Ibar$ are paths of the form $\xymatrix{v_\zb\ar[r]^{\zb_b}&\ell\ar[r]^{\zg_a}&v_\zg}$, for $\zb\zg\in I$. Therefore, 
statement \ref{prop:projective Abar:itm1} holds, because the arrow $\za_a\colon i\to v_\za$ in $\Qbar$ is not the initial arrow of a relation in $\Ibar.$ Parts \ref{prop:projective Abar:itm2} and \ref{prop:projective Abar:itm3} hold, because the relations end at a vertex of the form $v_\zg$. Moreover, the socle is simple in part \ref{prop:projective Abar:itm3}, because $\za_b$ is the  unique arrow starting at $v_\za$.

It remains to show part \ref{prop:projective Abar:itm4}.
Let $v, w$ and $w'$ be strings such that $\Lbar = M(v), \Pbar (i) = M(w)$ and $M(w')$ is a direct summand of the radical of $\Pbar (i)$. By assumption, $v$ is a down-set subwalk of $w$. 
Indeed, $v$ is a down-set subwalk of $w'$, since $v_\zb \neq i$. Note that $w'$ is a right maximal direct string (up to inversion). Hence, so is $v$. If $v$ is trivial, then $\Lbar = \Sbar(v_\zb)$ and it is a direct summand of the socle of $\Pbar(i)$. But by \ref{prop:projective Abar:itm2} this would mean that $v_\zb \in Q_0$. By this contradiction, we must have that $v$ is a non-trivial right maximal string. It follows that $\Lbar = \Pbar (v_\zb)$ since these modules have the same top, $\Sbar (v_\zb)$, and there is a unique arrow, $\zb_b$, starting at $v_\zb$.
\end{proof}

\begin{example}
In Example~\ref{exAbar}, the indecomposable projective $\Abar$-modules
are
\[ \begin{smallmatrix}
        1\\v_\za\\2\\v_\zg\\3
    \end{smallmatrix},\ 
     \begin{smallmatrix}
        v_\za\\2\\v_\zg\\3
    \end{smallmatrix},\ 
     \begin{smallmatrix}
    2\\v_\zb\ v_\zg\\1\ \ 3
    \end{smallmatrix},\ 
     \begin{smallmatrix}
    v_\zb\\ 1
    \end{smallmatrix},\ 
     \begin{smallmatrix}
    v_\zg\\ 3
    \end{smallmatrix},\ 
     \begin{smallmatrix}
    3\\v_\zd\\2\\v_\zb\\1
    \end{smallmatrix} \text{ and }\ 
     \begin{smallmatrix}
    v_\zd\\2\\v_\zb\\1
    \end{smallmatrix}.
\] 
\end{example}
\subsection{The tiling \texorpdfstring{$G(S,M,M^*,P)$}{G(S,M,M*,P)}} 
\label{sec:tiling G(S,M,P)}

From the tiled surface $\TiledSurface$, we construct a new tiled surface $G\TiledSurface=(G(S),G(M),G(M^*),G(P))$, given as follows: 

\begin{enumerate}[(i)]
\item 
$G(S)$ is the surface obtained from $S$ by replacing every puncture with a boundary component.
\item  
$G(M)$ is obtained from $M$ by adding one black marked point $x(\za)$ for every arrow $\za$ in $Q$. 
To be more precise, an arrow $\za\colon i\to j$ in $Q$ corresponds to a $P$-angle  $(x,\tau_i,\tau_j)$ in the tile $\zD$.

If $\zD$ has an edge on the boundary of $S$, the new marked point $x(\za)$ is placed on that edge. If $\zD$ does not have a boundary edge then the surface $G(M)$ has a new boundary component in the interior of $\zD$. In this case, the new marked point $x(\za)$ is placed on this boundary component.

If the $P$-angles corresponding to several arrows $\za_1,\za_2,\ldots,\za_t$ lie in the same tile $\zD$, then we arrange the new marked points $x(\za_1),x(\za_2),\ldots,x(\za_t)$ in the same order along the boundary.

\item 
$G(P)$ is obtained from $P$ by adding one arc $\tau(\za)$ for every arrow $\za\in Q_1$ as follows. If $\za$ corresponds to the $P$-angle $(x,\tau_i,\tau_j)$ in the tile $\zD$, then the new $P$-arc  $\tau(\za)$ runs from the marked point $x $ to the new marked point $x(\za)$. 
The arc $\tau(\za)$ cuts the tile $\zD$ into two tiles $\zD_a$ and $\zD_b$, where $\zD_a$ is bounded by $\tau_i$
and  $\zD_b$ is bounded by $\tau_j$.

If the $P$-angles corresponding to several arrows $\za_1,\za_2,\ldots,\za_t$ lie in the same tile $\zD$, then we arrange the new $P$-arcs $\tau(\za_1),\tau(\za_2),\ldots,\tau(\za_t)$ such that they do not cross.

In particular, every tile of $G(P)$ contains a boundary segment of $(G(S),G(M))$.
\item The set $G(M^*)$ consists of one point on every boundary segment of $(S,M)$. If $y\in M^*$ lies in the tile $\zD(y)$ of $P$, we let $G(y)\subset G(M^*)$ be the subset of all the points that lie in a tile of $G(P)$ that is obtained by subdividing the tile $\zD(y)$. 
\end{enumerate}

Two examples are given in Figures~\ref{fig ex 7.5} and~\ref{fig orpheus}.

\begin{proposition}
 \label{prop 7.2}
 The algebra $\Abar$ is the tiling algebra of the tiled surface
 $G\TiledSurface$. 
\end{proposition}
\begin{proof}
We first show that the quiver $\Qbar$ is the quiver of the tiling $G\TiledSurface$. Indeed, the new vertices $v_\za$ correspond to the new arcs $\tau(\za)$, for each $\za\in Q_1$. For every arrow $\za\colon i\to j$ in $Q$, the path  $\xymatrix{i\ar[r]^{\za_a}&v_\za\ar[r]^{\za_b}&j}$ corresponds to the two $G(P)$-angles $(x, \tau_i, \tau(\za))$ and $(x, \tau(\za), \tau_j)$. 
 
Now we show that the ideal $\Ibar$ is the ideal of the tiling algebra. First observe that the composition $\za_a\za_b$ is nonzero in $\Abar$ as well as in the tiling algebra. On the other hand, if 
$\xymatrix{i\ar[r]^{\za}&j\ar[r]^{\zb}&k}$ lies in the ideal $I$, and thus $\za_b\zb_a\in \Ibar$, then in the surfaces we have the configuration shown in Figure~\ref{fig prop7.2} (possibly $x_1 = x_2$; compare with Figures ~\ref{fig:relations:a equals b} and   \ref{fig:relations:case 3}). 
In particular the path $\za_b\zb_a$ is a generator of the ideal of the tiling algebra, and all generators are of this form, since each marked point $x(\zg)$, with $\zg \in Q_1$, is only incident to one arc, $\tau (\zg)$, in $G(P)$.

\begin{figure}
\begin{center}
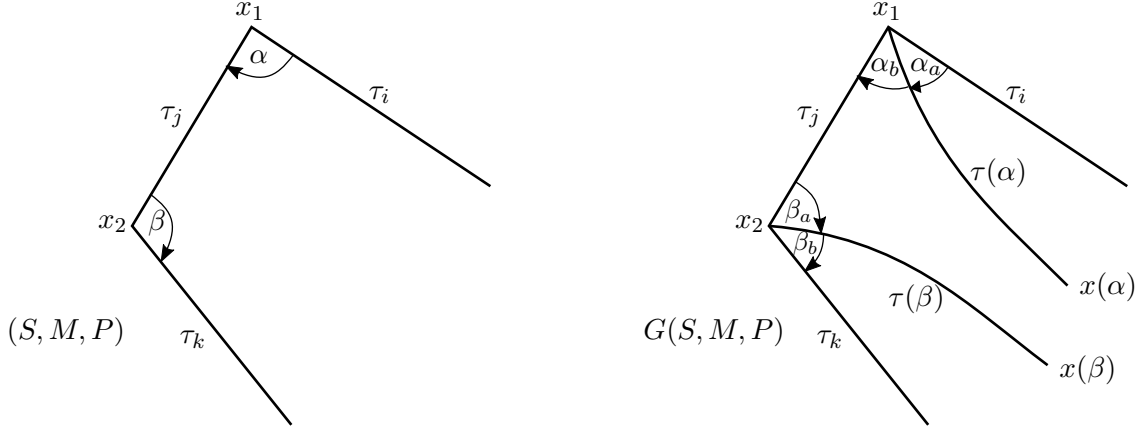
\caption{Proof of Proposition~\ref{prop 7.2}. The left picture shows a local configuration in the tiled surface $\TiledSurface$, and the right picture shows its image under $G$. The compositions $\za\zb$ and $\za_b\zb_a$ are zero in their respective tiling algebras.}
\label{fig prop7.2}
\end{center}
\end{figure}
\end{proof}

\subsection{The functor \texorpdfstring{$G\colon\textup{mod}\,A\to\textup{mod}\,\Abar$}{G}}\label{sect 7.3}
Given a string $w$ of $A$, we define $G(w)$ to be the string of $\Abar$ obtained by replacing each letter $\za$ with $\za_a\za_b$. If $M(w)$ is the string module of the string $w$, define $G(M(w))=M(G(w))$ to be the string module of the string $G(w)$.

If $w$ is a band of $A$, then $G(w)$ is a band of $\Abar$, again obtained by replacing each letter $\za$ with $\za_a\za_b$. 
If $M(w,\zl,r)$ is a band module of the band $w$ with parameters $\zl, r$, we define $G(M(w,\zl,r))=M(G(w),\zl,r)$ to be the band module of the band $G(w)$ with parameters $\zl,r$. This defines $G$ on indecomposable objects, and we extend $G$ additively to direct sums in $\textup{mod}\,A$. 

In particular, if $M=(M_i,M_\za)_{i\in Q_0,\za\in Q_1}$, we have the following description of the representation $G(M)$. On vertices $i$ of $\Qbar$, we have
$G(M)_i= M_i$ if $i\in Q_0$, and for all arrows $\za$ of $Q$, we have
$G(M)_{v_\za}=\im \,M_\za$. In particular, if $M=M(w)$ is a string module, then $G(M)_{v_\za}=\kb^{m_\za(w)}$ where $m_\za(w)$ is the multiplicity of the arrow $\za$ in the string $w$. 

On arrows of $\Qbar$, we can describe $G(M)$ by the following commutative diagram. 
\[\xymatrix@C34pt{
G(M)_{s(\za)}\ar[d]^=\ar[r]^{G(M)_{\za_a}}&G(M)_{v_\za}\ar[d]^=
\ar[r]^{G(M)_{\za_b} } &G(M)_{t(\za)}\ar[d]^=\\
M_{s(\za)}\ar[r]^{M_{\za}}&\im \,M_\za 
\ar[r]^{\textup incl. } &M_{t(\za)}\\
}\] 
Thus $G(M)_{\za_a}=M_\za$ and $G(M)_{\za_b}$ is the inclusion map $\im\, M_\za\subset M_{t(\za)}$. 
This defines the functor $G$ on objects.

To define $G$ on morphisms, let $f=(f_i)_{i\in Q_0}\colon M\to N$ be a morphism in $\textup{mod}\,A$. 
Let $G(f)_i=f_i$ if $i\in Q_0$, and for $\za\colon i\to j$ in $Q$ let $G(f)_{v_\za}=f_j|_{\im \,M_\za}$ be the restriction of the map $f_j$ to the image of $M_\za$. 
Let us check that $G(f)$ is indeed a morphism in $\textup{mod}\,\Abar$. For every arrow $\za\in Q_1$, the diagram 
\[\xymatrix@C34pt{
G(M)_i\ar[d]_{G(f_i)}\ar[r]^{G(M)_{\za_a}}&G(M)_{v_\za}\ar[d]^{G(f)_{v_\za}}
\ar[r]^{G(M)_{\za_b} } &G(M)_j\ar[d]^{G(f_j)}\\
G(N)_i\ar[r]^{G(N)_{\za_a}}&G(N)_{v_\za}
\ar[r]^{G(N)_{\za_b} } &G(N)_j\\
}\] 
is equal to the diagram
\[\xymatrix@C34pt{
M_i\ar[r]^{M_\za}\ar[d]_{f_i}&\im\,M_\za\ar[r]^{\textup incl}\ar[d]^{f_j|_{\im\, M\za}}&M_j\ar[d]^{f_j} 
\\
N_i\ar[r]^{N_\za}&\im\,N_\za\ar[r]^{\textup incl}&N_j }\]
and this diagram commutes because $f=(f_i)_{i\in Q_0}$ is a morphism in $\textup{mod}\,A$.
This shows that $G(f)$ is a morphism in $\textup{mod}\,\Abar$. To show $G$ is a covariant functor, it suffices to observe that $G(1_M)=1_{G(M)}$ and $G(f\circ g)=G(f)\circ G(g)$.
\begin{theorem}\label {thm functor G}
The functor $G\colon \textup{mod}\,A\to\textup{mod}\,\Abar$ is fully faithful. 
\end{theorem}
\begin{proof}
Let $f\in\Hom_A(M,N)$ be such that $G(f)=0.$ Then $G(f)_i=0$ for all vertices $i\in Q_0$, so by definition of $G$, we have $f_i=0$, for all $i\in Q_0$, hence $f=0$, and $G$ is faithful. 

Now let $h\in\Hom_{\Abar}(G(M),G(N)).$ By definition of $G$, we have the following commutative diagram for all arrows $\za\colon i\to j$ in $Q_1$.
\[\xymatrix@C34pt{
M_i\ar[r]^{M_\za}\ar[d]_{h_i}&\im\,M_\za\ar[r]^{\textup incl}\ar[d]^{h_{v_\za}}&M_j\ar[d]^{h_j} 
\\
N_i\ar[r]^{N_\za}&\im\,N_\za\ar[r]^{\textup incl}&N_j }\]
In particular, the commutativity of the right square implies $h_{v_\za}=h_j|_{\im\,M_\za}$.
Thus, we can define $f\in\Hom_A(M,N)$ by $f_i=h_i$ for all $i\in Q_0,$ and obtain $G(f)=h$. Hence $G$ is full.
\end{proof}

Note that $G$ is not an equivalence of categories, since $G$ is not dense. For example the simple modules $\Sbar(v_\za)$, with $\za\in Q_1$, do not lie in the image of $G$. The following proposition describes the image of $G$.
\begin{proposition}
    \label{prop im G}
    An  $\Abar$-module $\Mbar$ lies in the image of $G$ if and only if the simple summands of the top and the socle of $\Mbar$  are all of the form $\Sbar(i)$ with $i \in Q_0$.
\end{proposition}
\begin{proof} It suffices to prove the result in the case where $\Mbar$ is indecomposable.
If $\Mbar=G(M)$ for some $M\in \textup{mod}\,A$ with $\textup{top}\, M =\oplus_{i\in Q_0} m_i S(i)$ then $\textup{top}\, \Mbar=\oplus_{i\in Q_0} m_i \Sbar(i)$.  A similar formula holds for the socle.

Conversely, assume the top and socle of $\Mbar$ contain only summands  of the form $\Sbar(i)$, with $i \in Q_0$. If $\Mbar$ is a string $\Abar$-module then, whenever its string $\overline{w}$ contains the letter $\za_a$ (or $\za_b^{-1}$), it is immediately followed by the letter $\za_b$ (or $\za_a^{-1})$.   Define $w$ as the string in $A$ obtained from $\overline{w}$ by replacing every word $\za_a\za_b$ (or $\za_b^{-1}\za_a^{-1}$) with the letter $\za$ (or $\za^{-1}$). Then $G(w)=\overline{w}$ and thus $G(M(w))=\Mbar.$ The proof is similar if $\Mbar$ is a band module.
\end{proof}

Recall that the \emph{dimension} of a module is the sum of the entries in its dimension vector. 
\begin{proposition} \label{prop Gdim} 
Let $M$ be an indecomposable  $A$-module. 
\begin{enumerate}[(a)]
\item \label{prop Gdim:itm:a}
If $M$ is a string module then $\dim\, G(M)=2\,\dim\, M-1.$
\item \label{prop Gdim:itm:b}
If $M$ is a band module then $\dim\, G(M)=2\,\dim\, M.$ 
\end{enumerate} 
\end{proposition}
\begin{proof} \ref{prop Gdim:itm:a}
Let $d$ be the number of vertices in the coefficient quiver $\Gamma$ of $M$. Then $\Gamma$ has exactly $d-1$ arrows. The dimension of $M$ is the number of vertices of $\Gamma$. The dimension of $G(M)$ is the number of vertices plus the number of arrows. Thus $\dim\, M=d$ and $\dim\, G(M)=2d-1.$

\ref{prop Gdim:itm:b}
Similarly, the band defining $M$ has the same number of vertices and arrows. Thus the dimension of  $G(M)$ is twice the dimension of $M$.
\end{proof}
\subsection{\texorpdfstring{$G$ maps maximal almost rigid $A$-modules to tilting $\Abar$-modules}
{G maps maximal almost rigid A-modules to tilting Abar-modules}}
\label{sec:mar to tilting}

Tilting theory is an important branch of representation theory that plays a crucial role in several applications including the categorification of cluster algebras, see \cite{BMRRT06}. The main objects  of interest are the tilting modules and their endomorphism algebras.
We recall the following definition from  \cite[page 192--193, Definition VI.2.1]{ASS}.

Let $\Lambda$ be an algebra. A $\Lambda$-module $T$ is called a \emph{tilting module} if the following conditions are satisfied.
\begin{enumerate}[(1)]
\item\label{thm:tilting:itm1} The projective dimension of $T$ is at most 1.
\item\label{thm:tilting:itm2} $T$ has no self-extensions, that is, $\Ext^1(T,T)=0.$
\item \label{thm:tilting:itm3} 
There is a short exact sequence 
\[0\to \Lambda\to T'\to T''\to 0  \] with $T',T''\in \add \, T$.
 \end{enumerate}
By \cite[page 217, Corollary VI.4.4]{ASS},  condition \ref{thm:tilting:itm3} can be replaced by the following:
\begin{enumerate}[($3'$)]
\item
\label{thm:tilting:itm3prime} 
the number of pairwise non-isomorphic indecomposable summands of $T$ equals the number of pairwise non-isomorphic simple modules. 
\end{enumerate}

\smallskip

\begin{theorem}\label{thm:tilting}
Let $T$ be a maximal almost rigid module over a gentle algebra $A$. Then $G(T)$ is a tilting $\Abar$-module.
\end{theorem}

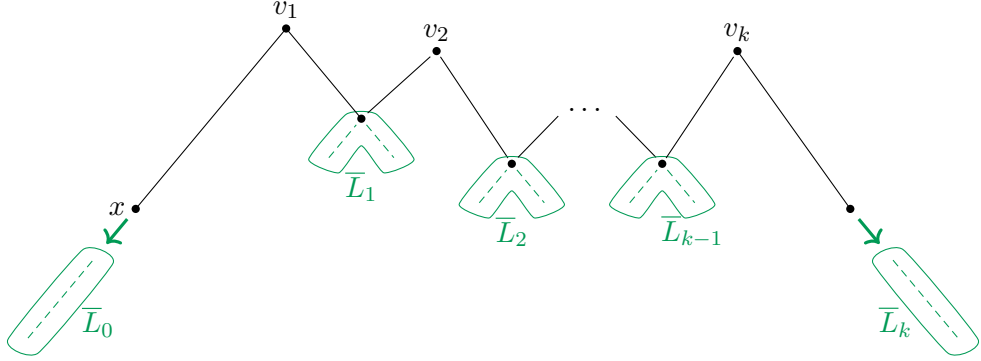
\begin{figure}
\centering
\begin{tikzpicture}[xscale=1,yscale=0.6]
\node (xLL) at (-1.5,-3) {};
\node (xL) at (-0.5,-1) {};

\node at (-0.5,-2.5) [ForestGreen] {$\Lbar_0$};

\draw[ForestGreen, densely dashed, shorten <=-2pt, shorten >=-2pt]  (xLL) -- (xL);

\draw plot [smooth cycle, tension=.3] 
coordinates {(-0.5-0.2,-1+0.1)(-0.5+0.2,-1-0.2) (-1.5+0.2,-3-0.2)(-1.5-0.2,-3+0.1)} [ForestGreen];

\node (x) at (0,0) {};
\node at (x) [fill,circle,inner sep=1.1] {};
\node at (x) [left] {$x$};

\draw [very thick, ForestGreen,->] (x) -- (xL);

\node (v1) at (2,4) {};
\node at (v1) [fill,circle,inner sep=1.1] {};
\node at (v1) [above] {$v_1$};

\node (v12) at (3,2) {};
\node at (v12) [fill,circle,inner sep=1.1] {};

\draw plot [smooth cycle, tension=.3] 
coordinates {
(3-0.2,2+0.1)
(3+0.2,2+0.1) 
(3.5+0.2,1+0.1)
(3.5-0.2,1-0.2)
(3,1.4)
(2.5+0.2,1-0.2)
(2.5-0.2,1+0.1)} [ForestGreen];

\node at (3,0.5) [ForestGreen] {$\Lbar_1$};

\node (v12L) at (2.5,1) {};
\node (v12R) at (3.5,1) {};

\draw [shorten <=-4pt, shorten >=-4pt] (x) -- (v1);
\draw [shorten <=-4pt, shorten >=-4pt] (v1) -- (v12);

\draw[ForestGreen, densely dashed, shorten <=-2pt, shorten >=-2pt] (v12L) -- (v12);
\draw[ForestGreen, densely dashed, shorten <=-2pt, shorten >=-2pt]  (v12) -- (v12R);

\node (v2) at (4,3.5) {};
\node at (v2) [fill,circle,inner sep=1.1] {};
\node at (v2) [above] {$v_2$};

\draw[shorten <=-2pt, shorten >=-2pt] (v2) -- (v12);

\node (v23) at (5,1) {};
\node at (v23) [fill,circle,inner sep=1.1] {};
\draw[shorten <=-2pt, shorten >=-2pt] (v2) -- (v23);

\node (v23L) at (4.5,0) {};
\node (v23R) at (5.5,0) {};

\draw plot [smooth cycle, tension=.3] 
    coordinates {
    (5-0.2,1+0.1)
    (5+0.2,1+0.1) 
    (5.5+0.2,0+0.1)
    (5.5-0.2,0-0.2)
    (5,.4)
    (4.5+0.2,0-0.2)
    (4.5-0.2,0+0.1)} [ForestGreen];

\node at (5,-0.5) [ForestGreen] {$\Lbar_2$};

\draw[ForestGreen, densely dashed, shorten <=-2pt, shorten >=-2pt] (v23L) -- (v23);
\draw[ForestGreen, densely dashed, shorten <=-2pt, shorten >=-2pt]  (v23) -- (v23R);

\node (v3) at (5.7,2.2) {};

\draw[shorten <=-2pt, shorten >=-2pt] (v3) -- (v23);

\node at (6,2.2) {$\dots$};

\node (v4) at (6.3,2.2) {};

\node (v45) at (7,1) {};
\node at (v45) [fill,circle,inner sep=1.1] {};
\draw[shorten <=-2pt, shorten >=-2pt] (v4) -- (v45);

\draw plot [smooth cycle, tension=.3] 
    coordinates {
    (7-0.2,1+0.1)
    (7+0.2,1+0.1) 
    (7.5+0.2,0+0.1)
    (7.5-0.2,0-0.2)
    (7,.4)
    (6.5+0.2,0-0.2)
    (6.5-0.2,0+0.1)} [ForestGreen];

\node at (7.4,-0.5) [ForestGreen] {$\Lbar_{k-1}$};

\node (v45L) at (6.5,0) {};
\node (v45R) at (7.5,0) {};

\draw[ForestGreen, densely dashed, shorten <=-2pt, shorten >=-2pt] (v45L) -- (v45);
\draw[ForestGreen, densely dashed, shorten <=-2pt, shorten >=-2pt]  (v45) -- (v45R);

\node (vk) at (8,3.5) {};
\node at (vk) [fill,circle,inner sep=1.1] {};
\node at (vk) [above] {$v_k$};

\draw[shorten <=-2pt, shorten >=-2pt] (vk) -- (v45);

\node (y) at (9.5,0) {};
\node at (y) [fill,circle,inner sep=1.1] {};

\draw [shorten <=-4pt, shorten >=-4pt] (y) -- (vk);

\node (yR) at (10,-1) {};
\node (yRR) at (11,-3) {};
\draw [very thick, ForestGreen,->] (y) -- (yR); 

\draw plot [smooth cycle, tension=.3] 
coordinates {(10-0.2,-1-0.2)(10+0.2,-1+0.1) (11+0.2,-3+0.1)(11-0.2,-3-0.2)} [ForestGreen];
    
\node at (10.1,-2.5) [ForestGreen] {$\Lbar_k$};

\draw[ForestGreen, densely dashed, shorten <=-2pt, shorten >=-2pt]  (yR) -- (yRR);
\end{tikzpicture}
\caption{Illustration for the proof of Theorem~\ref{thm:tilting}.}
\label{fig:thm:tilting}
\end{figure}

\begin{proof} 
\ref{thm:tilting:itm1} Let $T_1$ be an indecomposable summand of $T$ and $g \colon \overline{P}\to G(T_1)$ be a projective cover of $G(T_1)$ in $\textup{mod}\,\Abar$. Denote by $\Lbar$ the kernel of $g$, and write $\overline{P} = \Pbar_1 \oplus \dots \oplus \Pbar_k$ as a direct summand of indecomposable projective modules. We want to show that $\Lbar$ is projective. Since $\Abar$ is a gentle algebra, \cite[Proposition 3.1]{BS21} implies that $\Lbar=\Lbar_0\oplus \Lbar_1\oplus \cdots\oplus \Lbar_k$, where $\Lbar_1,\ldots,\Lbar_{k-1}$ are indecomposable projectives. It remains to show that $\Lbar_0$ and $\Lbar_k$ are either zero or projective modules. We will only show the former, as the proof of the latter is similar.

Let $w$ be a string of $\Abar$ for which $G(T_1) = M(w)$, and let $x=s(w)$ and $v_1, \ldots, v_k$ be the vertices defining the hills of $w$ in order. Note that $x, v_1, \ldots, v_k \in Q_0$, by Proposition~\ref{prop im G}. Moreover, up to reordering, we have $\Pbar_i = \Pbar(v_i)$, for $i=1, \ldots, k$, and $\Lbar_0$ is either zero or a uniserial submodule of the radical of $\Pbar_1$. See Figure~\ref{fig:thm:tilting} for an illustration.

Now, if there is $\za \in Q_1$ such that $s(\za) = x$ and $(\za_a \za_b)^{-1}w$ is a string of $\Abar$, then $\Lbar_0$ corresponds to a string whose first letter is $\za_b$. In particular, the top of $\Lbar_0$ is of the form $\Sbar (v_\za)$. It thus follows from Proposition~\ref{prop:projective Abar:itm4} that $\Lbar_0$ is projective. 

If no such arrow $\za$ exists, then $\Sbar (x)$ is a summand of the socle of $\Pbar_1$, and so $\Lbar_0 = 0$, which finishes the proof of~\ref{thm:tilting:itm1}.

\ref{thm:tilting:itm2} Assume there is a short exact sequence
\begin{equation}\label{eq ses}   \xymatrix{0\ar[r]&G(T_1)\ar[r]^-{\fbar}&\overline{E}\ar[r]^-{\overline{g}}&G(T_2)\ar[r]&0},
\end{equation}
with $T_1,T_2$ indecomposable summands of $T$. 

Our first goal is to show that $\overline{E}$ lies in the image of $G$. 
Suppose $i\colon\Sbar\hookrightarrow \overline E$ is a simple submodule of $\overline{E}$. If $\overline{g}\circ i =0$, then the universal property of the kernel yields a morphism $h\colon \Sbar\to G(T_1)$ such that $i=\fbar\circ h$. 
Since $\Sbar$ is simple and $h \neq 0$, we see that $h$ must be injective, and hence $\Sbar$ is a submodule of $G(T_1).$ Otherwise, $\overline{g}\circ i\ne 0$, and since $\Sbar$ is simple, this means $\overline{g}\circ i $ is injective. Thus $\Sbar $ is a submodule of $G(T_2).$ 
We have shown that every indecomposable summand of the socle of $\overline{E}$ is also a summand of the socle of $G(T_1)\oplus G(T_2).$ Using Proposition~\ref{prop im G}, we conclude that the socle of $\overline{E}$ only contains summands of the form $S(i)$, with $i\in Q_0.$ 
The analogous result for the top of $\overline{E}$ is proved in a similar way. 
Applying Proposition~\ref{prop im G} again, we see that $\overline{E}$ lies in the image of $G$.

Say $\overline{E}=G(E)$ for some $A$-module $E$. Then, since $G$ is full, there exist morphisms $f\colon T_1\to E$ and $g\colon E\to T_2$ in $\textup{mod}\,A$ such that $\fbar=G(f)$ and $\overline{g}=G(g)$. From the definition of $G$, we obtain
$\im \fbar_i=\im f_i $ and $\ker \overline{g}_i=\ker g_i$ for every vertex $i\in Q_0.$ Therefore we get an exact sequence
$\xymatrix{0\ar[r]&T_1\ar[r]^-{f}&E\ar[r]^-{g}&T_2\ar[r]&0}$
in $\textup{mod}\,A.$
Since $T$ is almost rigid, this sequence is either split, or the middle term $E$ is indecomposable. If the sequence splits, then there exists $u\colon E\to T_1$ such that $u\circ f=1_{T_1}$. 
Consequently $G(u)\circ \fbar=G(u)\circ G(f)=G(u\circ f)= 1_{G(T_1)}$ and hence the sequence \eqref{eq ses} splits. Suppose now that the middle term $E$ is indecomposable. Because of exactness, we have $\dim \,E=\dim \, T_1+\dim \, T_2$. Since $T_1,T_2$ are string modules,  Proposition~\ref{prop Gdim} implies  $\dim \, G(T_1)+\dim \, G(T_2)= 2(\dim \, T_1+\dim \, T_2)-2=2\,\dim \, E-2$. On the other hand, the same proposition yields $\dim \, G(E)= 2\dim \, E -1$, if 
 $E$ is a string module, and $\dim \, G(E)= 2\dim \, E $, if 
 $E$ is a band module. In both cases, we have a contradiction to the exactness of \eqref{eq ses}. 

Thus the sequence \eqref{eq ses} splits and we have proved condition \ref{thm:tilting:itm2}.

\ref{thm:tilting:itm3prime}
By Corollary \ref{cor MAR summands}, the number of indecomposable summands of $T$ is $|Q_0| +|Q_1|$. By definition, this is also equal to the number of vertices in $\Qbar,$ and hence the number of simple modules over $\Abar$. The result follows since $G$ is fully faithful.
\end{proof}

\subsection{The image of $A$-modules under $G$ in the geometric model}
\label{sect 7.5}

Let $X$ be an indecomposable string $A$-module and $\zg$ the corresponding permissible arc in $\TiledSurface$. 
We first assume that $X$ is not simple.
Recall that each arrow $\alpha_i$ or inverse arrow $(\alpha_i)^{-1}$ in the string of $X$ corresponds to a $P$-angle which is traversed by $\gamma$, and denoted $\angle\gamma^i$.
In particular, the curve $\gamma$ is determined by the sequence:
\[y, \angle\gamma^1, \angle\gamma^2,\ldots, \angle\gamma^t,y' \]
where $y,y'\in M^*$ are the endpoints of $\zg$.

Observe that in  $(G(S),G(M), G(P))$ each $P$-angle $\angle \gamma^i$ corresponding to an arrow $\alpha_i\in Q_1$ is bisected by a new $P$-arc $\tau(\alpha_i)$ to form two $P$-angles, $\angle \gamma^{ia}$ and $\angle \gamma^{ib}$.
The pair of $P$-angles $\angle \gamma^{ia}$ and $\angle \gamma^{ib}$ corresponds to the pair of arrows $\alpha_{ia}$ and $\alpha_{ib}$.
The permissible arc $G(\zg)$ in $(G(S),G(M),G(M^*),G(P))$ which represents the module $G(X)$ is defined by the sequence:

\[y_1, \angle\gamma^{1a},\angle\gamma^{1b},  \angle\gamma^{2a},\angle\gamma^{2b},\ldots, \angle\gamma^{ta}, \angle\gamma^{tb},y_1' \]
where $y_1, y_1'\in G(M^*)$ are uniquely determined by the rest of the sequence.

If $X$ is simple $S(i)$, then $\gamma$ traverses no angles and crosses only the $P$-arc $\tau_i$.
We define $G(\gamma)$ so that it also traverses no angles and crosses only $G(\tau_i)$.

If $T=T_1 \oplus \cdots \oplus T_m$ is a maximal almost rigid $A$-module and $\calt = \{\gamma_1, \ldots, \gamma_m\}$ is the corresponding triangulation of $(S, M^*)$, then $G(\calt)$ denotes the set of permissible arcs $\{G(\gamma_i) \mid i =1, \ldots m \}$ of $(G(S),G(M),G(M^*),G(P))$.

For an illustration, see Example~\ref{ex 7.5}. For instance, here the arc $\zg_2$ in $\calt$ that joins the points $y_2,y_3$ in $M^*$ corresponds to the $A$-module ~$\begin{smallmatrix}
 2\\1
\end{smallmatrix}$.
We have $G(y_2)=\{y_{21},y_{22}\}$ and $G(y_3)=\{y_3\}$. The starting point of $G(\zg_2)$ is $y_{22}$, the unique point of $G(M^*)$ in the first tile of $G(\zg)$.

\subsection{The tiled surface of \texorpdfstring{$\End_A T$}{End T}}\label{sect 7.6}

Recall that we denote by $C$ the endomorphism algebra $\End_A T$ of an MAR module $T$ over a gentle algebra $A$. Also recall the notation $T=T_1 \oplus \cdots \oplus T_m$ for the decomposition of $T$ into pairwise non-isomorphic indecomposable summands.
Now let $\calt = \{ \gamma_1, \ldots, \gamma_m\}$ denote the corresponding triangulation of $(S,M^*)$.

\begin{lemma} \label{lem end}
We have $C\cong \End_{\Abar} \ G(T).$ 
\end{lemma}
\begin{proof}
    This follows because $G$ is fully faithful by Theorem~\ref{thm functor G}.
\end{proof}
\begin{corollary}\label{cor:end-alg-tilt}
    \label{cor end}
    The algebras $C$ and $\Abar$ are derived equivalent. In particular, $C$ is a gentle algebra.
\end{corollary}
\begin{proof}
    We have seen in Theorem~\ref{thm:tilting} that $G(T)$ is a tilting $\Abar$-module, and thus  $C$ and $\Abar$ are derived equivalent. 
Proposition~\ref{probAbargentle} shows that $\Abar$ is gentle, and thus $ 
   C$ is gentle as well, by a result of \cite{SZ03}. 
\end{proof}

The next theorem describes the tiling for which $C$ is the tiling algebra. For the proof, recall from Theorem~\ref{thm hom basis} that morphisms between string modules (apart from isomorphisms)  
correspond to permissible angles, as defined in Definition~\ref{def:permissible-angle}. 

Since we are considering morphisms between summands of an MAR module, the only angles we need to consider are at endpoints, as illustrated in Figure~\ref{fig angles 1 later}. 

\begin{definition}($\calt$-angle)
We will call a \emph{$\calt$-angle} an angle  of the form $(x, \za^\circ, \zb^\circ)$, where $x \in M^*$ and $\za, \zb \in \calt$ are such that $\zb^\circ$ follows $\za^\circ$ immediately around $x$ in the counterclockwise order. 
\end{definition}

All $\calt$-angles illustrated in Figure~\ref{fig angles 1 later} are indicated by a blue arrow. The leftmost picture illustrates a $\calt$-angle of the form $(x, \za^\circ, \zb^\circ)$, where $x$ is not a puncture. 
The two  pictures on the right hand side of Figure~\ref{fig angles 1 later} show  configurations with multiple $\calt$-angles s at punctures ($x$ and $y$) involving the same arcs ($\za$ and $\zb$). In the first of these two pictures, the arc $\za$ starts and ends at the same point $x$. We have therefore three $\calt$-angles: $(x,\za_1^\circ, \zb^\circ)$, $(x, \zb^\circ, \za_2^\circ)$ and $(x, \za_2^\circ, \za_1^\circ)$. 
The rightmost picture has $\calt$-angles $(x, \zb^\circ, \za^\circ)$ and $(x, \za^\circ, \zb^\circ)$.

 \begin{figure}[ht!]
 \begin{center}
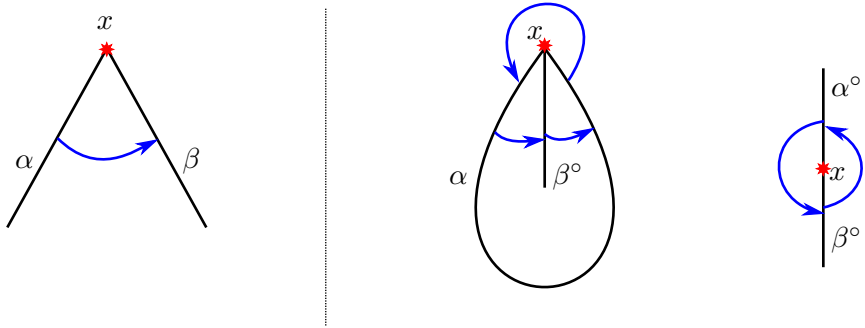
\caption{Examples of $\calt$-angles at an endpoint.}
\label{fig angles 1 later}
\end{center}
\end{figure}

The only $\calt$-angles at endpoints giving rise to a  nonzero morphism are the permissible ones, i.e. those traversed by an arc in $P$. See  Figure \ref{fig angles 2 later} for an illustration. In the first picture of that figure, the region is simply connected, so $\tau$ traverses the $\calt$-angle. In the second picture, the region contains a boundary component, thus $\tau$ does not traverse the $\calt$-angle. In the third picture, $\tau$ does not cross $\zb$.

\begin{figure}[ht!]
 \begin{center}
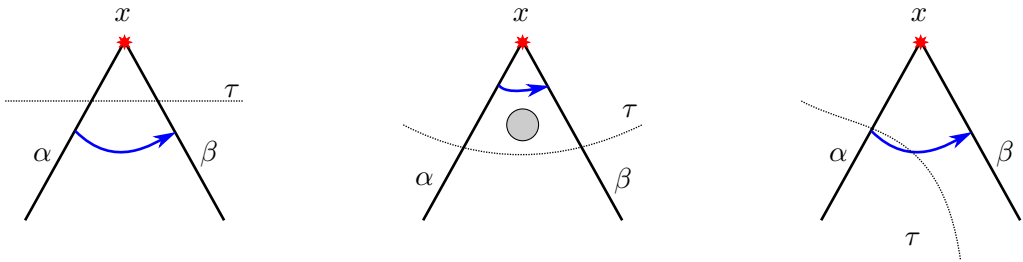
\caption{The $\calt$-angle $(x,\za^\circ,\zb^\circ)$ is traversed by the arc $\tau$ in the picture on the left, but not in the two pictures on the right. }
\label{fig angles 2 later}
\end{center}
\end{figure}

Let $\za, \zb$ be two permissible arcs and suppose there is a permissible $\calt$-angle $(x, \za^\circ, \zb^\circ)$. We denote by $\xymatrix@C45pt{M(\za)\ar[r]^{(x,\za^\circ,\zb^\circ)}
&M(\zb)}$ the corresponding morphism of string modules. We will need to consider composition of such morphisms, for which we need the following definition.

\begin{definition}
An ordered sequence of two $\calt$-angles $(x_1,\za^\circ,\zb_1^\circ)$,$(x_2,\zb_2^\circ,\zg^\circ)$, is called \emph{consecutive} if 
\begin{enumerate}
\item $x_1=x_2$ is the same vertex,
\item $\zb_1=\zb_2$ is the same arc, and 
\item $\zb_1^\circ=\zb_2^\circ$ is the same end of the arc.
\end{enumerate}   
See Figure \ref{fig angles 3} for an illustration.
\end{definition}

 \begin{figure}[ht!]
 \begin{center}
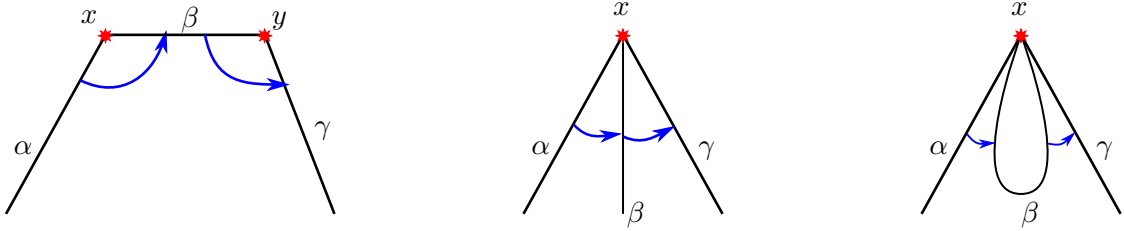
\caption{Left: The two blue angles are not consecutive because they do not share the same vertex. Right: The two blue angles are not consecutive because the two ends of the arc $\zb$ are not the same. Middle: The two angles are consecutive.} 
\label{fig angles 3}
\end{center}
\end{figure}

\begin{lemma}\label{lem:composition}
Let $\za, \zb$ and $\zg$ be permissible arcs with no crossings. Suppose we have permissible $\calt$-angles $(x,\za^\circ,\zb^\circ)$, $(y,\zb^{\circ \circ},\zg^\circ)$, where $\zb^{\circ \circ}$ may or may not coincide with $\zb^\circ$. 
The following statements are equivalent:
\begin{enumerate}[(1)]
\item\label{lem:composition:itm1} The composition 
\[\xymatrix@C45pt{M(\za)\ar[r]^{(x,\za^\circ,\zb^\circ)}
&M(\zb)\ar[r]^{(y,\zb^{\circ \circ},\zg^\circ)}
&M(\zg)}\]
is nonzero;
\item \label{lem:composition:itm2} the $\calt$-angles  $(x,\za^\circ,\zb^\circ)$, $(y,\zb^{\circ \circ},\zg^\circ)$ are consecutive, i.e.~$x=y$ and $\zb^\circ = \zb^{\circ \circ}$.
\end{enumerate}
\end{lemma}
\begin{proof}
First we show that \ref{lem:composition:itm2} implies \ref{lem:composition:itm1}. 
Suppose $x = y$ and $\zb^\circ = \zb^{\circ \circ}$. Let us orient the arcs $\za, \zb, \zg$ so that $x = y$ is their starting point, and recall, from Definition~\ref{def:string-from-arc}, that $w(\delta)$ denotes the string associated to the permissible arc $\delta$. By Theorem~\ref{thm hom basis}, the permissible $\calt$-angle $(x, \za^\circ, \zb^\circ)$ corresponds to a nonzero morphism from $M(\za)$ to $M(\zb)$. By Proposition~\ref{prop: string_hom_basis}, this non-zero morphism corresponds to a string $e$ which is an up-set of $w(\za)$ and a down-set of $w(\zb)$. The string $e$ can be read from the surface by the sequence of angles defined by the crossings with $P$ that $\za$ and $\zb$ have in common. Since $\za$ and $\zb$ meet at an endpoint, it follows that $e$ is a prefix of $w(\za)$ and of $w(\zb)$. Similarly, denote by $e'$ the up-set of $w(\zb)$ and down-set of $w(\zg)$ corresponding to the permissible $\calt$-angle $(x, \zb^\circ, \zg^\circ)$, and note that $e'$ is a prefix of $w(\zb)$ and of $w(\zg)$.

Now let $\tau$ (respectively $\tau'$) be the first arc of $P$ that traverses  $(x, \za^\circ, \zb^\circ)$ (respectively $(x,\zb^\circ, \zg^\circ)$). It follows from the definition of permissible angle and the fact that $P$ is a tiling that $\tau = \tau'$. So the angle $(x, \za^\circ, \zg^\circ)$ is permissible. The corresponding up-set of $w(\za)$ and down-set of $w(\zg)$ is the intersection of $e$ and $e'$, which corresponds to the composition map $\xymatrix@C45pt{M(\za)\ar[r]^{(x,\za^\circ,\zb^\circ)}
&M(\zb)\ar[r]^{(y,\zb^\circ,\zg^\circ)}
&M(\zg)}$, proving~\ref{lem:composition:itm1}.

Next we show \ref{lem:composition:itm1} implies \ref{lem:composition:itm2}.  
Let $e$ (respectively $e'$) be the up-set of $w(\za)$ (respectively $w(\zb)$) and down-set of $w(\zb)$ (respectively $w(\zg)$) associated to the angle $(x, \za^\circ, \zb^\circ)$ (respectively $(y, \zb^{\circ \circ}, \zg^\circ)$). We can assume $e$ is a prefix of both $w(\za)$ and $w(\zb)$. If $e'$ is also a prefix of $w(\zb)$ and $w(\zg)$, then \ref{lem:composition:itm2} follows immediately. So assume $e'$ is not a prefix of $w(\zb)$ and $w(\zg)$. Then $e'$ must be a suffix of both $w(\zb)$ and $w(\zg)$, as otherwise $\beta$ and $\gamma$ would cross.

Since the composition of the morphisms is nonzero, we must have $e= e_\ell a^{-1} e''$ and $e' = e'' b^{-1} e'_r$, where $a \in Q_1$ (since $e'$ is not a prefix of $w(\zb)$) and $b \in Q_1 \cup \{0\}$.

If $b = 0$, then either there is an overlap extension between $M(\za)$ and $M(\zg)$ with overlap $e''$ or $w(\za) = e = w(\zb)$. But the former contradicts the fact that there are no crossings between $\za$ and $\zg$, and the latter contradicts the fact that permissible angles don't correspond to the identity morphism. Therefore $b \neq 0$, and again there is an overlap extension between $M(\za)$ and $M(\zg)$ with overlap $e''$, a contradiction. Hence, $e'$ is indeed a prefix of both $w(\zb)$ and $w(\zg)$, and we are done.
\end{proof}

In Figure~\ref{fig angles 3}, the composition is zero in both the left hand side and right hand side pictures, and it is nonzero in the picture in the center.

\begin{theorem}\label{thm endoalg}
The quadruple $(G(S),G(M^*),G(M^*)^*,G(\calt))$ is a tiled surface and the corresponding tiling algebra is $C$. 
\end{theorem}

\begin{remark}
Recall that $G(S)$ and $G(M^*)$ are defined as in Section~\ref{sec:tiling G(S,M,P)}, $G(\calt)$ is defined as in Section~\ref{sect 7.5}, and $(G(M^*))^*$ is defined as in Section~\ref{sec:tiled surfaces}. In particular, we have that:
\begin{itemize}
    \item $G(S)$ is obtained from $S$ by replacing the punctures by boundary components;
    \item $G(M^*)$ consists of the endpoints of the arcs in $G(\calt)$, and
    \item $(G(M^*))^*$ consists of the endpoints of the arcs in $G(P)$. 
\end{itemize} 

The bottom right pictures of Figures~\ref{fig ex 7.5} and~\ref{fig orpheus} illustrate this tiling for the examples in Section~\ref{sect examples}.
\end{remark}

\begin{proof}[Proof of Theorem~\ref{thm endoalg}]
The quiver of $C$ has vertices $1,2,\ldots, m$ in bijection with the summands of $T$ and hence with the arcs of the triangulation $\mathcal{T}$. The number of arrows $i\to j$ in $C$ is equal to the dimension of the quotient of $\Hom_A(M(\zg_j),M(\zg_i))$ by the subspace of all morphisms that factor non-trivially through $T$. It follows from Lemma~\ref{lem:composition} that there is an arrow $\za\colon i \to j $ in $C$ if and only if there is a permissible $\calt$-angle $(x,\zg_j^\circ,\zg^\circ_i)$. 
Moreover the composition of two arrows $\xymatrix{i\ar[r]^\za&j\ar[r]^\zb&k}$ is nonzero if and only if the corresponding permissible $\calt$-angles $(x,\zg_j^\circ,\zg_i^\circ)$ and $(x,\zg_k^\circ,\zg_j^\circ)$ are consecutive.

Next, we consider the relation between the triangulation $\mathcal{T}$ of $(S,M^*)$ and the  surface tiling $(G(S),G(M^*) ,G(M^*)^*,G(\mathcal{T}))$. Let  $(y,\zg_i^\circ,\zg_j^\circ)$ be a $\calt$-angle in $(S,M^*)$. 
 From the definition of $G(M^*)$ it follows that if the $\calt$-angle is permissible, and $\tau$ is a $P$-arc of the tiling $P$ that traverses said angle, then the set $G(y)$ consists of a single point and $(G(y),G(\zg^\circ_i),G(\zg^\circ_j))$ is a $G(\calt)$-angle in $(G(S),G(M^*))$ that is traversed by the $P$-arc $G(\tau)$ in $G(P)$. 
On the other hand, 
if the $\calt$-angle $(y,\zg_i^\circ,\zg^\circ_j)$ is not permissible, then $G(y)$ consists of at least two points and the two ends $G(\zg_i^\circ), G(\zg_j^\circ)$ do not share an endpoint. We therefore obtain a bijection from the set of permissible $\calt$-angles to the set of $G(\calt)$-angles in $(G(S),G(M^*),G(M^*)^*,G(\calt))$. 
In particular, we have an isomorphism
\begin{equation*}
\label{eq iso hom}
\Hom_A(M(\zg_i),M(\zg_j))\quad \cong\quad e_i\, C'\, e_j,
\end{equation*}
where $C'$ denotes the tiling algebra of $(G(S),G(M^*),G(M^*)^*,G(\mathcal{T}))$ and $e_i,e_j$ are the primitive idempotents corresponding to the arcs $G(\zg_i),G(\zg_j)$
respectively. Taking the direct sum over all pairs of summands of $T$, we obtain
\begin{equation*}
C\ =\ \End_A\, T \ =\  \bigoplus_{i,j} \Hom_A(M(\zg_i),M(\zg_j)) \  \cong \ \bigoplus_{i,j}e_i\, C'\, e_j \ =\  C'.
\qedhere
\end{equation*}
\end{proof}

\subsection{The tiling algebra of a triangulation}
The permissible triangulation $\calt$ of $\TiledSurface$ is a tiling of the surface $(S,M^*,M)$ where we exchange the roles of $M$ and $M^*$. It is then natural to consider the associated tiling algebra, which we denote by $B$.

\begin{remark} It is important to note that this tiled surface $(S,M^*,M,\calt)$ is different from the ones considered so far in this paper because the first set of marked points $M^*$ may contain points in the interior of the surface. 
 The algebra $B$ is infinite dimensional if the global dimension of $A$ is infinite. 
 By \cite{PPP}, $B$ is locally gentle.
\end{remark} 
In this subsection, we study the relation between $B$ and the endomorphism algebra $C=\End_A T$ of section~\ref{sect 7.6}.

\begin{figure}[htbp!]
    \centering
\scalebox{0.6}{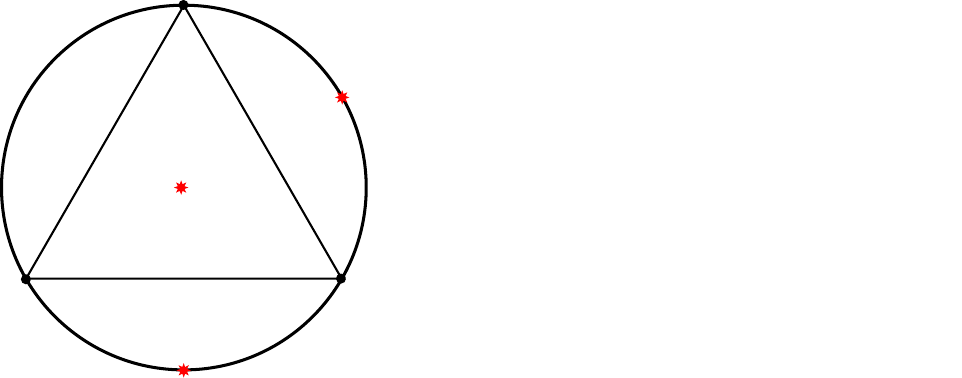}
    \caption{An illustration of Example \ref{ex skew gentle}. The left picture shows the tiled surface $\TiledSurface$ and the right picture shows the tiled surface $(S,M^*,M,\calt)$. }
    \label{fig skew gentle}
\end{figure}

\begin{example}\label{ex skew gentle} 
Let $A$ be the algebra given by the quiver $\xymatrix@R10pt@C10pt{1\ar[rr]^\za&&2\ar[ld]^\zb\\&3\ar[lu]^\zg}$ bound by the relations $\za\zb=\zb\zg=\zg\za=0.$ Then $A$ is the tiling algebra of the tiled surface $\TiledSurface$ shown in the left picture of Figure~\ref{fig skew gentle}. This algebra has precisely six indecomposable modules, namely $P(1),P(2),P(3),S(1),S(2)$ and $S(3)$.
The sum of all six modules is the unique MAR $A$-module. The corresponding triangulation $\calt$ is shown in the right picture of the same figure. The algebra $B$ is the tiling algebra of the tiled surface $(S,M^*,M,\calt)$. Note that the first set of marked points $M^*$ contains a point in the interior of the surface. Thus the algebra $B$ is given by the quiver \[\xymatrix@R15pt{&&P(3)\ar[ld]_{\zd}\\
&S(1)\ar[rr]^{\zd}\ar[ld]_{\zs}&&S(3)\ar[lu]_{\zd}\ar[ld]^{\tau}
\\
P(1)\ar[rr]_{\zs}&& S(2)\ar[lu]^{\zs}\ar[rr]_{\tau}&& P(2)\ar[lu]_{\tau}
}\]
with relations $\zd^2,\zs^2,\tau^2$. Note that there are no relations on the  3-cycle $\zd\tau\zs$, so this algebra is infinite dimensional.
\end{example}

We need some preliminary results. Recall from Definition \ref{def:permissible-angle} that an angle between arcs of $\calt$ is permissible if it is traversed by a $P$-arc of $P$.

\begin{lemma}
 \label{lem:1-angle-not-traversed} 
 Every triangle in $\calt$ contains a unique $\calt$-angle that is not permissible.
\end{lemma}
\begin{proof}
First consider a triangle which is not self-folded and suppose all $\calt$-angles of said triangle were permissible. Let $\tau_{i}$ be the last $P$-arc of $P$ traversing the $\calt$-angle at $x_i$, with $i =1, 2, 3$; see Figure~\ref{fig:every-angle-traversed}. 

\begin{figure}[ht!]
 \centering
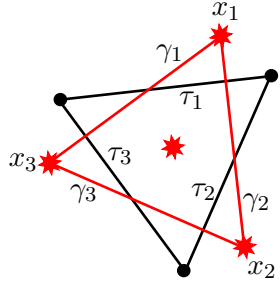
\caption{The case where every $\calt$-angle is traversed by arcs in $P$.}
\label{fig:every-angle-traversed}
\end{figure}

Since each arc $\gamma_i$ is permissible, it follows that the $P$-arcs $\tau_{i}$ and $\tau_{i-1}$ of $P$ form a $P$-angle, for each $1 \leq i \leq 3$ (where $i-1$ is taken mod 3). Hence $\tau_1,\tau_2,\tau_3$ form an internal $3$-gon tile $\Delta$ in $P$ such that the puncture in $M^*$ of $\Delta$ lies in the interior of the triangle in $\calt$ formed by $\zg_1,\zg_2,\zg_3$, see Figure~\ref{fig:every-angle-traversed}.
This contradicts the fact that $\calt$ is a triangulation of $(S,M^*)$. 

Therefore, there is at least one $\calt$-angle which is not permissible. Suppose for the sake of contradiction that there are two $\calt$-angles which are not permissible. Let $\gamma_1$ be the arc of $\calt$ in common between these two $\calt$-angles. The arc $\gamma_1$ must cross an arc $\tau$ of $P$. Since there are no marked points of $M$ in the interior of the red triangle, since $M\subset \partial S$, it follows that $\tau$ must traverse one of the two $\calt$-angles, contradicting the hypothesis. 

Finally, if the triangle is self-folded, let $\gamma_1=\gamma_2$ be the sides that coincide, $\gamma_3$ the third side of the triangle, and $x$ the puncture in $M^*$ which is the endpoint of $\gamma_1= \gamma_2$ but not of $\gamma_3$. 
Then $x$ must be the red puncture of a 1-gon tile $\Delta$ of $\TiledSurface$, since $\gamma_3$ is permissible. 
It is then clear that the $\calt$-angle $(x, \gamma_1^\circ, \gamma_2^\circ)$ is not traversed by an arc of $P$, and the other two $\calt$-angles are traversed by the loop arc defining $\Delta$.  
\end{proof} 

The following is a corollary of Theorem~\ref{thm endoalg} and Lemma~\ref{lem:1-angle-not-traversed}.

\begin{corollary}
 \label{cor endoalg}
The endomorphism algebra $C$ is a gentle algebra of global dimension at most $2$.
\end{corollary}
\begin{proof}
By Theorem~\ref{thm endoalg}, $C$ is the gentle algebra associated to $(G(S), G(M^*), G(M^*)^*, G(\calt))$. Note that this tiling is obtained from $(S, M^*, \calt)$ by replacing each puncture by a boundary component and by ``cutting'' each triangle in $\calt$ at its unique (by Lemma~\ref{lem:1-angle-not-traversed}) non-permissible $\calt$-angle. Therefore, each tile in $(G(S), G(M^*), G(\calt))$ is a quadrilateral with exactly one side a boundary segment. 
This implies that there are no overlapping relations, meaning the global dimension of $C$ is at most $2$. 
\end{proof}
 
Before stating our main result, we make a short digression into cluster-tilting theory. 
Given any algebra $\zL$ of global dimension at most 2, it is natural to consider the bimodule $E=\Ext^2_{\zL}(D\zL,\zL)$, where $D\zL$ is the direct sum of the indecomposable injective $\zL$-modules and $\zL$ the direct sum of the indecomposable projectives.
Indeed, if $\zL$ is a tilted algebra, then the tensor algebra $T_\zL (E)$ is the corresponding cluster-tilted algebra,  see \cite{ABS}. In the general case,   the tensor algebra is a 2-Calabi-Yau tilted algebra, which means that it is the endomorphism algebra of a cluster-tilting object in the cluster category of $\zL$, see \cite{Amiot09}. In both cases, the construction is fundamental in the relation between representation theory and cluster algebras.

\smallskip

It turns out that this construction also applies to our situation. 

\begin{theorem}
 \label{thm:tensor}
Let $A$ be the tiling algebra of the tiled surface $\TiledSurface$. 
Let $T$ be a maximal almost rigid $A$-module, let $C=\End_A T$, and let $\calt$ be the corresponding permissible triangulation. 
Denote by $B$ the tiling algebra of the tiled surface $(S,M^*,(M^*)^*,\calt)$.
Then there is an isomorphism of algebras
\[B\cong T_C(\Ext_C^2(DC,C).\]
\end{theorem}

\begin{proof}
 Let $C=\kb Q_C/I_C$ and let $R_C$ be a minimal set of generating relations for $I_C$. Since $C$ is  gentle, by Corollary \ref{cor:end-alg-tilt}, every relation in $R_C$ is a path of length 2. Throughout the proof, we denote the tensor algebra  $T_C(\Ext_C^2(DC,C)$ by $\widetilde{C}$. By \cite[Theorem 2.4]{ABS}, the quiver $Q_{\widetilde C}$ of  $\widetilde C$ is obtained from $Q_C$ by adding one arrow $\ze(r)\colon t(r)\to s(r)$ for each relation $r=\za\zb\in R$, where $s(r)=s(\za)$ and $t(r)=t(\zb)$. Define a potential $\widetilde W=\sum_{r\in R} \ze(r) \,r$. 
 Then by \cite[\S 3.6]{Amiot-survey} (see also~\cite[Theorem 6.12]{Keller}), the algebra $\widetilde C$ is isomorphic to the Jacobian algebra of the quiver with potential $(Q_{\widetilde C},\widetilde W)$. 
 
 In our situation, the algebra $C$ is gentle of global dimension at most 2. Therefore every arrow of $C$ is in at most one relation, and consequently, every arrow of $\widetilde C$ appears in at most one term of the potential. Thus the set of relations for $\widetilde C$ is precisely   the set of paths of the form $\za\zb, \ze(r)\za$ and $\zb\ze(r)$, with $r=\za\zb\in R_C$.
 
 On the other hand, the quiver $Q_B$ of the tiling algebra $B$ has one vertex for each arc of the triangulation $\calt=\{\zg_1,\ldots,\zg_m\}$, which clearly is in bijection with the vertices of $Q_{\widetilde C}$, since these are given by the indecomposable summands of the MAR module $T=T_1\oplus\cdots\oplus T_m$. 
 
 Moreover, the  arrows of $Q_B$ are in bijection with the angles between the arcs in $\calt$. There are two types of angles, 
 the permissible angles, and
the non-permissible angles.
The permissible angles correspond bijectively to the arrows in $Q_C$, by Lemma \ref{lem:composition}. 

 We will now show that the non-permissible angles are in bijection with the arrows in $(Q_{\widetilde C})_1\setminus(Q_{C})_1$. Let $r=\za\zb\in R_C$ be a generating relation, and let $(x,\gamma_k^\circ, \gamma_j^\circ)$ and $(y,\gamma_j^{\circ\circ}, \gamma_i^\circ)$ be the angles corresponding to the arrows $\za$ and $\zb$, respectively.  Note that both angles are permissible. Since $\za\zb$ is a path in $Q_C$, the same arc $\zg_j$ appears in both angles, however, since $\za\zb\in I_C$, the respective ends $\zg_j^\circ,\zg_j^{\circ\circ}$ are not the same. In particular, the arcs $\zg_i,\zg_j,\zg_k$ form a triangle in $\calt$. It then follows from Lemma~\ref{lem:1-angle-not-traversed} that the third angle of this triangle is not permissible. Thus this angle corresponds to an arrow of the second type in $Q_B$.  

Conversely, take an arrow $\ze\colon i \to k$ in $Q_B$ corresponding to a non-permissible angle of $\calt$. This angle lies in a unique triangle of $\calt$ and by Lemma~\ref{lem:1-angle-not-traversed}, the other two angles of this triangle are permissible. So they correspond to arrows $\alpha\colon k \to j$ and $\beta\colon j \to i$ in $Q_C$. The corresponding composition of morphisms in $\textup{mod}\,{A}$ must be zero, since we have distinct end segments  
of $\gamma_j$ in both angles. Since $C= \End_A T$, it follows that $\alpha \beta \in R_C$. Hence, there is an arrow $\ze({\alpha \beta})$ in $Q_{\widetilde C}$. 

 We have thus defined two maps between the set of arrows of the form $\ze(r)$ in $Q_{\widetilde C}$ and the set of arrows corresponding to non-permissible angles in $Q_B$, and these maps are inverses of each other.
 This shows that there is an isomorphism between the quivers of the algebras $\widetilde C$ and $B$ that induces a bijection between the defining relations of the algebras. Hence the theorem is proved.
\end{proof}

 Since $C$ is of global dimension at most 2, it follows that the algebra $B$ is the endomorphism algebra of a cluster-tilting object in Amiot's generalized cluster category $\calc_C$ of $C$, see \cite{Amiot09}. 
 
\begin{corollary}
 The algebra $B$ is 2-Calabi-Yau tilted and $B\cong \End_{\calc_C} C$.
\end{corollary}

\section{Examples}\label{sect examples}
We finish by giving two examples to illustrate the constructions in this article.

We note that in both examples, string modules are completely determined by their composition series. This explains the notation we use for string modules in this section.

\subsection{First example}\label{ex 7.5}
 
\begin{figure}
\begin{center}
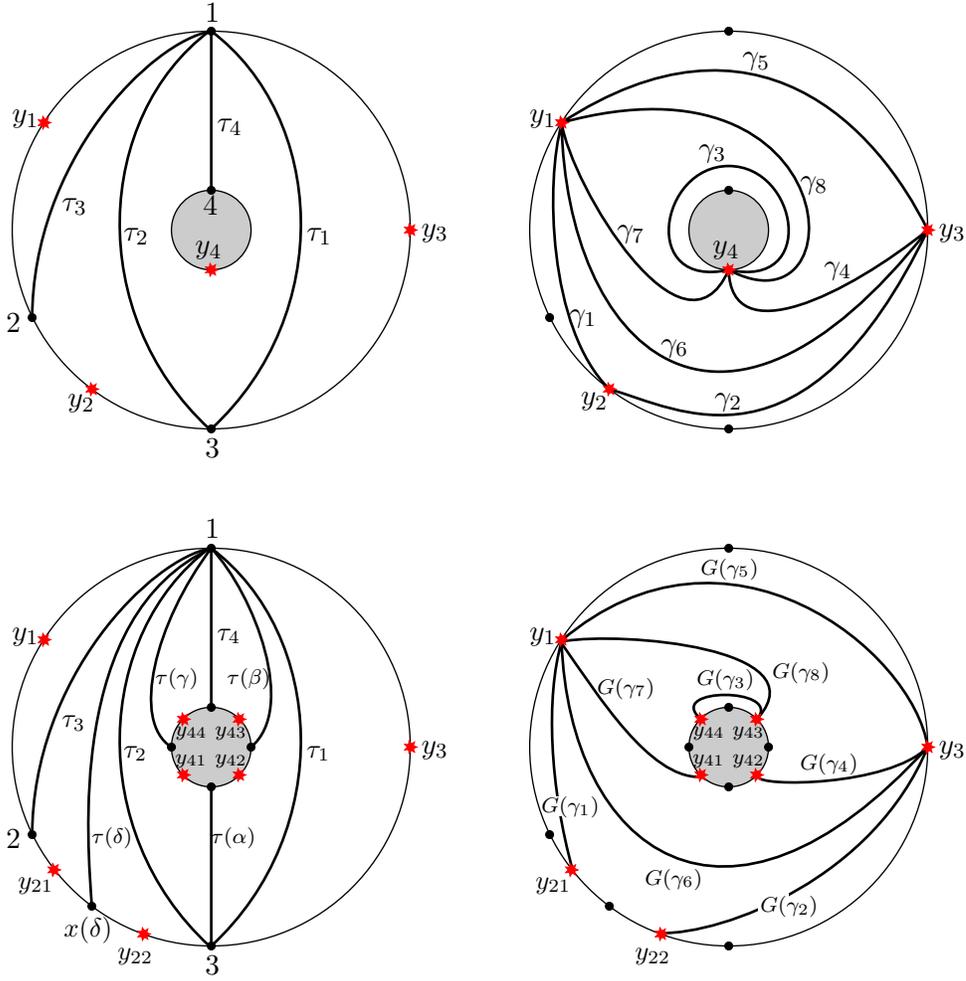
\caption{Illustration of Example~\ref{ex 7.5}. The tiled surface  $\TiledSurface$ of the algebra $A$ (top left), the triangulation $\calt$ of the maximal almost rigid $A$-module $T$ (top right), the  tiled surface  $G\TiledSurface$ of the algebra $\Abar$ (bottom left), and the image $G(\calt)$ of the triangulation $\calt$ under the map $G$ (bottom right).}
\label{fig ex 7.5}
\end{center}
\end{figure}

 Let $A=\kb Q/I$ with $Q$ the quiver 
 $\vcenter{\vbox{\xymatrix@R1pt{&4\ar[dr]^\zg\\
 1\ar[ru]^\zb&&2\ar[ll]^\za\ar[r]^\zd&3
}}}$ and  $I=\langle\za\zb, \zg\za\rangle$. 
The tiled surface $\TiledSurface$ of $A$ is shown in the top left picture in Figure~\ref{fig ex 7.5}. 
The Auslander-Reiten quiver of $A$ is given below in \eqref{eq:ex 7.5:AR quiver}, 
where the two copies of the simple module~$1$ should be identified. The required summands are typeset between vertical bars. 
\begin{equation}
\label{eq:ex 7.5:AR quiver}
\xymatrix@R10pt@C10pt{&&&&
{\left|\begin{smallmatrix}
 1\\4\\2\\3
\end{smallmatrix}\right|} \ar[dr]
\\
&&&{\begin{smallmatrix}
 4\\2\\3
\end{smallmatrix}} \ar[dr]\ar[ur]
&&{\begin{smallmatrix}
1\\ 4\\2
\end{smallmatrix}} \ar[dr]
\\
{\left|\begin{smallmatrix}
1
\end{smallmatrix}\right|} \ar[dr]
&&{\begin{smallmatrix}
 2\\3
\end{smallmatrix}} \ar[dr]\ar[ur]
&&{\begin{smallmatrix}
 4\\2
\end{smallmatrix}} \ar[dr]\ar[ur]
&&{\begin{smallmatrix}
 1\\4
 \end{smallmatrix}} \ar[dr]
 \\
 &
 {\begin{smallmatrix}
2\\1\ 3
\end{smallmatrix}} \ar[dr]\ar[ur]
&&{\begin{smallmatrix}
2
\end{smallmatrix}} \ar[ur]
&&{\left|\begin{smallmatrix}
4
\end{smallmatrix}\right|} \ar[ur]
&&{\left|\begin{smallmatrix}
1
\end{smallmatrix}\right|} \ar[dr]
\\
{\left|\begin{smallmatrix}
3
\end{smallmatrix}\right|} \ar[ur]
&&{\left|\begin{smallmatrix}
2\\1
\end{smallmatrix}\right|} \ar[ur]
 &&&&&&\ldots}
\end{equation}

Let $T$ be the maximal almost rigid $A$-module
\[T={\begin{smallmatrix}
3
\end{smallmatrix}}\oplus
{\begin{smallmatrix}
2\\1
\end{smallmatrix}}
\oplus
{\begin{smallmatrix}
4
\end{smallmatrix}}
\oplus
{\begin{smallmatrix}
1
\end{smallmatrix}}\oplus
{\begin{smallmatrix}
1\\4\\2\\3
\end{smallmatrix}}\oplus
{\begin{smallmatrix}
2\\1\ 3
\end{smallmatrix}}\oplus
{\begin{smallmatrix}
2\\3
\end{smallmatrix}}\oplus
{\begin{smallmatrix}
4\\2\\3
\end{smallmatrix}}.
\] 
Note that $T$ is the unique MAR module $M_\proj$ that contains all projectives. 
We see directly from the Auslander-Reiten quiver that the endomorphism algebra $\End_A T$ is given by the quiver \eqref{eq:ex75: quiver} below,  bound by the relations 
$fa=gb=ce=hd=0$. 
\begin{equation}\label{eq:ex75: quiver}
 \xymatrix@R10pt@C10pt{&&&&5\ar[ld]_a\\
&&&8\ar[ld]_b&&\\
4\ar@/^10pt/[rrrruu]^f&&7\ar[ld]_c\\
&6\ar[ld]_d\ar[lu]_e&&&&3\ar[lluu]_g\\
1&&2\ar[lu]_h}
\end{equation}

The maximal almost rigid module $T$ induces a  triangulation $\calt$ of $(S,M^*)$ which is shown in the top right picture in Figure~\ref{fig ex 7.5}. The arcs $\zg_i$ of the triangulation are determined by the dimension vectors of the direct summands of $T$. For example, the arc $\zg_1$ corresponds to the direct summand $T_1=3$, because $\zg_1$ crosses the arc $\tau_3$ exactly once and it doesn't cross any other arc of the tiling $P$.

The algebra $\Abar$ is given by the quiver $\Qbar= \vcenter{\vbox{\xymatrix@R1pt@C=20pt{&&4\ar[dr]^{\zg_a}\\
&v_\zb\ar[ru]^{\zb_b} && v_\zg\ar[rd]^{\zg_b}\\
 1\ar[ru]^{\zb_a}&&v_\za\ar[ll]^{\za_b}&&2\ar[ll]^{\za_a}\ar[r]^{\zd_a}&v_\zd\ar[r]^{\zd_b}&3
}}}$ and the ideal $\Ibar=\langle\za_b\zb_a, \zg_b\za_a\rangle$.
By Proposition~\ref{prop 7.2}, its tiled surface is equal to $G\TiledSurface$, which is shown in the bottom left picture of Figure~\ref{fig ex 7.5}. This also illustrates the action of $G$ on $\TiledSurface$. 

The functor $G$ maps the $A$-module $T$ to an $\Abar $-module $G(T)$ as described in section~\ref{sect 7.3}. The corresponding image $G(\calt)$ of the triangulation $\calt$ is shown in the bottom right picture in Figure~\ref{fig ex 7.5}. 
The set of red points $G(M^*)$ consists of the points $y_1,y_{21},y_{22}, y_3,y_{41},y_{42},y_{43},y_{44}$. Note that each tile in $G(\calt)$ has a boundary edge.

We obtain a new tiled surface $(G(S),G(M^*),G(M^*)^*,G(\calt))$, where $G(M^*)^*$ has exactly one point on each boundary component and no punctures. The  corresponding gentle algebra is isomorphic to the endomorphism algebra $\End_A T $. Indeed the quiver of the tiling algebra is the same as the one shown in \eqref{eq:ex75: quiver}. There are four generating relations, each corresponding to one of the four tiles that have an edge on the boundary. For example, the tile bounded by $G(\zg_4),G(\zg_5),$ and $G(\zg_8)$ induces the relation 
$0=fa=\xymatrix{4\ar[r]^f&5\ar[r]^a&8}$.

The quiver of the tensor algebra $B=T_C(\Ext^2_C(DC,C))$ is drawn in \eqref{eq:ex75: quiverB} below. 
\begin{equation}\label{eq:ex75: quiverB}
\xymatrix@R15pt@C15pt{&&&&5\ar[ld]^a\\
&&&8\ar[ld]^b\ar[dlll]_{\ze(fa)}&&\\
4\ar@/^15pt/[rrrruu]^f\ar[rr]_{\ze(ce)}&&7\ar[ld]^c\ar[rrrd]_{\ze(gb)}\\
&6\ar[ld]_d\ar[lu]^e&&&&3\ar[lluu]_g\\
1\ar[rr]_{\ze(hd)}&&2\ar[lu]_h}
\end{equation}

For a comparison of the permissible triangulations associated to $M_\proj$ and $M_\inj$ (defined in Section~\ref{sec:Mproj}), see Figure~\ref{fig:triangulations-proj-inj-mars}.

\begin{figure}[ht!]
\centering 
\scalebox{0.8}{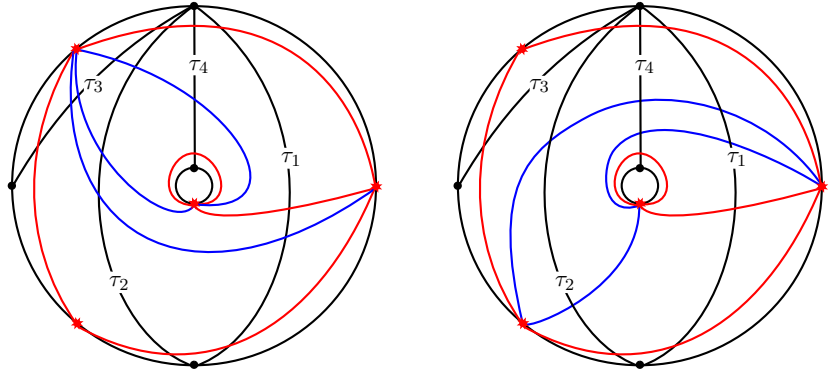}
\caption{Left: The triangulation corresponding to $M_\proj$. Right: The triangulation corresponding to $M_\inj$. In both pictures, the required summands of the MAR module are drawn in red, and the remaining summands are drawn in blue. The tiling $P$ of the surface is drawn in black.}
\label{fig:triangulations-proj-inj-mars}
\end{figure}

\subsection{Second example: The Orpheus algebra} \label{sect Orpheus Algebra}

Let $A$ be the gentle algebra given by the quiver $\xymatrix{1\ar@<2pt>[r]^{\za_1}&2\ar@<2pt>[r]^{\zb_1}\ar@<2pt>[l]^{\za_2}&3\ar@<2pt>[l]^{\zb_2}\ar@<2pt>[r]^{\zg_1}&4\ar@<2pt>[l]^{\zg_2}}$ bound by the relations $\za_1\za_2, \,\za_2\za_1, 
\,
\zb_1\zb_2, \,\zb_2\zb_1,\,
\zg_1\zg_2,\, \zg_2\zg_1$. Because of the relations, a path is zero as soon as it changes direction in the quiver. So the nonzero paths go from left to right or from right to left. Because of this ``never turn around'' property, this algebra is called the \emph{Orpheus algebra}.

The Auslander-Reiten quiver is given in Figure~\ref{fig orpheus ar}, where the modules with the same labels should be identified.  The corresponding tiled surface is shown in the top left picture of  Figure~\ref{fig orpheus}.

The top right picture of Figure~\ref{fig orpheus} shows the triangulation corresponding to the unique MAR module 
\[T=M_{\textup{proj}} =
{\begin{smallmatrix}
 1\\2\\3\\4
\end{smallmatrix}}
\oplus
{\begin{smallmatrix}
 1
\end{smallmatrix}}\oplus
{\begin{smallmatrix}
 2\\3\\4
\end{smallmatrix}} 
\oplus
{\begin{smallmatrix}
2\\1\, 3\\\  \ 4
\end{smallmatrix}} 
\oplus
{\begin{smallmatrix}
2\\1
\end{smallmatrix}}  
\oplus
{\begin{smallmatrix}
 3\\4
\end{smallmatrix}}
\oplus
{\begin{smallmatrix}
3\\2\,4\\1\ \ 
\end{smallmatrix}}
\oplus
{\begin{smallmatrix}
3\\2\\1 
\end{smallmatrix}} 
\oplus
{\begin{smallmatrix}
 4
\end{smallmatrix}}
\oplus{\begin{smallmatrix}
 4\\3\\2\\1
\end{smallmatrix}}
\]
that contains all projective summands. The projective modules correspond to the four arcs whose endpoints are the two boundary points. The remaining arcs are summands of radicals. There are precisely four required summands: the projective-injectives $P(1), P(4)$ which correspond to the leftmost and the rightmost arc in the triangulation, and the simples $S(1),S(4)$ which correspond to the short arcs from the boundary points to the nearest puncture. 

In the Auslander-Reiten quiver the required summands are typeset between vertical bars. The circles in the Auslander-Reiten quiver indicate holes, that is the module to the left of a circle is injective and the module to the right of a circle is projective.

The endomorphism algebra $C=\End_A T$ is given by the quiver
\[\xymatrix@R10pt{
&{\begin{smallmatrix}
1
\end{smallmatrix}}  
\ar[ld] _\za
\ar@{<-}[rd] ^\zb
&&
{\begin{smallmatrix}
 2\\1
\end{smallmatrix}}
\ar[ld]_\zg 
\ar@{<-}[rd] ^\zd
&&
{\begin{smallmatrix}
3\\2\\1 
\end{smallmatrix}} 
\ar[ld]_\zs
\ar@{<-}[rd] ^\tau
\\
{\begin{smallmatrix}
 1\\2\\3\\4
\end{smallmatrix}}
&&
{\begin{smallmatrix}
2\\1\, 3\\\  \ 4
\end{smallmatrix}} 
&&
{\begin{smallmatrix}
3\\2\,4\\1\ \ 
\end{smallmatrix}}
&&
{\begin{smallmatrix}
 4\\3\\2\\1
\end{smallmatrix}} 
\\ 
&{\begin{smallmatrix}
 2\\3\\4
\end{smallmatrix}}  
\ar@{<-}[lu] ^\za
\ar[ru] _\zb
&&
{\begin{smallmatrix}
 3\\ 4
\end{smallmatrix}}
\ar@{<-}[lu]^\zg 
\ar[ru]_\zd
&&
{\begin{smallmatrix}
 4
\end{smallmatrix}}
\ar@{<-}[lu] ^\zs
\ar[ru] _
\tau
}
\]
bound by the relations $\za^2,\zb^2,\zg^2,\zd^2,\zs^2$ and $\tau^2$. Therefore the tensor algebra $B$ is given by the quiver
\[\xymatrix@R10pt{
&{\begin{smallmatrix}
1
\end{smallmatrix}}  
\ar[ld] _\za
\ar@{<-}[rd] ^\zb
\ar@{<-}@<-2pt>[dd]_{\za}
\ar@{<-}@<2pt>@{<-}[dd]^{\zb}
&&
{\begin{smallmatrix}
 2\\1
\end{smallmatrix}}
\ar[ld]_\zg 
\ar@{<-}[rd] ^\zd
\ar@{<-}@<-2pt>[dd]_{\zg}
\ar@{<-}@<2pt>@{<-}[dd]^{\zd}
&&
{\begin{smallmatrix}
3\\2\\1 
\end{smallmatrix}} 
\ar[ld]_\zs
\ar@{<-}[rd] ^\tau
\ar@{<-}@<-2pt>[dd]_{\zs}
\ar@{<-}@<2pt>@{<-}[dd]^{\tau}\\
{\begin{smallmatrix}
 1\\2\\3\\4
\end{smallmatrix}}
&&
{\begin{smallmatrix}
2\\1\, 3\\\  \ 4
\end{smallmatrix}} 
&&
{\begin{smallmatrix}
3\\2\,4\\1\ \ 
\end{smallmatrix}}
&&
{\begin{smallmatrix}
 4\\3\\2\\1
\end{smallmatrix}} 
\\ 
&{\begin{smallmatrix}
 2\\3\\4
\end{smallmatrix}}  
\ar@{<-}[lu] ^\za
\ar[ru] _\zb
&&
{\begin{smallmatrix}
 3\\ 4
\end{smallmatrix}}
\ar@{<-}[lu]^\zg 
\ar[ru]_\zd
&&
{\begin{smallmatrix}
 4
\end{smallmatrix}}
\ar@{<-}[lu] ^\zs
\ar[ru]_
\tau
}
\]
bound by the relations $\za^2,\zb^2,\zg^2,\zd^2,\zs^2$ and $\tau^2$.
It is rather curious that, in this example, the quiver with relations of $C$ is isomorphic to the quiver with relations of $\Abar.$

The bottom pictures of Figure \ref{fig orpheus}  show the tiled surface of the algebra $\Abar$ of section~\ref{sect 6} as well as the image of $M_\proj$ under the map $G$. 

 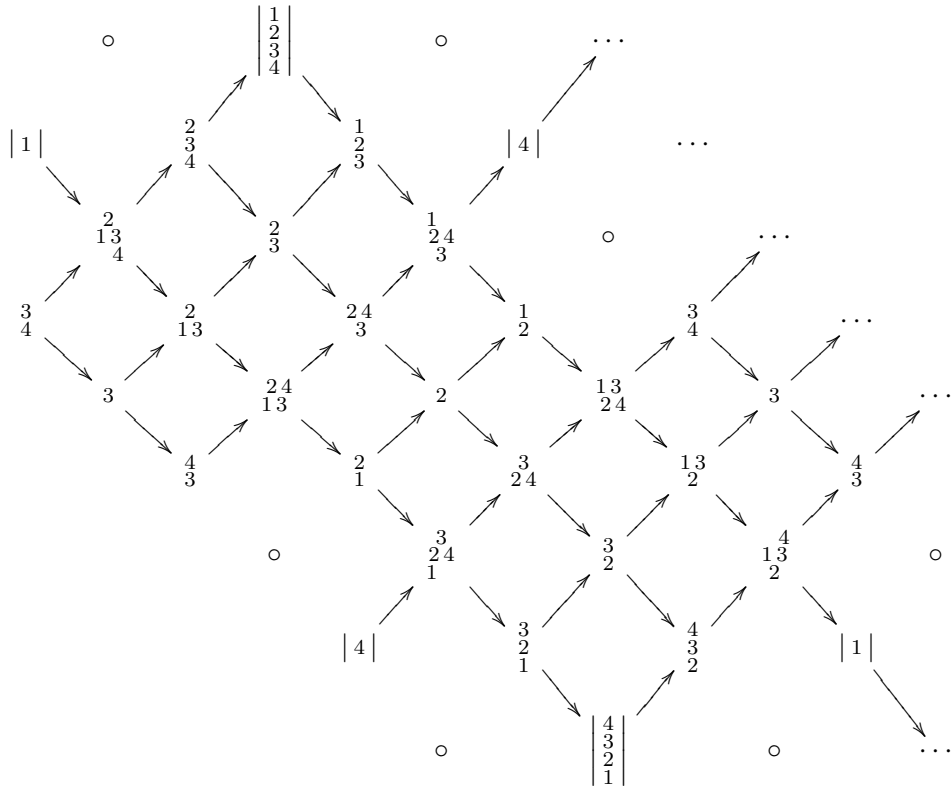
\begin{figure}
\begin{center}
\[\xymatrix@R10pt@C10pt{&\circ&&
\left|{\begin{smallmatrix}
 1\\2\\3\\4
\end{smallmatrix}}\right| \ar[rd] &&\circ&&\cdots
\\
\left|{\begin{smallmatrix}
 1
\end{smallmatrix}}\right| \ar[rd]
&&{\begin{smallmatrix}
 2\\3\\4
\end{smallmatrix}} \ar[rd]\ar[ru]
&&{\begin{smallmatrix}
 1\\2\\3
\end{smallmatrix}} \ar[rd]
&&
\left|{\begin{smallmatrix}
 4
\end{smallmatrix}}\right| \ar[ru] &&\cdots
\\
&{\begin{smallmatrix}
2\\1\, 3\\\  \ 4
\end{smallmatrix}} \ar[rd]\ar[ru]
&
&{\begin{smallmatrix}
 2\\3
\end{smallmatrix}} \ar[rd]\ar[ru]
&
&{\begin{smallmatrix}
 1\  \ \\2\, 4\\3
\end{smallmatrix}} \ar[rd]\ar[ru]
&&\circ
&&\cdots
\\
{\begin{smallmatrix}
3\\4
\end{smallmatrix}} \ar[rd]\ar[ru]&&{\begin{smallmatrix}
2\\1\, 3\\
\end{smallmatrix}} \ar[rd]\ar[ru]
&&{\begin{smallmatrix}
2\, 4\\ \, 3\\
\end{smallmatrix}} \ar[rd]\ar[ru]
&&{\begin{smallmatrix}
1\\2
\end{smallmatrix}} \ar[rd]
&&{\begin{smallmatrix}
 3\\4
\end{smallmatrix}} \ar[rd]\ar[ru]
&&\cdots
\\
&{\begin{smallmatrix}
3
\end{smallmatrix}}\ar[ru]\ar[rd]&&{\begin{smallmatrix}
\ 2\,4\\ 1\, 3\\
\end{smallmatrix}} \ar[rd]\ar[ru]
&&{\begin{smallmatrix}
2
\end{smallmatrix}} \ar[rd]\ar[ru]
&&{\begin{smallmatrix}
 1\, 3\\\  2\,4\\
\end{smallmatrix}} \ar[rd]\ar[ru]
&&{\begin{smallmatrix}
3
\end{smallmatrix}} \ar[rd]\ar[ru]
&&\cdots
\\
&&{\begin{smallmatrix}
4\\3
\end{smallmatrix}}\ar[ru]&&{\begin{smallmatrix}
2\\1
\end{smallmatrix}} \ar[rd]\ar[ru]
&&{\begin{smallmatrix}
3\\ 2\,4
\end{smallmatrix}} \ar[rd]\ar[ru]
&&{\begin{smallmatrix}
1\,3\\2
\end{smallmatrix}} \ar[rd]\ar[ru]
&&{\begin{smallmatrix}
4\\3
\end{smallmatrix}} \ar[ru]
&&
\\
&&&\circ
&&{\begin{smallmatrix}
3\\2\,4\\1\ \ 
\end{smallmatrix}} \ar[rd]\ar[ru]
&&{\begin{smallmatrix}
3\\2
\end{smallmatrix}} \ar[rd]\ar[ru]
&&{\begin{smallmatrix}
\ \ 4\\1\,3\\2
\end{smallmatrix}} \ar[rd]\ar[ru]
&&\circ
\\
&&&&\left|{\begin{smallmatrix}
4
\end{smallmatrix}} \right|\ar[ru]
&&{\begin{smallmatrix}
3\\2\\1 
\end{smallmatrix}} \ar[rd]\ar[ru]
&&{\begin{smallmatrix}
4\\3\\2 
\end{smallmatrix}} \ar[ru]
&&\left|{\begin{smallmatrix}
1 
\end{smallmatrix}}\right| \ar[rd]
\\
&&&&&\circ
&&\left|{\begin{smallmatrix}
4\\3\\2\\1 
\end{smallmatrix}}\right| \ar[ru] 
 &&\circ&&\cdots
}\]
\caption{The Auslander-Reiten quiver of the Orpheus algebra. Modules with the same labels need to be identified. The required summands are typeset between vertical bars. The circles $\circ$ represent holes.}
\label{fig orpheus ar}
\end{center}
\end{figure}

\begin{figure}
\begin{center}
\scalebox{.8}{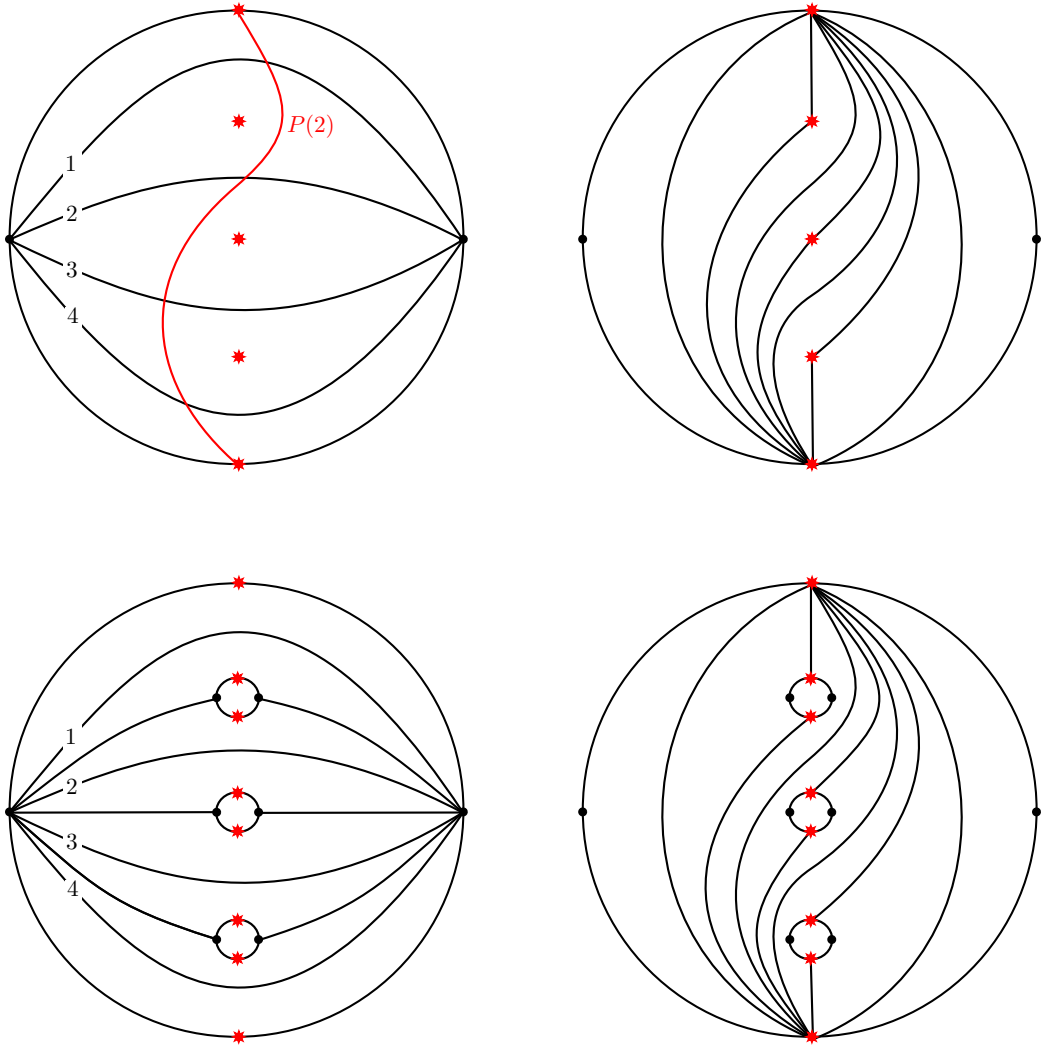} 
\caption{Illustration of Example \ref{sect Orpheus Algebra}. The tiled surface  $\TiledSurface$ of the orpheus algebra $A$ in black and the arc corresponding to the projective module
$\begin{smallmatrix}
2\\1\, 3\\\  \ 4
\end{smallmatrix}$
 in red (top left), the triangulation  $\calt$ corresponding to the MAR module $M_{\textup{proj}}$ (top right), the  tiled surface  $G\TiledSurface$ of the algebra $\Abar$ (bottom left), and the image $G(\calt)$ of the triangulation $\calt$ under the map $G$ (bottom right).}
\label{fig orpheus}
\end{center}
\end{figure}

\subsection*{Acknowledgements} 
The authors are grateful to Thomas Br\"{u}stle, Alastair King, Yann Palu, Matthew Pressland, and Yadira Valdivieso-D\'iaz for helpful discussions. 
The authors would like to thank the Isaac Newton Institute for Mathematical Sciences for support and hospitality during the programme Cluster Algebras and Representation Theory where work on this paper was undertaken. This work was supported by: EPSRC Grant Number EP/R014604/1. RCS acknowledges the financial support of European Union’s Horizon 2020 research and innovation programme through the Marie Sk\l odowska-Curie Individual Fellowship grant 838706. RS was supported by the NSF grants  DMS-2054561 and DMS-2348909.

\printbibliography

\end{document}

%% file: figbigonmonogon.pdf_tex
\begingroup%
  \makeatletter%
  \providecommand\color[2][]{%
    \errmessage{(Inkscape) Color is used for the text in Inkscape, but the package 'color.sty' is not loaded}%
    \renewcommand\color[2][]{}%
  }%
  \providecommand\transparent[1]{%
    \errmessage{(Inkscape) Transparency is used (non-zero) for the text in Inkscape, but the package 'transparent.sty' is not loaded}%
    \renewcommand\transparent[1]{}%
  }%
  \providecommand\rotatebox[2]{#2}%
  \newcommand*\fsize{\dimexpr\f@size pt\relax}%
  \newcommand*\lineheight[1]{\fontsize{\fsize}{#1\fsize}\selectfont}%
  \ifx\svgwidth\undefined%
    \setlength{\unitlength}{255.00001039bp}%
    \ifx\svgscale\undefined%
      \relax%
    \else%
      \setlength{\unitlength}{\unitlength * \real{\svgscale}}%
    \fi%
  \else%
    \setlength{\unitlength}{\svgwidth}%
  \fi%
  \global\let\svgwidth\undefined%
  \global\let\svgscale\undefined%
  \makeatother%
  \begin{picture}(1,0.13519593)%
    \lineheight{1}%
    \setlength\tabcolsep{0pt}%
    \put(0,0){\includegraphics[width=\unitlength,page=1]{figbigonmonogon.pdf}}%
    \put(0.18823541,0.10381861){\makebox(0,0)[lt]{\lineheight{1.25}\smash{\begin{tabular}[t]{l}$\tau$\end{tabular}}}}%
    \put(0,0){\includegraphics[width=\unitlength,page=2]{figbigonmonogon.pdf}}%
    \put(0.72941179,0.10381861){\makebox(0,0)[lt]{\lineheight{1.25}\smash{\begin{tabular}[t]{l}$\tau$\end{tabular}}}}%
    \put(0,0){\includegraphics[width=\unitlength,page=3]{figbigonmonogon.pdf}}%
  \end{picture}%
\endgroup%

%% file: figtiles.pdf_tex
\begingroup%
  \makeatletter%
  \providecommand\color[2][]{%
    \errmessage{(Inkscape) Color is used for the text in Inkscape, but the package 'color.sty' is not loaded}%
    \renewcommand\color[2][]{}%
  }%
  \providecommand\transparent[1]{%
    \errmessage{(Inkscape) Transparency is used (non-zero) for the text in Inkscape, but the package 'transparent.sty' is not loaded}%
    \renewcommand\transparent[1]{}%
  }%
  \providecommand\rotatebox[2]{#2}%
  \newcommand*\fsize{\dimexpr\f@size pt\relax}%
  \newcommand*\lineheight[1]{\fontsize{\fsize}{#1\fsize}\selectfont}%
  \ifx\svgwidth\undefined%
    \setlength{\unitlength}{228.34128709bp}%
    \ifx\svgscale\undefined%
      \relax%
    \else%
      \setlength{\unitlength}{\unitlength * \real{\svgscale}}%
    \fi%
  \else%
    \setlength{\unitlength}{\svgwidth}%
  \fi%
  \global\let\svgwidth\undefined%
  \global\let\svgscale\undefined%
  \makeatother%
  \begin{picture}(1,0.91944465)%
    \lineheight{1}%
    \setlength\tabcolsep{0pt}%
    \put(0,0){\includegraphics[width=\unitlength,page=1]{figtiles.pdf}}%
  \end{picture}%
\endgroup%

%% file: figfan.pdf_tex
\begingroup%
  \makeatletter%
  \providecommand\color[2][]{%
    \errmessage{(Inkscape) Color is used for the text in Inkscape, but the package 'color.sty' is not loaded}%
    \renewcommand\color[2][]{}%
  }%
  \providecommand\transparent[1]{%
    \errmessage{(Inkscape) Transparency is used (non-zero) for the text in Inkscape, but the package 'transparent.sty' is not loaded}%
    \renewcommand\transparent[1]{}%
  }%
  \providecommand\rotatebox[2]{#2}%
  \newcommand*\fsize{\dimexpr\f@size pt\relax}%
  \newcommand*\lineheight[1]{\fontsize{\fsize}{#1\fsize}\selectfont}%
  \ifx\svgwidth\undefined%
    \setlength{\unitlength}{301.05403682bp}%
    \ifx\svgscale\undefined%
      \relax%
    \else%
      \setlength{\unitlength}{\unitlength * \real{\svgscale}}%
    \fi%
  \else%
    \setlength{\unitlength}{\svgwidth}%
  \fi%
  \global\let\svgwidth\undefined%
  \global\let\svgscale\undefined%
  \makeatother%
  \begin{picture}(1,0.26559334)%
    \lineheight{1}%
    \setlength\tabcolsep{0pt}%
    \put(0,0){\includegraphics[width=\unitlength,page=1]{figfan.pdf}}%
    \put(0.11359696,0.03574999){\makebox(0,0)[lt]{\lineheight{1.25}\smash{\begin{tabular}[t]{l}$x$\end{tabular}}}}%
    \put(0,0){\includegraphics[width=\unitlength,page=2]{figfan.pdf}}%
    \put(0.23566812,0.1005224){\makebox(0,0)[lt]{\lineheight{1.25}\smash{\begin{tabular}[t]{l}$\tau_4$\end{tabular}}}}%
    \put(0,0){\includegraphics[width=\unitlength,page=3]{figfan.pdf}}%
    \put(0.18584316,0.16031232){\makebox(0,0)[lt]{\lineheight{1.25}\smash{\begin{tabular}[t]{l}$\tau_3$\end{tabular}}}}%
    \put(0.08619325,0.15532982){\makebox(0,0)[lt]{\lineheight{1.25}\smash{\begin{tabular}[t]{l}$\tau_2$\end{tabular}}}}%
    \put(0.02142084,0.13041733){\makebox(0,0)[lt]{\lineheight{1.25}\smash{\begin{tabular}[t]{l}$\tau_1$\end{tabular}}}}%
    \put(0,0){\includegraphics[width=\unitlength,page=4]{figfan.pdf}}%
    \put(0.84851435,0.13041733){\makebox(0,0)[lt]{\lineheight{1.25}\smash{\begin{tabular}[t]{l}$\tau_2$\end{tabular}}}}%
    \put(0.61184633,0.02578513){\makebox(0,0)[lt]{\lineheight{1.25}\smash{\begin{tabular}[t]{l}$x$\end{tabular}}}}%
    \put(0.61184633,0.22508474){\makebox(0,0)[lt]{\lineheight{1.25}\smash{\begin{tabular}[t]{l}$x$\end{tabular}}}}%
    \put(0.85598849,0.22508474){\makebox(0,0)[lt]{\lineheight{1.25}\smash{\begin{tabular}[t]{l}$x$\end{tabular}}}}%
    \put(0.85598849,0.02578513){\makebox(0,0)[lt]{\lineheight{1.25}\smash{\begin{tabular}[t]{l}$x$\end{tabular}}}}%
    \put(0,0){\includegraphics[width=\unitlength,page=5]{figfan.pdf}}%
    \put(0.71896952,0.13041733){\makebox(0,0)[lt]{\lineheight{1.25}\smash{\begin{tabular}[t]{l}$\tau_3$\end{tabular}}}}%
    \put(0.60935482,0.13041733){\makebox(0,0)[lt]{\lineheight{1.25}\smash{\begin{tabular}[t]{l}$\tau_2$\end{tabular}}}}%
    \put(0.73391708,0.24003216){\makebox(0,0)[lt]{\lineheight{1.25}\smash{\begin{tabular}[t]{l}$\tau_1$\end{tabular}}}}%
    \put(0.73391708,0.00585507){\makebox(0,0)[lt]{\lineheight{1.25}\smash{\begin{tabular}[t]{l}$\tau_1$\end{tabular}}}}%
  \end{picture}%
\endgroup%

%% file: figarcsegments.pdf_tex
\begingroup%
  \makeatletter%
  \providecommand\color[2][]{%
    \errmessage{(Inkscape) Color is used for the text in Inkscape, but the package 'color.sty' is not loaded}%
    \renewcommand\color[2][]{}%
  }%
  \providecommand\transparent[1]{%
    \errmessage{(Inkscape) Transparency is used (non-zero) for the text in Inkscape, but the package 'transparent.sty' is not loaded}%
    \renewcommand\transparent[1]{}%
  }%
  \providecommand\rotatebox[2]{#2}%
  \newcommand*\fsize{\dimexpr\f@size pt\relax}%
  \newcommand*\lineheight[1]{\fontsize{\fsize}{#1\fsize}\selectfont}%
  \ifx\svgwidth\undefined%
    \setlength{\unitlength}{250.69722843bp}%
    \ifx\svgscale\undefined%
      \relax%
    \else%
      \setlength{\unitlength}{\unitlength * \real{\svgscale}}%
    \fi%
  \else%
    \setlength{\unitlength}{\svgwidth}%
  \fi%
  \global\let\svgwidth\undefined%
  \global\let\svgscale\undefined%
  \makeatother%
  \begin{picture}(1,0.60690716)%
    \lineheight{1}%
    \setlength\tabcolsep{0pt}%
    \put(0,0){\includegraphics[width=\unitlength,page=1]{figarcsegments.pdf}}%
    \put(0.66788012,0.46300629){\makebox(0,0)[lt]{\lineheight{1.25}\smash{\begin{tabular}[t]{l}$\tau_{\ell+1}$\end{tabular}}}}%
    \put(0.30888133,0.46300629){\makebox(0,0)[lt]{\lineheight{1.25}\smash{\begin{tabular}[t]{l}$\tau_2$\end{tabular}}}}%
    \put(0.15929849,0.46300629){\makebox(0,0)[lt]{\lineheight{1.25}\smash{\begin{tabular}[t]{l}$\tau_1$\end{tabular}}}}%
    \put(0.09049041,0.352315){\makebox(0,0)[lt]{\lineheight{1.25}\smash{\begin{tabular}[t]{l}$\zg^0$\end{tabular}}}}%
    \put(0.68283837,0.33436505){\makebox(0,0)[lt]{\lineheight{1.25}\smash{\begin{tabular}[t]{l}$\zg^{\ell+1}$\end{tabular}}}}%
    \put(0.2101567,0.352315){\makebox(0,0)[lt]{\lineheight{1.25}\smash{\begin{tabular}[t]{l}$\zg^1$\end{tabular}}}}%
    \put(0,0){\includegraphics[width=\unitlength,page=2]{figarcsegments.pdf}}%
  \end{picture}%
\endgroup%

%% file: figpermissiblesegment.pdf_tex
\begingroup%
  \makeatletter%
  \providecommand\color[2][]{%
    \errmessage{(Inkscape) Color is used for the text in Inkscape, but the package 'color.sty' is not loaded}%
    \renewcommand\color[2][]{}%
  }%
  \providecommand\transparent[1]{%
    \errmessage{(Inkscape) Transparency is used (non-zero) for the text in Inkscape, but the package 'transparent.sty' is not loaded}%
    \renewcommand\transparent[1]{}%
  }%
  \providecommand\rotatebox[2]{#2}%
  \newcommand*\fsize{\dimexpr\f@size pt\relax}%
  \newcommand*\lineheight[1]{\fontsize{\fsize}{#1\fsize}\selectfont}%
  \ifx\svgwidth\undefined%
    \setlength{\unitlength}{152.93533717bp}%
    \ifx\svgscale\undefined%
      \relax%
    \else%
      \setlength{\unitlength}{\unitlength * \real{\svgscale}}%
    \fi%
  \else%
    \setlength{\unitlength}{\svgwidth}%
  \fi%
  \global\let\svgwidth\undefined%
  \global\let\svgscale\undefined%
  \makeatother%
  \begin{picture}(1,0.8854651)%
    \lineheight{1}%
    \setlength\tabcolsep{0pt}%
    \put(0,0){\includegraphics[width=\unitlength,page=1]{figpermissiblesegment.pdf}}%
    \put(0.36582267,0.25647178){\makebox(0,0)[lt]{\lineheight{1.25}\smash{\begin{tabular}[t]{l}$\zg^i$\end{tabular}}}}%
    \put(0.54727186,0.54090578){\makebox(0,0)[lt]{\lineheight{1.25}\smash{\begin{tabular}[t]{l}$\tau_{i+1}$\end{tabular}}}}%
    \put(0.15494917,0.54090578){\makebox(0,0)[lt]{\lineheight{1.25}\smash{\begin{tabular}[t]{l}$\tau_{i}$\end{tabular}}}}%
    \put(0,0){\includegraphics[width=\unitlength,page=2]{figpermissiblesegment.pdf}}%
    \put(0.3707267,0.83514778){\makebox(0,0)[lt]{\lineheight{1.25}\smash{\begin{tabular}[t]{l}$x_i$\end{tabular}}}}%
  \end{picture}%
\endgroup%

%% file: figconnector.pdf_tex
\begingroup%
  \makeatletter%
  \providecommand\color[2][]{%
    \errmessage{(Inkscape) Color is used for the text in Inkscape, but the package 'color.sty' is not loaded}%
    \renewcommand\color[2][]{}%
  }%
  \providecommand\transparent[1]{%
    \errmessage{(Inkscape) Transparency is used (non-zero) for the text in Inkscape, but the package 'transparent.sty' is not loaded}%
    \renewcommand\transparent[1]{}%
  }%
  \providecommand\rotatebox[2]{#2}%
  \newcommand*\fsize{\dimexpr\f@size pt\relax}%
  \newcommand*\lineheight[1]{\fontsize{\fsize}{#1\fsize}\selectfont}%
  \ifx\svgwidth\undefined%
    \setlength{\unitlength}{159.46057892bp}%
    \ifx\svgscale\undefined%
      \relax%
    \else%
      \setlength{\unitlength}{\unitlength * \real{\svgscale}}%
    \fi%
  \else%
    \setlength{\unitlength}{\svgwidth}%
  \fi%
  \global\let\svgwidth\undefined%
  \global\let\svgscale\undefined%
  \makeatother%
  \begin{picture}(1,0.50219073)%
    \lineheight{1}%
    \setlength\tabcolsep{0pt}%
    \put(0,0){\includegraphics[width=\unitlength,page=1]{figconnector.pdf}}%
    \put(0.29685492,0.25309612){\makebox(0,0)[lt]{\lineheight{1.25}\smash{\begin{tabular}[t]{l}$\za$\end{tabular}}}}%
    \put(0.66927564,0.23547013){\makebox(0,0)[lt]{\lineheight{1.25}\smash{\begin{tabular}[t]{l}$\zb$\end{tabular}}}}%
    \put(0,0){\includegraphics[width=\unitlength,page=2]{figconnector.pdf}}%
  \end{picture}%
\endgroup%

%% file: figoneloop.pdf_tex
\begingroup%
  \makeatletter%
  \providecommand\color[2][]{%
    \errmessage{(Inkscape) Color is used for the text in Inkscape, but the package 'color.sty' is not loaded}%
    \renewcommand\color[2][]{}%
  }%
  \providecommand\transparent[1]{%
    \errmessage{(Inkscape) Transparency is used (non-zero) for the text in Inkscape, but the package 'transparent.sty' is not loaded}%
    \renewcommand\transparent[1]{}%
  }%
  \providecommand\rotatebox[2]{#2}%
  \newcommand*\fsize{\dimexpr\f@size pt\relax}%
  \newcommand*\lineheight[1]{\fontsize{\fsize}{#1\fsize}\selectfont}%
  \ifx\svgwidth\undefined%
    \setlength{\unitlength}{346.39280316bp}%
    \ifx\svgscale\undefined%
      \relax%
    \else%
      \setlength{\unitlength}{\unitlength * \real{\svgscale}}%
    \fi%
  \else%
    \setlength{\unitlength}{\svgwidth}%
  \fi%
  \global\let\svgwidth\undefined%
  \global\let\svgscale\undefined%
  \makeatother%
  \begin{picture}(1,0.2724781)%
    \lineheight{1}%
    \setlength\tabcolsep{0pt}%
    \put(0,0){\includegraphics[width=\unitlength,page=1]{figoneloop.pdf}}%
    \put(0.04624883,0.09766002){\makebox(0,0)[lt]{\lineheight{1.25}\smash{\begin{tabular}[t]{l}$\tau_i$\end{tabular}}}}%
    \put(0,0){\includegraphics[width=\unitlength,page=2]{figoneloop.pdf}}%
    \put(0.48518504,0.17050985){\makebox(0,0)[lt]{\lineheight{1.25}\smash{\begin{tabular}[t]{l}$\zg$\end{tabular}}}}%
    \put(0.43109506,0.04491105){\makebox(0,0)[lt]{\lineheight{1.25}\smash{\begin{tabular}[t]{l}$\zd$\end{tabular}}}}%
    \put(0,0){\includegraphics[width=\unitlength,page=3]{figoneloop.pdf}}%
    \put(0.82144536,0.13758521){\makebox(0,0)[lt]{\lineheight{1.25}\smash{\begin{tabular}[t]{l}3\end{tabular}}}}%
    \put(0.8215299,0.18082403){\makebox(0,0)[lt]{\lineheight{1.25}\smash{\begin{tabular}[t]{l}2\end{tabular}}}}%
    \put(0.81988064,0.22879573){\makebox(0,0)[lt]{\lineheight{1.25}\smash{\begin{tabular}[t]{l}1\end{tabular}}}}%
    \put(0,0){\includegraphics[width=\unitlength,page=4]{figoneloop.pdf}}%
  \end{picture}%
\endgroup%

%% file: figanglecrossing.pdf_tex
\begingroup%
  \makeatletter%
  \providecommand\color[2][]{%
    \errmessage{(Inkscape) Color is used for the text in Inkscape, but the package 'color.sty' is not loaded}%
    \renewcommand\color[2][]{}%
  }%
  \providecommand\transparent[1]{%
    \errmessage{(Inkscape) Transparency is used (non-zero) for the text in Inkscape, but the package 'transparent.sty' is not loaded}%
    \renewcommand\transparent[1]{}%
  }%
  \providecommand\rotatebox[2]{#2}%
  \newcommand*\fsize{\dimexpr\f@size pt\relax}%
  \newcommand*\lineheight[1]{\fontsize{\fsize}{#1\fsize}\selectfont}%
  \ifx\svgwidth\undefined%
    \setlength{\unitlength}{547.29520556bp}%
    \ifx\svgscale\undefined%
      \relax%
    \else%
      \setlength{\unitlength}{\unitlength * \real{\svgscale}}%
    \fi%
  \else%
    \setlength{\unitlength}{\svgwidth}%
  \fi%
  \global\let\svgwidth\undefined%
  \global\let\svgscale\undefined%
  \makeatother%
  \begin{picture}(1,0.19048221)%
    \lineheight{1}%
    \setlength\tabcolsep{0pt}%
    \put(0,0){\includegraphics[width=\unitlength,page=1]{figanglecrossing.pdf}}%
    \put(0.21014037,0.10868742){\makebox(0,0)[lt]{\lineheight{1.25}\smash{\begin{tabular}[t]{l}$x$\end{tabular}}}}%
    \put(0.78569818,0.10868742){\makebox(0,0)[lt]{\lineheight{1.25}\smash{\begin{tabular}[t]{l}$x$\end{tabular}}}}%
    \put(0.00975589,0.1743585){\makebox(0,0)[lt]{\lineheight{1.25}\smash{\begin{tabular}[t]{l}$\za$\end{tabular}}}}%
    \put(0.58531351,0.1743585){\makebox(0,0)[lt]{\lineheight{1.25}\smash{\begin{tabular}[t]{l}$\overline\za$\end{tabular}}}}%
    \put(0.00998469,0.04455706){\makebox(0,0)[lt]{\lineheight{1.25}\smash{\begin{tabular}[t]{l}$\zb$\end{tabular}}}}%
    \put(0.58554239,0.04455706){\makebox(0,0)[lt]{\lineheight{1.25}\smash{\begin{tabular}[t]{l}$\zb$\end{tabular}}}}%
  \end{picture}%
\endgroup%

%% file: fig_overlap1.pdf_tex
\begingroup%
  \makeatletter%
  \providecommand\color[2][]{%
    \errmessage{(Inkscape) Color is used for the text in Inkscape, but the package 'color.sty' is not loaded}%
    \renewcommand\color[2][]{}%
  }%
  \providecommand\transparent[1]{%
    \errmessage{(Inkscape) Transparency is used (non-zero) for the text in Inkscape, but the package 'transparent.sty' is not loaded}%
    \renewcommand\transparent[1]{}%
  }%
  \providecommand\rotatebox[2]{#2}%
  \newcommand*\fsize{\dimexpr\f@size pt\relax}%
  \newcommand*\lineheight[1]{\fontsize{\fsize}{#1\fsize}\selectfont}%
  \ifx\svgwidth\undefined%
    \setlength{\unitlength}{371.7043755bp}%
    \ifx\svgscale\undefined%
      \relax%
    \else%
      \setlength{\unitlength}{\unitlength * \real{\svgscale}}%
    \fi%
  \else%
    \setlength{\unitlength}{\svgwidth}%
  \fi%
  \global\let\svgwidth\undefined%
  \global\let\svgscale\undefined%
  \makeatother%
  \begin{picture}(1,0.4093305)%
    \lineheight{1}%
    \setlength\tabcolsep{0pt}%
    \put(0,0){\includegraphics[width=\unitlength,page=1]{fig_overlap1.pdf}}%
    \put(0.71627651,0.30420516){\makebox(0,0)[lt]{\lineheight{1.25}\smash{\begin{tabular}[t]{l}$d$\end{tabular}}}}%
    \put(0.26026911,0.27192145){\makebox(0,0)[lt]{\lineheight{1.25}\smash{\begin{tabular}[t]{l}$\zD_2$\end{tabular}}}}%
    \put(0.17552435,0.27192145){\makebox(0,0)[lt]{\lineheight{1.25}\smash{\begin{tabular}[t]{l}$\zD_1$\end{tabular}}}}%
    \put(0,0){\includegraphics[width=\unitlength,page=2]{fig_overlap1.pdf}}%
    \put(0.71627651,0.12260919){\makebox(0,0)[lt]{\lineheight{1.25}\smash{\begin{tabular}[t]{l}$b$\end{tabular}}}}%
    \put(0.11095692,0.12260919){\makebox(0,0)[lt]{\lineheight{1.25}\smash{\begin{tabular}[t]{l}$c$\end{tabular}}}}%
    \put(0.11095692,0.30016969){\makebox(0,0)[lt]{\lineheight{1.25}\smash{\begin{tabular}[t]{l}$a$\end{tabular}}}}%
    \put(-0.00203612,0.23156679){\makebox(0,0)[lt]{\lineheight{1.25}\smash{\begin{tabular}[t]{l}$\zD_0$\end{tabular}}}}%
    \put(0.62346085,0.27192145){\makebox(0,0)[lt]{\lineheight{1.25}\smash{\begin{tabular}[t]{l}$\zD_{t-1}$\end{tabular}}}}%
    \put(0.82523412,0.23156679){\makebox(0,0)[lt]{\lineheight{1.25}\smash{\begin{tabular}[t]{l}$\zD_{t}$\end{tabular}}}}%
    \put(0,0){\includegraphics[width=\unitlength,page=3]{fig_overlap1.pdf}}%
    \put(0.81316047,0.36704802){\makebox(0,0)[lt]{\lineheight{1.25}\smash{\begin{tabular}[t]{l}$\zg_{w_2}$\end{tabular}}}}%
    \put(0.81719599,0.06438813){\makebox(0,0)[lt]{\lineheight{1.25}\smash{\begin{tabular}[t]{l}$\zg_{w_1}$\end{tabular}}}}%
    \put(0.01413857,0.3468707){\makebox(0,0)[lt]{\lineheight{1.25}\smash{\begin{tabular}[t]{l}$\zg_{w_1}$\end{tabular}}}}%
    \put(0.0101031,0.08456544){\makebox(0,0)[lt]{\lineheight{1.25}\smash{\begin{tabular}[t]{l}$\zg_{w_2}$\end{tabular}}}}%
  \end{picture}%
\endgroup%

%% file: fig_overlap2.pdf_tex
\begingroup%
  \makeatletter%
  \providecommand\color[2][]{%
    \errmessage{(Inkscape) Color is used for the text in Inkscape, but the package 'color.sty' is not loaded}%
    \renewcommand\color[2][]{}%
  }%
  \providecommand\transparent[1]{%
    \errmessage{(Inkscape) Transparency is used (non-zero) for the text in Inkscape, but the package 'transparent.sty' is not loaded}%
    \renewcommand\transparent[1]{}%
  }%
  \providecommand\rotatebox[2]{#2}%
  \newcommand*\fsize{\dimexpr\f@size pt\relax}%
  \newcommand*\lineheight[1]{\fontsize{\fsize}{#1\fsize}\selectfont}%
  \ifx\svgwidth\undefined%
    \setlength{\unitlength}{197.36012857bp}%
    \ifx\svgscale\undefined%
      \relax%
    \else%
      \setlength{\unitlength}{\unitlength * \real{\svgscale}}%
    \fi%
  \else%
    \setlength{\unitlength}{\svgwidth}%
  \fi%
  \global\let\svgwidth\undefined%
  \global\let\svgscale\undefined%
  \makeatother%
  \begin{picture}(1,0.7242012)%
    \lineheight{1}%
    \setlength\tabcolsep{0pt}%
    \put(0,0){\includegraphics[width=\unitlength,page=1]{fig_overlap2.pdf}}%
    \put(0.46465914,0.08114478){\makebox(0,0)[lt]{\lineheight{1.25}\smash{\begin{tabular}[t]{l}$\myPangle$\end{tabular}}}}%
    \put(0,0){\includegraphics[width=\unitlength,page=2]{fig_overlap2.pdf}}%
    \put(0.22238943,0.32407679){\makebox(0,0)[lt]{\lineheight{1.25}\smash{\begin{tabular}[t]{l}$\zD$\end{tabular}}}}%
    \put(0,0){\includegraphics[width=\unitlength,page=3]{fig_overlap2.pdf}}%
    \put(0.6438163,0.12809293){\makebox(0,0)[lt]{\lineheight{1.25}\smash{\begin{tabular}[t]{l}$\rho$\end{tabular}}}}%
    \put(0.56969331,0.30553561){\makebox(0,0)[lt]{\lineheight{1.25}\smash{\begin{tabular}[t]{l}$\zg_1$\end{tabular}}}}%
    \put(0.51771492,0.48924317){\makebox(0,0)[lt]{\lineheight{1.25}\smash{\begin{tabular}[t]{l}$\zg_2$\end{tabular}}}}%
  \end{picture}%
\endgroup%

%% file: fig_overlap4.pdf_tex
\begingroup%
  \makeatletter%
  \providecommand\color[2][]{%
    \errmessage{(Inkscape) Color is used for the text in Inkscape, but the package 'color.sty' is not loaded}%
    \renewcommand\color[2][]{}%
  }%
  \providecommand\transparent[1]{%
    \errmessage{(Inkscape) Transparency is used (non-zero) for the text in Inkscape, but the package 'transparent.sty' is not loaded}%
    \renewcommand\transparent[1]{}%
  }%
  \providecommand\rotatebox[2]{#2}%
  \newcommand*\fsize{\dimexpr\f@size pt\relax}%
  \newcommand*\lineheight[1]{\fontsize{\fsize}{#1\fsize}\selectfont}%
  \ifx\svgwidth\undefined%
    \setlength{\unitlength}{336.95591724bp}%
    \ifx\svgscale\undefined%
      \relax%
    \else%
      \setlength{\unitlength}{\unitlength * \real{\svgscale}}%
    \fi%
  \else%
    \setlength{\unitlength}{\svgwidth}%
  \fi%
  \global\let\svgwidth\undefined%
  \global\let\svgscale\undefined%
  \makeatother%
  \begin{picture}(1,0.45154257)%
    \lineheight{1}%
    \setlength\tabcolsep{0pt}%
    \put(0,0){\includegraphics[width=\unitlength,page=1]{fig_overlap4.pdf}}%
    \put(0.73223495,0.36228595){\makebox(0,0)[lt]{\lineheight{1.25}\smash{\begin{tabular}[t]{l}$d$\end{tabular}}}}%
    \put(0,0){\includegraphics[width=\unitlength,page=2]{fig_overlap4.pdf}}%
    \put(0.73223495,0.10409184){\makebox(0,0)[lt]{\lineheight{1.25}\smash{\begin{tabular}[t]{l}$b$\end{tabular}}}}%
    \put(0.11791148,0.10409184){\makebox(0,0)[lt]{\lineheight{1.25}\smash{\begin{tabular}[t]{l}$c$\end{tabular}}}}%
    \put(0.11791148,0.36673758){\makebox(0,0)[lt]{\lineheight{1.25}\smash{\begin{tabular}[t]{l}$a$\end{tabular}}}}%
    \put(0.69217038,0.23764053){\makebox(0,0)[lt]{\lineheight{1.25}\smash{\begin{tabular}[t]{l}$\rho_r$\end{tabular}}}}%
    \put(0,0){\includegraphics[width=\unitlength,page=3]{fig_overlap4.pdf}}%
    \put(0.87472302,0.40489971){\makebox(0,0)[lt]{\lineheight{1.25}\smash{\begin{tabular}[t]{l}$\zg_2$\end{tabular}}}}%
    \put(0.87917471,0.07102813){\makebox(0,0)[lt]{\lineheight{1.25}\smash{\begin{tabular}[t]{l}$\zg_1$\end{tabular}}}}%
    \put(0.00220553,0.38264162){\makebox(0,0)[lt]{\lineheight{1.25}\smash{\begin{tabular}[t]{l}$\zg_1$\end{tabular}}}}%
    \put(-0.00224609,0.09328622){\makebox(0,0)[lt]{\lineheight{1.25}\smash{\begin{tabular}[t]{l}$\zg_2$\end{tabular}}}}%
    \put(0,0){\includegraphics[width=\unitlength,page=4]{fig_overlap4.pdf}}%
    \put(0.15152465,0.23745148){\makebox(0,0)[lt]{\lineheight{1.25}\smash{\begin{tabular}[t]{l}$\rho_1$\end{tabular}}}}%
    \put(0,0){\includegraphics[width=\unitlength,page=5]{fig_overlap4.pdf}}%
    \put(0.41529809,0.15038264){\makebox(0,0)[lt]{\lineheight{1.25}\smash{\begin{tabular}[t]{l}$\myPangle$\end{tabular}}}}%
    \put(0,0){\includegraphics[width=\unitlength,page=6]{fig_overlap4.pdf}}%
  \end{picture}%
\endgroup%

%% file: figintersectionarrow.pdf_tex
\begingroup%
  \makeatletter%
  \providecommand\color[2][]{%
    \errmessage{(Inkscape) Color is used for the text in Inkscape, but the package 'color.sty' is not loaded}%
    \renewcommand\color[2][]{}%
  }%
  \providecommand\transparent[1]{%
    \errmessage{(Inkscape) Transparency is used (non-zero) for the text in Inkscape, but the package 'transparent.sty' is not loaded}%
    \renewcommand\transparent[1]{}%
  }%
  \providecommand\rotatebox[2]{#2}%
  \newcommand*\fsize{\dimexpr\f@size pt\relax}%
  \newcommand*\lineheight[1]{\fontsize{\fsize}{#1\fsize}\selectfont}%
  \ifx\svgwidth\undefined%
    \setlength{\unitlength}{133.29478346bp}%
    \ifx\svgscale\undefined%
      \relax%
    \else%
      \setlength{\unitlength}{\unitlength * \real{\svgscale}}%
    \fi%
  \else%
    \setlength{\unitlength}{\svgwidth}%
  \fi%
  \global\let\svgwidth\undefined%
  \global\let\svgscale\undefined%
  \makeatother%
  \begin{picture}(1,0.65664767)%
    \lineheight{1}%
    \setlength\tabcolsep{0pt}%
    \put(0,0){\includegraphics[width=\unitlength,page=1]{figintersectionarrow.pdf}}%
    \put(0.34350444,0.59891626){\makebox(0,0)[lt]{\lineheight{1.25}\smash{\begin{tabular}[t]{l}$x$\end{tabular}}}}%
    \put(0.34451516,0.013224){\makebox(0,0)[lt]{\lineheight{1.25}\smash{\begin{tabular}[t]{l}$x'$\end{tabular}}}}%
    \put(-0.0056779,0.45854469){\makebox(0,0)[lt]{\lineheight{1.25}\smash{\begin{tabular}[t]{l}$\tau_2$\end{tabular}}}}%
    \put(0.65826373,0.45854469){\makebox(0,0)[lt]{\lineheight{1.25}\smash{\begin{tabular}[t]{l}$\tau_1$\end{tabular}}}}%
    \put(-0.0056779,0.17721332){\makebox(0,0)[lt]{\lineheight{1.25}\smash{\begin{tabular}[t]{l}$\zg_2$\end{tabular}}}}%
    \put(0.66951748,0.17721332){\makebox(0,0)[lt]{\lineheight{1.25}\smash{\begin{tabular}[t]{l}$\zg_1$\end{tabular}}}}%
  \end{picture}%
\endgroup%

%% file: fig_sectfive1.pdf_tex
\begingroup%
  \makeatletter%
  \providecommand\color[2][]{%
    \errmessage{(Inkscape) Color is used for the text in Inkscape, but the package 'color.sty' is not loaded}%
    \renewcommand\color[2][]{}%
  }%
  \providecommand\transparent[1]{%
    \errmessage{(Inkscape) Transparency is used (non-zero) for the text in Inkscape, but the package 'transparent.sty' is not loaded}%
    \renewcommand\transparent[1]{}%
  }%
  \providecommand\rotatebox[2]{#2}%
  \newcommand*\fsize{\dimexpr\f@size pt\relax}%
  \newcommand*\lineheight[1]{\fontsize{\fsize}{#1\fsize}\selectfont}%
  \ifx\svgwidth\undefined%
    \setlength{\unitlength}{433.74718107bp}%
    \ifx\svgscale\undefined%
      \relax%
    \else%
      \setlength{\unitlength}{\unitlength * \real{\svgscale}}%
    \fi%
  \else%
    \setlength{\unitlength}{\svgwidth}%
  \fi%
  \global\let\svgwidth\undefined%
  \global\let\svgscale\undefined%
  \makeatother%
  \begin{picture}(1,0.37029492)%
    \lineheight{1}%
    \setlength\tabcolsep{0pt}%
    \put(0,0){\includegraphics[width=\unitlength,page=1]{fig_sectfive1.pdf}}%
    \put(0.661885,0.0439088){\makebox(0,0)[lt]{\lineheight{1.25}\smash{\begin{tabular}[t]{l}$\zD''$\end{tabular}}}}%
    \put(0.62038603,0.11653177){\makebox(0,0)[lt]{\lineheight{1.25}\smash{\begin{tabular}[t]{l}$\zD'$\end{tabular}}}}%
    \put(0.85675758,0.11167324){\makebox(0,0)[lt]{\lineheight{1.25}\smash{\begin{tabular}[t]{l}$\za^i$\end{tabular}}}}%
    \put(0.77806954,0.03829329){\makebox(0,0)[lt]{\lineheight{1.25}\smash{\begin{tabular}[t]{l}$p$\end{tabular}}}}%
    \put(0,0){\includegraphics[width=\unitlength,page=2]{fig_sectfive1.pdf}}%
    \put(0.95789787,0.33915984){\makebox(0,0)[lt]{\lineheight{1.25}\smash{\begin{tabular}[t]{l}$q$\end{tabular}}}}%
    \put(0.77622482,0.28569707){\makebox(0,0)[lt]{\lineheight{1.25}\smash{\begin{tabular}[t]{l}$j$\end{tabular}}}}%
    \put(0,0){\includegraphics[width=\unitlength,page=3]{fig_sectfive1.pdf}}%
    \put(0.78202863,0.19399078){\makebox(0,0)[lt]{\lineheight{1.25}\smash{\begin{tabular}[t]{l}$\zg$\end{tabular}}}}%
    \put(0,0){\includegraphics[width=\unitlength,page=4]{fig_sectfive1.pdf}}%
    \put(0.10856727,0.0439088){\makebox(0,0)[lt]{\lineheight{1.25}\smash{\begin{tabular}[t]{l}$\zD''$\end{tabular}}}}%
    \put(0.06706828,0.11653177){\makebox(0,0)[lt]{\lineheight{1.25}\smash{\begin{tabular}[t]{l}$\zD'$\end{tabular}}}}%
    \put(0.30343968,0.11167324){\makebox(0,0)[lt]{\lineheight{1.25}\smash{\begin{tabular}[t]{l}$\za^i$\end{tabular}}}}%
    \put(0.22475174,0.03829329){\makebox(0,0)[lt]{\lineheight{1.25}\smash{\begin{tabular}[t]{l}$p$\end{tabular}}}}%
    \put(0,0){\includegraphics[width=\unitlength,page=5]{fig_sectfive1.pdf}}%
    \put(0.22475174,0.33915984){\makebox(0,0)[lt]{\lineheight{1.25}\smash{\begin{tabular}[t]{l}$q$\end{tabular}}}}%
    \put(0.26094715,0.28569707){\makebox(0,0)[lt]{\lineheight{1.25}\smash{\begin{tabular}[t]{l}$j$\end{tabular}}}}%
    \put(0,0){\includegraphics[width=\unitlength,page=6]{fig_sectfive1.pdf}}%
    \put(0.22179445,0.19399078){\makebox(0,0)[lt]{\lineheight{1.25}\smash{\begin{tabular}[t]{l}$\zg$\end{tabular}}}}%
    \put(0,0){\includegraphics[width=\unitlength,page=7]{fig_sectfive1.pdf}}%
    \put(0.59016349,0.18648882){\makebox(0,0)[lt]{\lineheight{1.25}\smash{\begin{tabular}[t]{l}$j'$\end{tabular}}}}%
    \put(0,0){\includegraphics[width=\unitlength,page=8]{fig_sectfive1.pdf}}%
    \put(0.65471393,0.21572024){\makebox(0,0)[lt]{\lineheight{1.25}\smash{\begin{tabular}[t]{l}\small$\zb^1$\end{tabular}}}}%
    \put(0.68157241,0.23999206){\makebox(0,0)[lt]{\lineheight{1.25}\smash{\begin{tabular}[t]{l}$c$\end{tabular}}}}%
    \put(0,0){\includegraphics[width=\unitlength,page=9]{fig_sectfive1.pdf}}%
    \put(0.65471393,0.18113787){\makebox(0,0)[lt]{\lineheight{1.25}\smash{\begin{tabular}[t]{l}\small$\zb^s$\end{tabular}}}}%
    \put(0.03684582,0.18648882){\makebox(0,0)[lt]{\lineheight{1.25}\smash{\begin{tabular}[t]{l}$j'$\end{tabular}}}}%
  \end{picture}%
\endgroup%

%% file: figprop73.pdf_tex
\begingroup%
  \makeatletter%
  \providecommand\color[2][]{%
    \errmessage{(Inkscape) Color is used for the text in Inkscape, but the package 'color.sty' is not loaded}%
    \renewcommand\color[2][]{}%
  }%
  \providecommand\transparent[1]{%
    \errmessage{(Inkscape) Transparency is used (non-zero) for the text in Inkscape, but the package 'transparent.sty' is not loaded}%
    \renewcommand\transparent[1]{}%
  }%
  \providecommand\rotatebox[2]{#2}%
  \newcommand*\fsize{\dimexpr\f@size pt\relax}%
  \newcommand*\lineheight[1]{\fontsize{\fsize}{#1\fsize}\selectfont}%
  \ifx\svgwidth\undefined%
    \setlength{\unitlength}{446.36418596bp}%
    \ifx\svgscale\undefined%
      \relax%
    \else%
      \setlength{\unitlength}{\unitlength * \real{\svgscale}}%
    \fi%
  \else%
    \setlength{\unitlength}{\svgwidth}%
  \fi%
  \global\let\svgwidth\undefined%
  \global\let\svgscale\undefined%
  \makeatother%
  \begin{picture}(1,0.36406961)%
    \lineheight{1}%
    \setlength\tabcolsep{0pt}%
    \put(0,0){\includegraphics[width=\unitlength,page=1]{figprop73.pdf}}%
    \put(0.19153231,0.34682966){\makebox(0,0)[lt]{\lineheight{1.25}\smash{\begin{tabular}[t]{l}$x_1$\end{tabular}}}}%
    \put(0.0755956,0.16704375){\makebox(0,0)[lt]{\lineheight{1.25}\smash{\begin{tabular}[t]{l}$x_2$\end{tabular}}}}%
    \put(0.20329401,0.30650388){\makebox(0,0)[lt]{\lineheight{1.25}\smash{\begin{tabular}[t]{l}$\za$\end{tabular}}}}%
    \put(0.11760162,0.16536348){\makebox(0,0)[lt]{\lineheight{1.25}\smash{\begin{tabular}[t]{l}$\zb$\end{tabular}}}}%
    \put(0.30410853,0.27625949){\makebox(0,0)[lt]{\lineheight{1.25}\smash{\begin{tabular}[t]{l}$\tau_i$\end{tabular}}}}%
    \put(0.12768311,0.25945705){\makebox(0,0)[lt]{\lineheight{1.25}\smash{\begin{tabular}[t]{l}$\tau_j$\end{tabular}}}}%
    \put(0.14448552,0.06958965){\makebox(0,0)[lt]{\lineheight{1.25}\smash{\begin{tabular}[t]{l}$\tau_k$\end{tabular}}}}%
    \put(-0.00169555,0.07295009){\makebox(0,0)[lt]{\lineheight{1.25}\smash{\begin{tabular}[t]{l}$(S,M,P)$\end{tabular}}}}%
    \put(0,0){\includegraphics[width=\unitlength,page=2]{figprop73.pdf}}%
    \put(0.72920979,0.34682966){\makebox(0,0)[lt]{\lineheight{1.25}\smash{\begin{tabular}[t]{l}$x_1$\end{tabular}}}}%
    \put(0.61327303,0.16704375){\makebox(0,0)[lt]{\lineheight{1.25}\smash{\begin{tabular}[t]{l}$x_2$\end{tabular}}}}%
    \put(0.76113406,0.2997829){\makebox(0,0)[lt]{\lineheight{1.25}\smash{\begin{tabular}[t]{l}$\za_a$\end{tabular}}}}%
    \put(0.6552791,0.17544495){\makebox(0,0)[lt]{\lineheight{1.25}\smash{\begin{tabular}[t]{l}\small$\zb_a$\end{tabular}}}}%
    \put(0.84178606,0.27625949){\makebox(0,0)[lt]{\lineheight{1.25}\smash{\begin{tabular}[t]{l}$\tau_i$\end{tabular}}}}%
    \put(0.66536052,0.25945705){\makebox(0,0)[lt]{\lineheight{1.25}\smash{\begin{tabular}[t]{l}$\tau_j$\end{tabular}}}}%
    \put(0.68216296,0.06958965){\makebox(0,0)[lt]{\lineheight{1.25}\smash{\begin{tabular}[t]{l}$\tau_k$\end{tabular}}}}%
    \put(0.53598205,0.07295009){\makebox(0,0)[lt]{\lineheight{1.25}\smash{\begin{tabular}[t]{l}$G(S,M,P)$\end{tabular}}}}%
    \put(0,0){\includegraphics[width=\unitlength,page=3]{figprop73.pdf}}%
    \put(0.90395511,0.11159573){\makebox(0,0)[lt]{\lineheight{1.25}\smash{\begin{tabular}[t]{l}$x(\za)$\end{tabular}}}}%
    \put(0.88715272,0.04102553){\makebox(0,0)[lt]{\lineheight{1.25}\smash{\begin{tabular}[t]{l}$x(\zb)$\end{tabular}}}}%
    \put(0,0){\includegraphics[width=\unitlength,page=4]{figprop73.pdf}}%
    \put(0.72752937,0.2997829){\makebox(0,0)[lt]{\lineheight{1.25}\smash{\begin{tabular}[t]{l}$\za_b$\end{tabular}}}}%
    \put(0.66200008,0.14856109){\makebox(0,0)[lt]{\lineheight{1.25}\smash{\begin{tabular}[t]{l}\small$\zb_b$\end{tabular}}}}%
    \put(0,0){\includegraphics[width=\unitlength,page=5]{figprop73.pdf}}%
    \put(0.81154181,0.20736956){\makebox(0,0)[lt]{\lineheight{1.25}\smash{\begin{tabular}[t]{l}$\tau(\za)$\end{tabular}}}}%
    \put(0.74265193,0.10151426){\makebox(0,0)[lt]{\lineheight{1.25}\smash{\begin{tabular}[t]{l}$\tau(\zb)$\end{tabular}}}}%
  \end{picture}%
\endgroup%

%% file: figangles1.pdf_tex
\begingroup%
  \makeatletter%
  \providecommand\color[2][]{%
    \errmessage{(Inkscape) Color is used for the text in Inkscape, but the package 'color.sty' is not loaded}%
    \renewcommand\color[2][]{}%
  }%
  \providecommand\transparent[1]{%
    \errmessage{(Inkscape) Transparency is used (non-zero) for the text in Inkscape, but the package 'transparent.sty' is not loaded}%
    \renewcommand\transparent[1]{}%
  }%
  \providecommand\rotatebox[2]{#2}%
  \newcommand*\fsize{\dimexpr\f@size pt\relax}%
  \newcommand*\lineheight[1]{\fontsize{\fsize}{#1\fsize}\selectfont}%
  \ifx\svgwidth\undefined%
    \setlength{\unitlength}{369.19856676bp}%
    \ifx\svgscale\undefined%
      \relax%
    \else%
      \setlength{\unitlength}{\unitlength * \real{\svgscale}}%
    \fi%
  \else%
    \setlength{\unitlength}{\svgwidth}%
  \fi%
  \global\let\svgwidth\undefined%
  \global\let\svgscale\undefined%
  \makeatother%
  \begin{picture}(1,0.33120915)%
    \lineheight{1}%
    \setlength\tabcolsep{0pt}%
    \put(0,0){\includegraphics[width=\unitlength,page=1]{figangles1.pdf}}%
    \put(0.09259809,0.30471408){\makebox(0,0)[lt]{\lineheight{1.25}\smash{\begin{tabular}[t]{l}$x$\end{tabular}}}}%
    \put(0.00930954,0.16251418){\makebox(0,0)[lt]{\lineheight{1.25}\smash{\begin{tabular}[t]{l}$\za$\end{tabular}}}}%
    \put(0.17994944,0.16251418){\makebox(0,0)[lt]{\lineheight{1.25}\smash{\begin{tabular}[t]{l}$\zb$\end{tabular}}}}%
    \put(0,0){\includegraphics[width=\unitlength,page=2]{figangles1.pdf}}%
    \put(0.53138632,0.2945436){\makebox(0,0)[lt]{\lineheight{1.25}\smash{\begin{tabular}[t]{l}$x$\end{tabular}}}}%
    \put(0,0){\includegraphics[width=\unitlength,page=3]{figangles1.pdf}}%
    \put(0.45216073,0.14626274){\makebox(0,0)[lt]{\lineheight{1.25}\smash{\begin{tabular}[t]{l}$\za$\end{tabular}}}}%
    \put(0.55779499,0.14626274){\makebox(0,0)[lt]{\lineheight{1.25}\smash{\begin{tabular}[t]{l}$\zb^\circ$\end{tabular}}}}%
    \put(0.84016302,0.15032563){\makebox(0,0)[lt]{\lineheight{1.25}\smash{\begin{tabular}[t]{l}$x$\end{tabular}}}}%
    \put(0,0){\includegraphics[width=\unitlength,page=4]{figangles1.pdf}}%
    \put(0.84219442,0.23564554){\makebox(0,0)[lt]{\lineheight{1.25}\smash{\begin{tabular}[t]{l}$\za^\circ$\end{tabular}}}}%
    \put(0.84219442,0.08125705){\makebox(0,0)[lt]{\lineheight{1.25}\smash{\begin{tabular}[t]{l}$\zb^\circ$\end{tabular}}}}%
    \put(0,0){\includegraphics[width=\unitlength,page=5]{figangles1.pdf}}%
  \end{picture}%
\endgroup%

%% file: figangles2.pdf_tex
\begingroup%
  \makeatletter%
  \providecommand\color[2][]{%
    \errmessage{(Inkscape) Color is used for the text in Inkscape, but the package 'color.sty' is not loaded}%
    \renewcommand\color[2][]{}%
  }%
  \providecommand\transparent[1]{%
    \errmessage{(Inkscape) Transparency is used (non-zero) for the text in Inkscape, but the package 'transparent.sty' is not loaded}%
    \renewcommand\transparent[1]{}%
  }%
  \providecommand\rotatebox[2]{#2}%
  \newcommand*\fsize{\dimexpr\f@size pt\relax}%
  \newcommand*\lineheight[1]{\fontsize{\fsize}{#1\fsize}\selectfont}%
  \ifx\svgwidth\undefined%
    \setlength{\unitlength}{401.5663353bp}%
    \ifx\svgscale\undefined%
      \relax%
    \else%
      \setlength{\unitlength}{\unitlength * \real{\svgscale}}%
    \fi%
  \else%
    \setlength{\unitlength}{\svgwidth}%
  \fi%
  \global\let\svgwidth\undefined%
  \global\let\svgscale\undefined%
  \makeatother%
  \begin{picture}(1,0.24328558)%
    \lineheight{1}%
    \setlength\tabcolsep{0pt}%
    \put(0,0){\includegraphics[width=\unitlength,page=1]{figangles2.pdf}}%
    \put(0.10272276,0.22412237){\makebox(0,0)[lt]{\lineheight{1.25}\smash{\begin{tabular}[t]{l}$x$\end{tabular}}}}%
    \put(0.02614759,0.09338432){\makebox(0,0)[lt]{\lineheight{1.25}\smash{\begin{tabular}[t]{l}$\za$\end{tabular}}}}%
    \put(0.18303326,0.09338432){\makebox(0,0)[lt]{\lineheight{1.25}\smash{\begin{tabular}[t]{l}$\zb$\end{tabular}}}}%
    \put(0,0){\includegraphics[width=\unitlength,page=2]{figangles2.pdf}}%
    \put(0.20544552,0.15315025){\makebox(0,0)[lt]{\lineheight{1.25}\smash{\begin{tabular}[t]{l}$\tau$\end{tabular}}}}%
    \put(0,0){\includegraphics[width=\unitlength,page=3]{figangles2.pdf}}%
    \put(0.47626004,0.22412238){\makebox(0,0)[lt]{\lineheight{1.25}\smash{\begin{tabular}[t]{l}$x$\end{tabular}}}}%
    \put(0.3847434,0.07097208){\makebox(0,0)[lt]{\lineheight{1.25}\smash{\begin{tabular}[t]{l}$\za$\end{tabular}}}}%
    \put(0.57151212,0.0672367){\makebox(0,0)[lt]{\lineheight{1.25}\smash{\begin{tabular}[t]{l}$\zb$\end{tabular}}}}%
    \put(0,0){\includegraphics[width=\unitlength,page=4]{figangles2.pdf}}%
    \put(0.57898271,0.13447341){\makebox(0,0)[lt]{\lineheight{1.25}\smash{\begin{tabular}[t]{l}$\tau$\end{tabular}}}}%
    \put(0,0){\includegraphics[width=\unitlength,page=5]{figangles2.pdf}}%
    \put(0.84979715,0.22412237){\makebox(0,0)[lt]{\lineheight{1.25}\smash{\begin{tabular}[t]{l}$x$\end{tabular}}}}%
    \put(0.77322207,0.09338432){\makebox(0,0)[lt]{\lineheight{1.25}\smash{\begin{tabular}[t]{l}$\za$\end{tabular}}}}%
    \put(0.93010783,0.09338432){\makebox(0,0)[lt]{\lineheight{1.25}\smash{\begin{tabular}[t]{l}$\zb$\end{tabular}}}}%
    \put(0,0){\includegraphics[width=\unitlength,page=6]{figangles2.pdf}}%
    \put(0.84414257,0.01404589){\makebox(0,0)[lt]{\lineheight{1.25}\smash{\begin{tabular}[t]{l}$\tau$\end{tabular}}}}%
    \put(0,0){\includegraphics[width=\unitlength,page=7]{figangles2.pdf}}%
  \end{picture}%
\endgroup%

%% file: figangles3.pdf_tex
\begingroup%
  \makeatletter%
  \providecommand\color[2][]{%
    \errmessage{(Inkscape) Color is used for the text in Inkscape, but the package 'color.sty' is not loaded}%
    \renewcommand\color[2][]{}%
  }%
  \providecommand\transparent[1]{%
    \errmessage{(Inkscape) Transparency is used (non-zero) for the text in Inkscape, but the package 'transparent.sty' is not loaded}%
    \renewcommand\transparent[1]{}%
  }%
  \providecommand\rotatebox[2]{#2}%
  \newcommand*\fsize{\dimexpr\f@size pt\relax}%
  \newcommand*\lineheight[1]{\fontsize{\fsize}{#1\fsize}\selectfont}%
  \ifx\svgwidth\undefined%
    \setlength{\unitlength}{439.50350015bp}%
    \ifx\svgscale\undefined%
      \relax%
    \else%
      \setlength{\unitlength}{\unitlength * \real{\svgscale}}%
    \fi%
  \else%
    \setlength{\unitlength}{\svgwidth}%
  \fi%
  \global\let\svgwidth\undefined%
  \global\let\svgscale\undefined%
  \makeatother%
  \begin{picture}(1,0.19979487)%
    \lineheight{1}%
    \setlength\tabcolsep{0pt}%
    \put(0,0){\includegraphics[width=\unitlength,page=1]{figangles3.pdf}}%
    \put(0.5214681,0.18228581){\makebox(0,0)[lt]{\lineheight{1.25}\smash{\begin{tabular}[t]{l}$x$\end{tabular}}}}%
    \put(0.4515028,0.06283285){\makebox(0,0)[lt]{\lineheight{1.25}\smash{\begin{tabular}[t]{l}$\za$\end{tabular}}}}%
    \put(0.59484641,0.06283285){\makebox(0,0)[lt]{\lineheight{1.25}\smash{\begin{tabular}[t]{l}$\zg$\end{tabular}}}}%
    \put(0,0){\includegraphics[width=\unitlength,page=2]{figangles3.pdf}}%
    \put(0.06518595,0.17546095){\makebox(0,0)[lt]{\lineheight{1.25}\smash{\begin{tabular}[t]{l}$x$\end{tabular}}}}%
    \put(0.00782035,0.06283285){\makebox(0,0)[lt]{\lineheight{1.25}\smash{\begin{tabular}[t]{l}$\za$\end{tabular}}}}%
    \put(0.1512243,0.17197249){\makebox(0,0)[lt]{\lineheight{1.25}\smash{\begin{tabular}[t]{l}$\zb$\end{tabular}}}}%
    \put(0,0){\includegraphics[width=\unitlength,page=3]{figangles3.pdf}}%
    \put(0.22881578,0.17422291){\makebox(0,0)[lt]{\lineheight{1.25}\smash{\begin{tabular}[t]{l}$y$\end{tabular}}}}%
    \put(0,0){\includegraphics[width=\unitlength,page=4]{figangles3.pdf}}%
    \put(0.2654975,0.07989754){\makebox(0,0)[lt]{\lineheight{1.25}\smash{\begin{tabular}[t]{l}$\zg$\end{tabular}}}}%
    \put(0,0){\includegraphics[width=\unitlength,page=5]{figangles3.pdf}}%
    \put(0.53341344,0.00481278){\makebox(0,0)[lt]{\lineheight{1.25}\smash{\begin{tabular}[t]{l}$\zb$\end{tabular}}}}%
    \put(0,0){\includegraphics[width=\unitlength,page=6]{figangles3.pdf}}%
    \put(0.8627623,0.18228581){\makebox(0,0)[lt]{\lineheight{1.25}\smash{\begin{tabular}[t]{l}$x$\end{tabular}}}}%
    \put(0.79279699,0.06283285){\makebox(0,0)[lt]{\lineheight{1.25}\smash{\begin{tabular}[t]{l}$\za$\end{tabular}}}}%
    \put(0.9361408,0.06283285){\makebox(0,0)[lt]{\lineheight{1.25}\smash{\begin{tabular}[t]{l}$\zg$\end{tabular}}}}%
    \put(0,0){\includegraphics[width=\unitlength,page=7]{figangles3.pdf}}%
    \put(0.87129469,0.00481278){\makebox(0,0)[lt]{\lineheight{1.25}\smash{\begin{tabular}[t]{l}$\zb$\end{tabular}}}}%
    \put(0,0){\includegraphics[width=\unitlength,page=8]{figangles3.pdf}}%
  \end{picture}%
\endgroup%

%% file: figskewgentle.pdf_tex
\begingroup%
  \makeatletter%
  \providecommand\color[2][]{%
    \errmessage{(Inkscape) Color is used for the text in Inkscape, but the package 'color.sty' is not loaded}%
    \renewcommand\color[2][]{}%
  }%
  \providecommand\transparent[1]{%
    \errmessage{(Inkscape) Transparency is used (non-zero) for the text in Inkscape, but the package 'transparent.sty' is not loaded}%
    \renewcommand\transparent[1]{}%
  }%
  \providecommand\rotatebox[2]{#2}%
  \newcommand*\fsize{\dimexpr\f@size pt\relax}%
  \newcommand*\lineheight[1]{\fontsize{\fsize}{#1\fsize}\selectfont}%
  \ifx\svgwidth\undefined%
    \setlength{\unitlength}{461.51870057bp}%
    \ifx\svgscale\undefined%
      \relax%
    \else%
      \setlength{\unitlength}{\unitlength * \real{\svgscale}}%
    \fi%
  \else%
    \setlength{\unitlength}{\svgwidth}%
  \fi%
  \global\let\svgwidth\undefined%
  \global\let\svgscale\undefined%
  \makeatother%
  \begin{picture}(1,0.39271121)%
    \lineheight{1}%
    \setlength\tabcolsep{0pt}%
    \put(0,0){\includegraphics[width=\unitlength,page=1]{figskewgentle.pdf}}%
    \put(0.09048237,0.23990522){\makebox(0,0)[lt]{\lineheight{1.25}\smash{\begin{tabular}[t]{l}1\end{tabular}}}}%
    \put(0.17916507,0.07739827){\makebox(0,0)[lt]{\lineheight{1.25}\smash{\begin{tabular}[t]{l}2\end{tabular}}}}%
    \put(0.27899044,0.24023319){\makebox(0,0)[lt]{\lineheight{1.25}\smash{\begin{tabular}[t]{l}3\end{tabular}}}}%
    \put(0,0){\includegraphics[width=\unitlength,page=2]{figskewgentle.pdf}}%
  \end{picture}%
\endgroup%

%% file: figsectionsix3.pdf_tex
\begingroup%
  \makeatletter%
  \providecommand\color[2][]{%
    \errmessage{(Inkscape) Color is used for the text in Inkscape, but the package 'color.sty' is not loaded}%
    \renewcommand\color[2][]{}%
  }%
  \providecommand\transparent[1]{%
    \errmessage{(Inkscape) Transparency is used (non-zero) for the text in Inkscape, but the package 'transparent.sty' is not loaded}%
    \renewcommand\transparent[1]{}%
  }%
  \providecommand\rotatebox[2]{#2}%
  \newcommand*\fsize{\dimexpr\f@size pt\relax}%
  \newcommand*\lineheight[1]{\fontsize{\fsize}{#1\fsize}\selectfont}%
  \ifx\svgwidth\undefined%
    \setlength{\unitlength}{128.30040344bp}%
    \ifx\svgscale\undefined%
      \relax%
    \else%
      \setlength{\unitlength}{\unitlength * \real{\svgscale}}%
    \fi%
  \else%
    \setlength{\unitlength}{\svgwidth}%
  \fi%
  \global\let\svgwidth\undefined%
  \global\let\svgscale\undefined%
  \makeatother%
  \begin{picture}(1,0.83365435)%
    \lineheight{1}%
    \setlength\tabcolsep{0pt}%
    \put(0,0){\includegraphics[width=\unitlength,page=1]{figsectionsix3.pdf}}%
    \put(0.5903577,0.77367561){\makebox(0,0)[lt]{\lineheight{1.25}\smash{\begin{tabular}[t]{l}$x_1$\end{tabular}}}}%
    \put(0.68199094,0.21352169){\makebox(0,0)[lt]{\lineheight{1.25}\smash{\begin{tabular}[t]{l}$\zg_2$\end{tabular}}}}%
    \put(0.42750244,0.65229784){\makebox(0,0)[lt]{\lineheight{1.25}\smash{\begin{tabular}[t]{l}$\zg_1$\end{tabular}}}}%
    \put(0.17450858,0.24045663){\makebox(0,0)[lt]{\lineheight{1.25}\smash{\begin{tabular}[t]{l}$\zg_3$\end{tabular}}}}%
    \put(0.4948873,0.51161495){\makebox(0,0)[lt]{\lineheight{1.25}\smash{\begin{tabular}[t]{l}$\tau_1$\end{tabular}}}}%
    \put(0,0){\includegraphics[width=\unitlength,page=2]{figsectionsix3.pdf}}%
    \put(0.52996127,0.2427184){\makebox(0,0)[lt]{\lineheight{1.25}\smash{\begin{tabular}[t]{l}$\tau_2$\end{tabular}}}}%
    \put(0.28444383,0.34794047){\makebox(0,0)[lt]{\lineheight{1.25}\smash{\begin{tabular}[t]{l}$\tau_3$\end{tabular}}}}%
    \put(0.69557961,0.01373878){\makebox(0,0)[lt]{\lineheight{1.25}\smash{\begin{tabular}[t]{l}$x_2$\end{tabular}}}}%
    \put(-0.00589892,0.32940501){\makebox(0,0)[lt]{\lineheight{1.25}\smash{\begin{tabular}[t]{l}$x_3$\end{tabular}}}}%
    \put(0,0){\includegraphics[width=\unitlength,page=3]{figsectionsix3.pdf}}%
  \end{picture}%
\endgroup%

%% file: figex75.pdf_tex
\begingroup%
  \makeatletter%
  \providecommand\color[2][]{%
    \errmessage{(Inkscape) Color is used for the text in Inkscape, but the package 'color.sty' is not loaded}%
    \renewcommand\color[2][]{}%
  }%
  \providecommand\transparent[1]{%
    \errmessage{(Inkscape) Transparency is used (non-zero) for the text in Inkscape, but the package 'transparent.sty' is not loaded}%
    \renewcommand\transparent[1]{}%
  }%
  \providecommand\rotatebox[2]{#2}%
  \newcommand*\fsize{\dimexpr\f@size pt\relax}%
  \newcommand*\lineheight[1]{\fontsize{\fsize}{#1\fsize}\selectfont}%
  \ifx\svgwidth\undefined%
    \setlength{\unitlength}{407.48699771bp}%
    \ifx\svgscale\undefined%
      \relax%
    \else%
      \setlength{\unitlength}{\unitlength * \real{\svgscale}}%
    \fi%
  \else%
    \setlength{\unitlength}{\svgwidth}%
  \fi%
  \global\let\svgwidth\undefined%
  \global\let\svgscale\undefined%
  \makeatother%
  \begin{picture}(1,0.89989688)%
    \lineheight{1}%
    \setlength\tabcolsep{0pt}%
    \put(0,0){\includegraphics[width=\unitlength,page=1]{figex75.pdf}}%
    \put(0.18212573,0.47891432){\makebox(0,0)[lt]{\lineheight{1.25}\smash{\begin{tabular}[t]{l}3\end{tabular}}}}%
    \put(-0.00192922,0.59486897){\makebox(0,0)[lt]{\lineheight{1.25}\smash{\begin{tabular}[t]{l}2\end{tabular}}}}%
    \put(0.18212573,0.8819947){\makebox(0,0)[lt]{\lineheight{1.25}\smash{\begin{tabular}[t]{l}1\end{tabular}}}}%
    \put(0.19316903,0.78076447){\makebox(0,0)[lt]{\lineheight{1.25}\smash{\begin{tabular}[t]{l}$\tau_4$\end{tabular}}}}%
    \put(0.27599374,0.67953425){\makebox(0,0)[lt]{\lineheight{1.25}\smash{\begin{tabular}[t]{l}$\tau_1$\end{tabular}}}}%
    \put(0.1066632,0.67953425){\makebox(0,0)[lt]{\lineheight{1.25}\smash{\begin{tabular}[t]{l}$\tau_2$\end{tabular}}}}%
    \put(0.04960617,0.7071425){\makebox(0,0)[lt]{\lineheight{1.25}\smash{\begin{tabular}[t]{l}$\tau_3$\end{tabular}}}}%
    \put(0,0){\includegraphics[width=\unitlength,page=2]{figex75.pdf}}%
    \put(0.51894629,0.6040717){\makebox(0,0)[lt]{\lineheight{1.25}\smash{\begin{tabular}[t]{l}$\zg_1$\end{tabular}}}}%
    \put(0.65330641,0.5267686){\makebox(0,0)[lt]{\lineheight{1.25}\smash{\begin{tabular}[t]{l}$\zg_2$\end{tabular}}}}%
    \put(0.63858194,0.75683733){\makebox(0,0)[lt]{\lineheight{1.25}\smash{\begin{tabular}[t]{l}$\zg_3$\end{tabular}}}}%
    \put(0.67907412,0.84150261){\makebox(0,0)[lt]{\lineheight{1.25}\smash{\begin{tabular}[t]{l}$\zg_5$\end{tabular}}}}%
    \put(0.75453654,0.64824494){\makebox(0,0)[lt]{\lineheight{1.25}\smash{\begin{tabular}[t]{l}$\zg_4$\end{tabular}}}}%
    \put(0.60361139,0.57646346){\makebox(0,0)[lt]{\lineheight{1.25}\smash{\begin{tabular}[t]{l}$\zg_6$\end{tabular}}}}%
    \put(0.56311974,0.68137473){\makebox(0,0)[lt]{\lineheight{1.25}\smash{\begin{tabular}[t]{l}$\zg_7$\end{tabular}}}}%
    \put(0.73245009,0.72738848){\makebox(0,0)[lt]{\lineheight{1.25}\smash{\begin{tabular}[t]{l}$\zg_8$\end{tabular}}}}%
    \put(0,0){\includegraphics[width=\unitlength,page=3]{figex75.pdf}}%
    \put(0.18212573,0.00037146){\makebox(0,0)[lt]{\lineheight{1.25}\smash{\begin{tabular}[t]{l}3\end{tabular}}}}%
    \put(-0.00192922,0.11632628){\makebox(0,0)[lt]{\lineheight{1.25}\smash{\begin{tabular}[t]{l}2\end{tabular}}}}%
    \put(0.18212573,0.40345185){\makebox(0,0)[lt]{\lineheight{1.25}\smash{\begin{tabular}[t]{l}1\end{tabular}}}}%
    \put(0.19316903,0.30958403){\makebox(0,0)[lt]{\lineheight{1.25}\smash{\begin{tabular}[t]{l}\small$\tau_4$\end{tabular}}}}%
    \put(0.27599374,0.20099153){\makebox(0,0)[lt]{\lineheight{1.25}\smash{\begin{tabular}[t]{l}\small$\tau_1$\end{tabular}}}}%
    \put(0.1066632,0.20099153){\makebox(0,0)[lt]{\lineheight{1.25}\smash{\begin{tabular}[t]{l}\small$\tau_2$\end{tabular}}}}%
    \put(0.04960617,0.22859977){\makebox(0,0)[lt]{\lineheight{1.25}\smash{\begin{tabular}[t]{l}\small$\tau_3$\end{tabular}}}}%
    \put(0.15451748,0.1954698){\makebox(0,0)[lt]{\lineheight{1.25}\smash{\begin{tabular}[t]{l}\scriptsize $y_{41}$\end{tabular}}}}%
    \put(0,0){\includegraphics[width=\unitlength,page=4]{figex75.pdf}}%
    \put(0.07721441,0.1218479){\makebox(0,0)[lt]{\lineheight{1.25}\smash{\begin{tabular}[t]{l}\scriptsize$\tau(\zd)$\end{tabular}}}}%
    \put(0,0){\includegraphics[width=\unitlength,page=5]{figex75.pdf}}%
    \put(0.0514467,0.03534212){\makebox(0,0)[lt]{\lineheight{1.25}\smash{\begin{tabular}[t]{l}\small$x(\zd)$\end{tabular}}}}%
    \put(0,0){\includegraphics[width=\unitlength,page=6]{figex75.pdf}}%
    \put(0.2023719,0.26909163){\makebox(0,0)[lt]{\lineheight{1.25}\smash{\begin{tabular}[t]{l}\scriptsize$\tau(\zb)$\end{tabular}}}}%
    \put(0.13611205,0.26909163){\makebox(0,0)[lt]{\lineheight{1.25}\smash{\begin{tabular}[t]{l}\scriptsize$\tau(\zg)$\end{tabular}}}}%
    \put(0.00911408,0.07951525){\makebox(0,0)[lt]{\lineheight{1.25}\smash{\begin{tabular}[t]{l}\small$y_{21}$\end{tabular}}}}%
    \put(0.10114163,0.0132553){\makebox(0,0)[lt]{\lineheight{1.25}\smash{\begin{tabular}[t]{l}\small$y_{22}$\end{tabular}}}}%
    \put(0.18764741,0.1218479){\makebox(0,0)[lt]{\lineheight{1.25}\smash{\begin{tabular}[t]{l}\scriptsize$\tau(\za)$\end{tabular}}}}%
    \put(0.15451748,0.22123746){\makebox(0,0)[lt]{\lineheight{1.25}\smash{\begin{tabular}[t]{l}\scriptsize $y_{44}$\end{tabular}}}}%
    \put(0.19132851,0.22123751){\makebox(0,0)[lt]{\lineheight{1.25}\smash{\begin{tabular}[t]{l}\scriptsize $y_{43}$\end{tabular}}}}%
    \put(0.63674121,0.26725122){\makebox(0,0)[lt]{\lineheight{1.25}\smash{\begin{tabular}[t]{l}\scriptsize$G(\zg_3)$\end{tabular}}}}%
    \put(0.70668205,0.27461358){\makebox(0,0)[lt]{\lineheight{1.25}\smash{\begin{tabular}[t]{l}\scriptsize$G(\zg_8)$\end{tabular}}}}%
    \put(0.54471375,0.25988906){\makebox(0,0)[lt]{\lineheight{1.25}\smash{\begin{tabular}[t]{l}\scriptsize$G(\zg_7)$\end{tabular}}}}%
    \put(0.73244987,0.18994541){\makebox(0,0)[lt]{\lineheight{1.25}\smash{\begin{tabular}[t]{l}\scriptsize$G(\zg_4)$\end{tabular}}}}%
    \put(0,0){\includegraphics[width=\unitlength,page=7]{figex75.pdf}}%
    \put(0.69563858,0.05742895){\makebox(0,0)[lt]{\lineheight{1.25}\smash{\begin{tabular}[t]{l}\scriptsize$G(\zg_2)$\end{tabular}}}}%
    \put(0,0){\includegraphics[width=\unitlength,page=8]{figex75.pdf}}%
    \put(0.58888693,0.08319736){\makebox(0,0)[lt]{\lineheight{1.25}\smash{\begin{tabular}[t]{l}\scriptsize$G(\zg_6)$\end{tabular}}}}%
    \put(0.49317815,0.14945678){\makebox(0,0)[lt]{\lineheight{1.25}\smash{\begin{tabular}[t]{l}\scriptsize$G(\zg_1)$\end{tabular}}}}%
    \put(0,0){\includegraphics[width=\unitlength,page=9]{figex75.pdf}}%
    \put(0.172923,0.6666504){\makebox(0,0)[lt]{\lineheight{1.25}\smash{\begin{tabular}[t]{l}$y_4$\end{tabular}}}}%
    \put(0,0){\includegraphics[width=\unitlength,page=10]{figex75.pdf}}%
    \put(0.38204509,0.68289692){\makebox(0,0)[lt]{\lineheight{1.25}\smash{\begin{tabular}[t]{l}$y_3$\end{tabular}}}}%
    \put(0,0){\includegraphics[width=\unitlength,page=11]{figex75.pdf}}%
    \put(0.00359244,0.78812668){\makebox(0,0)[lt]{\lineheight{1.25}\smash{\begin{tabular}[t]{l}$y_1$\end{tabular}}}}%
    \put(0,0){\includegraphics[width=\unitlength,page=12]{figex75.pdf}}%
    \put(0.52968713,0.52709665){\makebox(0,0)[lt]{\lineheight{1.25}\smash{\begin{tabular}[t]{l}$y_2$\end{tabular}}}}%
    \put(0,0){\includegraphics[width=\unitlength,page=13]{figex75.pdf}}%
    \put(0.6514659,0.6666504){\makebox(0,0)[lt]{\lineheight{1.25}\smash{\begin{tabular}[t]{l}$y_4$\end{tabular}}}}%
    \put(0,0){\includegraphics[width=\unitlength,page=14]{figex75.pdf}}%
    \put(0.86058799,0.68289692){\makebox(0,0)[lt]{\lineheight{1.25}\smash{\begin{tabular}[t]{l}$y_3$\end{tabular}}}}%
    \put(0,0){\includegraphics[width=\unitlength,page=15]{figex75.pdf}}%
    \put(0.48213534,0.78812668){\makebox(0,0)[lt]{\lineheight{1.25}\smash{\begin{tabular}[t]{l}$y_1$\end{tabular}}}}%
    \put(0,0){\includegraphics[width=\unitlength,page=16]{figex75.pdf}}%
    \put(0.18028517,0.7034614){\makebox(0,0)[lt]{\lineheight{1.25}\smash{\begin{tabular}[t]{l}4\end{tabular}}}}%
    \put(0,0){\includegraphics[width=\unitlength,page=17]{figex75.pdf}}%
    \put(0.38204509,0.20435402){\makebox(0,0)[lt]{\lineheight{1.25}\smash{\begin{tabular}[t]{l}$y_3$\end{tabular}}}}%
    \put(0,0){\includegraphics[width=\unitlength,page=18]{figex75.pdf}}%
    \put(0.00359244,0.30958378){\makebox(0,0)[lt]{\lineheight{1.25}\smash{\begin{tabular}[t]{l}$y_1$\end{tabular}}}}%
    \put(0,0){\includegraphics[width=\unitlength,page=19]{figex75.pdf}}%
    \put(0.4876569,0.07951525){\makebox(0,0)[lt]{\lineheight{1.25}\smash{\begin{tabular}[t]{l}\small$y_{21}$\end{tabular}}}}%
    \put(0.57968441,0.0132553){\makebox(0,0)[lt]{\lineheight{1.25}\smash{\begin{tabular}[t]{l}\small$y_{22}$\end{tabular}}}}%
    \put(0,0){\includegraphics[width=\unitlength,page=20]{figex75.pdf}}%
    \put(0.86058799,0.20435402){\makebox(0,0)[lt]{\lineheight{1.25}\smash{\begin{tabular}[t]{l}$y_3$\end{tabular}}}}%
    \put(0,0){\includegraphics[width=\unitlength,page=21]{figex75.pdf}}%
    \put(0.48213534,0.30958378){\makebox(0,0)[lt]{\lineheight{1.25}\smash{\begin{tabular}[t]{l}$y_1$\end{tabular}}}}%
    \put(0,0){\includegraphics[width=\unitlength,page=22]{figex75.pdf}}%
    \put(0.64042231,0.37032177){\makebox(0,0)[lt]{\lineheight{1.25}\smash{\begin{tabular}[t]{l}\scriptsize$G(\zg_5)$\end{tabular}}}}%
    \put(0.05512781,0.52492808){\makebox(0,0)[lt]{\lineheight{1.25}\smash{\begin{tabular}[t]{l}$y_2$\end{tabular}}}}%
    \put(0.19132851,0.1954698){\makebox(0,0)[lt]{\lineheight{1.25}\smash{\begin{tabular}[t]{l}\scriptsize$y_{42}$\end{tabular}}}}%
    \put(0.63306032,0.22123746){\makebox(0,0)[lt]{\lineheight{1.25}\smash{\begin{tabular}[t]{l}\scriptsize $y_{44}$\end{tabular}}}}%
    \put(0.66987129,0.22123751){\makebox(0,0)[lt]{\lineheight{1.25}\smash{\begin{tabular}[t]{l}\scriptsize $y_{43}$\end{tabular}}}}%
    \put(0.66987129,0.1954698){\makebox(0,0)[lt]{\lineheight{1.25}\smash{\begin{tabular}[t]{l}\scriptsize$y_{42}$\end{tabular}}}}%
    \put(0.63306079,0.1954698){\makebox(0,0)[lt]{\lineheight{1.25}\smash{\begin{tabular}[t]{l}\scriptsize$y_{41}$\end{tabular}}}}%
    \put(0,0){\includegraphics[width=\unitlength,page=23]{figex75.pdf}}%
  \end{picture}%
\endgroup%

%% file: figsectfour1.pdf_tex
\begingroup%
  \makeatletter%
  \providecommand\color[2][]{%
    \errmessage{(Inkscape) Color is used for the text in Inkscape, but the package 'color.sty' is not loaded}%
    \renewcommand\color[2][]{}%
  }%
  \providecommand\transparent[1]{%
    \errmessage{(Inkscape) Transparency is used (non-zero) for the text in Inkscape, but the package 'transparent.sty' is not loaded}%
    \renewcommand\transparent[1]{}%
  }%
  \providecommand\rotatebox[2]{#2}%
  \newcommand*\fsize{\dimexpr\f@size pt\relax}%
  \newcommand*\lineheight[1]{\fontsize{\fsize}{#1\fsize}\selectfont}%
  \ifx\svgwidth\undefined%
    \setlength{\unitlength}{386.54232259bp}%
    \ifx\svgscale\undefined%
      \relax%
    \else%
      \setlength{\unitlength}{\unitlength * \real{\svgscale}}%
    \fi%
  \else%
    \setlength{\unitlength}{\svgwidth}%
  \fi%
  \global\let\svgwidth\undefined%
  \global\let\svgscale\undefined%
  \makeatother%
  \begin{picture}(1,0.4475745)%
    \lineheight{1}%
    \setlength\tabcolsep{0pt}%
    \put(0,0){\includegraphics[width=\unitlength,page=1]{figsectfour1.pdf}}%
    \put(0.33333243,0.25222229){\makebox(0,0)[lt]{\lineheight{1.25}\smash{\begin{tabular}[t]{l}$\tau_1$\end{tabular}}}}%
    \put(0,0){\includegraphics[width=\unitlength,page=2]{figsectfour1.pdf}}%
    \put(0.1237823,0.10088059){\makebox(0,0)[lt]{\lineheight{1.25}\smash{\begin{tabular}[t]{l}$\tau_2$\end{tabular}}}}%
    \put(0,0){\includegraphics[width=\unitlength,page=3]{figsectfour1.pdf}}%
    \put(0.22079624,0.3647585){\makebox(0,0)[lt]{\lineheight{1.25}\smash{\begin{tabular}[t]{l}$\tau_4$\end{tabular}}}}%
    \put(0,0){\includegraphics[width=\unitlength,page=4]{figsectfour1.pdf}}%
    \put(0.09273786,0.34147513){\makebox(0,0)[lt]{\lineheight{1.25}\smash{\begin{tabular}[t]{l}$\tau_3$\end{tabular}}}}%
    \put(0,0){\includegraphics[width=\unitlength,page=5]{figsectfour1.pdf}}%
    \put(0.87661083,0.25222229){\makebox(0,0)[lt]{\lineheight{1.25}\smash{\begin{tabular}[t]{l}$\tau_1$\end{tabular}}}}%
    \put(0,0){\includegraphics[width=\unitlength,page=6]{figsectfour1.pdf}}%
    \put(0.66706037,0.10088059){\makebox(0,0)[lt]{\lineheight{1.25}\smash{\begin{tabular}[t]{l}$\tau_2$\end{tabular}}}}%
    \put(0,0){\includegraphics[width=\unitlength,page=7]{figsectfour1.pdf}}%
    \put(0.76407447,0.3647585){\makebox(0,0)[lt]{\lineheight{1.25}\smash{\begin{tabular}[t]{l}$\tau_4$\end{tabular}}}}%
    \put(0,0){\includegraphics[width=\unitlength,page=8]{figsectfour1.pdf}}%
    \put(0.63601587,0.34147513){\makebox(0,0)[lt]{\lineheight{1.25}\smash{\begin{tabular}[t]{l}$\tau_3$\end{tabular}}}}%
    \put(0,0){\includegraphics[width=\unitlength,page=9]{figsectfour1.pdf}}%
  \end{picture}%
\endgroup%

%% file: figorpheus2.pdf_tex
\begingroup%
  \makeatletter%
  \providecommand\color[2][]{%
    \errmessage{(Inkscape) Color is used for the text in Inkscape, but the package 'color.sty' is not loaded}%
    \renewcommand\color[2][]{}%
  }%
  \providecommand\transparent[1]{%
    \errmessage{(Inkscape) Transparency is used (non-zero) for the text in Inkscape, but the package 'transparent.sty' is not loaded}%
    \renewcommand\transparent[1]{}%
  }%
  \providecommand\rotatebox[2]{#2}%
  \newcommand*\fsize{\dimexpr\f@size pt\relax}%
  \newcommand*\lineheight[1]{\fontsize{\fsize}{#1\fsize}\selectfont}%
  \ifx\svgwidth\undefined%
    \setlength{\unitlength}{487.98539913bp}%
    \ifx\svgscale\undefined%
      \relax%
    \else%
      \setlength{\unitlength}{\unitlength * \real{\svgscale}}%
    \fi%
  \else%
    \setlength{\unitlength}{\svgwidth}%
  \fi%
  \global\let\svgwidth\undefined%
  \global\let\svgscale\undefined%
  \makeatother%
  \begin{picture}(1,1.00752663)%
    \lineheight{1}%
    \setlength\tabcolsep{0pt}%
    \put(0,0){\includegraphics[width=\unitlength,page=1]{figorpheus2.pdf}}%
    \put(0.05917409,0.79562246){\makebox(0,0)[lt]{\lineheight{1.25}\smash{\begin{tabular}[t]{l}2\end{tabular}}}}%
    \put(0,0){\includegraphics[width=\unitlength,page=2]{figorpheus2.pdf}}%
    \put(0.05800338,0.84379075){\makebox(0,0)[lt]{\lineheight{1.25}\smash{\begin{tabular}[t]{l}1\end{tabular}}}}%
    \put(0,0){\includegraphics[width=\unitlength,page=3]{figorpheus2.pdf}}%
    \put(0.05911404,0.74217494){\makebox(0,0)[lt]{\lineheight{1.25}\smash{\begin{tabular}[t]{l}3\end{tabular}}}}%
    \put(0,0){\includegraphics[width=\unitlength,page=4]{figorpheus2.pdf}}%
    \put(0.06001459,0.69704648){\makebox(0,0)[lt]{\lineheight{1.25}\smash{\begin{tabular}[t]{l}4\end{tabular}}}}%
    \put(0,0){\includegraphics[width=\unitlength,page=5]{figorpheus2.pdf}}%
    \put(0.27224031,0.882059){\makebox(0,0)[lt]{\lineheight{1.25}\smash{\begin{tabular}[t]{l}\color{red}$P(2)$\end{tabular}}}}%
    \put(0,0){\includegraphics[width=\unitlength,page=6]{figorpheus2.pdf}}%
    \put(0.05917409,0.24232704){\makebox(0,0)[lt]{\lineheight{1.25}\smash{\begin{tabular}[t]{l}2\end{tabular}}}}%
    \put(0,0){\includegraphics[width=\unitlength,page=7]{figorpheus2.pdf}}%
    \put(0.05800338,0.29049542){\makebox(0,0)[lt]{\lineheight{1.25}\smash{\begin{tabular}[t]{l}1\end{tabular}}}}%
    \put(0,0){\includegraphics[width=\unitlength,page=8]{figorpheus2.pdf}}%
    \put(0.05911404,0.18887948){\makebox(0,0)[lt]{\lineheight{1.25}\smash{\begin{tabular}[t]{l}3\end{tabular}}}}%
    \put(0,0){\includegraphics[width=\unitlength,page=9]{figorpheus2.pdf}}%
    \put(0.06001459,0.14375097){\makebox(0,0)[lt]{\lineheight{1.25}\smash{\begin{tabular}[t]{l}4\end{tabular}}}}%
    \put(0,0){\includegraphics[width=\unitlength,page=10]{figorpheus2.pdf}}%
  \end{picture}%
\endgroup%